\documentclass[10pt]{amsart}
\usepackage[margin=2.5cm]{geometry}

\usepackage{tikz-cd}
\usepackage{graphicx}				% Use pdf, png, jpg, or eps§ with pdflatex; use eps in DVI mode	
\usepackage{amsmath}
\usepackage{subfig, graphicx}
\usepackage{amssymb, amsthm, indentfirst}
\usepackage{relsize}
\usepackage{hyperref}
\usepackage[toc,page]{appendix}

\numberwithin{equation}{section}
\usepackage{color}
\theoremstyle{plain}
\newtheorem{thm}{Theorem}[section]

\newtheorem{lemma}[thm]{Lemma}
\newtheorem{prop}[thm]{Proposition}
\newtheorem{corollary}[thm]{Corollary}

\theoremstyle{definition}

\newtheorem{rmk}{Remark}
\newtheorem{example}{Example}

\newtheorem*{hypothesis*}{Hypothesis}

\title{Pullback formulae for nearly holomorphic Saito-Kurokawa lifts}
\subjclass[2010]{11F27, 11F67}

\author{SHIH-YU CHEN}
\address{Institute of Mathematics~\\Academia Sinica~\\ 6F, Astronomy-Mathematics Building, No.\,1, Sec.\,4, Roosevelt Road, Taipei 10617, R.O.C. (Taiwan)~}

\email{sychen0626@gate.sinica.edu.tw}
\date{\today}
	
\def\SL{{\rm{SL}}}
\def\GL{{\rm{GL}}}
\def\GSp{{\rm GSp}}

\def\Sp{{\rm Sp}}

\def\A{{\mathbb A}}
\def\C{{\mathbb C}}
\def\R{{\mathbb R}}
\def\Q{{\mathbb Q}}
\def\Z{{\mathbb Z}}
\def\N{{\mathbb N}}

\def\I{\mathbb{I}}
\def\K{\mathcal{K}}
\def\<{\langle}
\def\>{\rangle}
\def\G{{\bf G}}

\def\bp{\begin{pmatrix}}
\def\ep{\end{pmatrix}}
\def\wt{\widetilde}
\def\<{\langle}
\def\>{\rangle}

\begin{document}
\begin{abstract}
We give explicit pullback formulae for nearly holomorphic Saito-Kurokawa lifts restrict to product of upper half-plane against with product of elliptic modular forms. We generalize the formula of Ichino to modular forms of higher level and free the restriction on weights. The explicit formulae provide non-trivial examples for the refined Gan-Gross-Prasad conjecture for $({\rm SO}_5,{\rm SO}_4)$ in the non-tempered cases. As an application, we obtain Deligne's conjecture for critical values of certain automorphic $L$-functions for $\GL_3 \times \GL_2$. We also expect to apply our pullback formulae to construct two-variables $p$-adic $L$-functions for $\GL_3 \times \GL_2$ in the future.
\end{abstract}

\maketitle
%\section{}
%\subsection{}
\tableofcontents

\section{Introduction}
Period integrals of automorphic forms are often related to critical values of $L$-functions. It has important applications to the analytic and algebraic theory of $L$-functions. The main theme of this article concerns with a special case, namely the pullback of Saito-Kurokawa lifts. This is a special case for the Gan-Gross-Prasad conjecture (cf.\,\cite{GP1992} and \cite{GGP2012}) for $({\rm SO}_5,{\rm SO}_4)$ in the non-tempered case. The restriction problem and the refined Gan-Gross-Prasad conjecture under general setting was studied by Gan-Gurevich and Qiu in \cite{GG2009} and \cite{Qiu2014}. The purpose here is to generalize Ichino's explicit pullback formulae in \cite{Ichino2005} to pullback formulae for nearly holomorphic Saito-Kurokawa lifts. More precisely, let $f\in S_{2\kappa'}(\Gamma_0(N))$ and $g\in S_{\kappa+1}(\Gamma_0(N))$ be normalized elliptic newforms. We assume the following assumptions on the level and weights:
\begin{hypothesis*}[H]
\noindent
\begin{itemize}
\item $\kappa-\kappa'=2m \in 2\Z_{\geq 0}$.
\item $N$ is an odd square-free integer.
\end{itemize}
\end{hypothesis*}
Let ${h}\in S_{\kappa'+1/2}^+(\Gamma_0(4N))$ be a newform in the sense of \cite{Kohnen1982} associated to ${f}$ by the Shintani lift. Let $F \in S_{\kappa'+1}(\Gamma_0^{(2)}(N))$ be the Saito-Kurokawa lift of $h$. Let $\Delta_r$ be the weight $r$ differential operator studied by Maass in \cite{MaassLNM}. We write $\Delta_{\kappa'+1}^m = \Delta_{\kappa-1}\circ \cdots \circ \Delta_{\kappa'+1}$. Then $\Delta_{\kappa'+1}^mF$ is a nearly holomorphic Siegel modular form of level $\Gamma_0^{(2)}(N)$ and weight $\kappa+1$ on $\frak{H}_2$. We consider the period integral defined by
\begin{align*}
\<\Delta_{\kappa'+1}^mF\vert_{\frak{H}\times\frak{H}},g\times g\> = \frac{1}{\left [\SL_2(\Z) : \Gamma_0(N) \right ]^2}\int_{\Gamma_0(N) \backslash \frak{H}}\int_{\Gamma_0(N) \backslash \frak{H}}\Delta_{\kappa'+1}^mF\left ( \begin{pmatrix} \tau_1 & 0 \\ 0 & \tau_2   \end{pmatrix}\right )\overline{{ g}(\tau_1){ g}(\tau_2)}y_1^{\kappa+1}y_2^{\kappa+1}\frac{d\tau_1d\tau_2}{y_1^2y_2^2}.
\end{align*}
Let $\Lambda(s,{\rm Sym}^2(g)\otimes f)$ be the completed $L$-function of ${\rm Sym}^2(g)\otimes f$. The following is our main theorem.

\begin{thm}\label{T:1}(Theorem \ref {T:pullback formula})  Assume Hypothesis (H) holds. We have
\begin{align*}
\frac{\vert \langle \Delta_{\kappa'+1}^m{F}|_{\frak{H}\times \frak{H}}, { g}\times{ g} \rangle \vert^2} {\<g,g\>^2}&=2^{-\kappa-6m-1}\prod_{p\mid N}p(1+p)^{-2}C(\kappa,\kappa')
\frac{\langle { h},{ h} \rangle}{\langle { f},{ f} \rangle} \Lambda\left(\kappa+\kappa',{\rm Sym}^2(g)\otimes f\right) .
\end{align*}
Here $C(\kappa,\kappa')$ is a non-zero rational number defined in (\ref{E:archimedean constant}).
\end{thm}
When $\kappa=\kappa'$, Theorem \ref{T:1} is first proved by Ichino in \cite[Theorem 2.1]{Ichino2005} if $N=1$ and a pullback formula is obtained in \cite{PV-P2018} if $N$ is odd and square-free. Our formulae strengthen these works by allowing $\kappa> \kappa'$. Our motivation to extend Ichino's formula to $\kappa\geq \kappa'$ stems from the potential application to a construction of a two-variables $p$-adic $L$-function interpolating the central value $\Lambda(\kappa+\kappa',{\rm Sym}^2({\mathbf g_{\kappa'}})\otimes {\mathbf f_{\kappa}})$ for two Hida families $\mathbf f=\{\mathbf f_{\kappa}\}$ and $\mathbf g=\{\mathbf g_{\kappa'}\}$. In the course of the proof, we introduce certain Mass differential operators $\Delta_r$ on ${\rm GSp}(4)$, which have been crucial in the theory of $p$-adic $L$-functions. For example, see the constuctions of Hida's Rankin-Selberg $p$-adic $L$-functions in \cite{Hida85Inv} and \cite{Hida88Fourier}, Bertolini-Darmon-Prasanna anticyclotomic $p$-adic $L$-functions in \cite{BDP2013}, and Eischen-Harris-Li-Skinner $p$-adic $L$-functions for unitary groups in \cite{EHLS2016}. We hope to come back to this problem in the near future.

Now we provide two examples to verify numerically Theorem \ref{T:1}.
\begin{example}
Let $f \in S_{18}(\Gamma_0(1))$ and $g \in S_{12}(\Gamma_0(1))$ be the normalized newforms. Then $\kappa=11$, $\kappa'=9$, and $N=1$. Let $h \in S_{19/2}^+(\Gamma_0(4))$ be the newform associated to $f$ normalized so that
$$h(\tau) = q^3-2q^4-16q^7+36q^8+99q^{11}+\cdots.$$
By \cite[Proposition 6.6]{IP2018} and the Kohnen-Zagier formula \cite[Corollary 1]{Kohnen1985}, we have
\begin{align*}
\<g,g\> &= 2^{-15}\cdot5\cdot \Lambda(12,g,{\rm Ad}),\\
\<f,f\>\<h,h\>^{-1} &=2^8\cdot3^{17/2}\cdot \Lambda(9,f\otimes\chi_{-3}). 
\end{align*}
Here $\Lambda(s,g,{\rm Ad})$ and $\Lambda(s,f\otimes \chi_{-3})$ are the completed adjoint $L$-function of $g$ and the completed twisted $L$-function of $f$ by $\chi_{-3}$, respectively.
By computer calculation based on Dokchitser's program \cite{Dokchitser2004}, we deduce that
\begin{align*}
C(11,9) & = 1,\\
\<g,g\> &=0.0000010353620568043209223478168122251645\cdots,\\
\<f,f\>\<h,h\>^{-1} & = 75633.942121560198996880460854760845132468\cdots,\\
\Lambda(20,{\rm Sym}^2(g)\otimes f) & = 0.0053135057875930754652200977341472154100\cdots.
\end{align*}
On the other hand, by the formula for $\Delta_{10}$ in (\ref{E:Maass differential operator}), we have
$$\frac{ \langle \Delta_{10}{F}|_{\frak{H}\times \frak{H}}, { g}\times{ g} \rangle } {\<g,g\>^2} = \sum_{b \in \Z,\,b^2<4}(1-b^2/4)\cdot A\left( \bp 1 & b/2 \\ b/2 & 1 \ep \right) = \frac{3}{2}c_h(3)+c_h(4)=-2^{-1}.$$
Therefore the pullback formula in Theorem \ref{T:1} holds numerically.
\end{example}

\begin{example}
Let $f=g \in S_{2}(\Gamma_0(15))$ be the normalized newform. Then $\kappa=\kappa'=1$, and $N=15$. Let $h\in S_{3/2}^+(\Gamma_0(60))$ be the newform associated to $f$ normalized so that 
$$h(\tau) = q^3-2q^8-q^{15}+2q^{20}+2q^{23}+\cdots.$$
Similarly as in the previous example, by computer calculation, we have
\begin{align*}
C(1,1)& = 1,\\
\<g,g\>&=0.0023596244145167680294160631624014882733\cdots,\\
\<f,f\>\<h,h\>^{-1}&=1.0161993600970582320694739236097011625363\cdots,\\
\Lambda(2,{\rm Sym}^2(g)\otimes f) & = 
0.0034762890966413331251690052554140352448\cdots.
\end{align*}
On the other hand, we have
$$\frac{ \langle {F}|_{\frak{H}\times \frak{H}}, { g}\times{ g} \rangle } {\<g,g\>^2} = \sum_{b \in \Z,\,b^2<4}A\left( \bp 1 & b/2 \\ b/2 & 1 \ep \right) = 2c_h(3)+c_h(4)=2.$$
Therefore the pullback formula in Theorem \ref{T:1} holds numerically.
\end{example} 

An immediate consequence of our pullback formulae is the Deligne conjecture for the central value $\Lambda(\kappa+\kappa',{\rm Sym}^2({g})\otimes {f})$.
\begin{corollary} (Corollary \ref{C:algebraicity})\label{C}
Assume Hypothesis (H) holds. For $\sigma \in {\rm Aut}(\C)$, we have
$$\left (\frac{\Lambda(\kappa+\kappa',{\rm Sym}^2({ g})\otimes { f})}{\langle { g},{ g}\rangle^2\Omega_f^+} \right )^{\sigma}=\frac{\Lambda(\kappa+\kappa',{\rm Sym}^2({ g}^{\sigma})\otimes {f}^{\sigma})}{\langle { g}^{\sigma},{ g}^{\sigma}\rangle^2\Omega_{f^{\sigma}}^+}.$$
Here $\Omega_f^+$ is the plus period of ${ f}$ defined in \cite{Shimura1977}.
\end{corollary}
\begin{rmk}\label{rmk2}\noindent
\begin{itemize}
\item[(1)]If $\Lambda(\kappa',f) \neq 0$, then Corollary \ref{C} follows from the algebraicity of the central value of the triple product $L$-function $\Lambda(s,g\otimes g \otimes f)$ proved by Harris and Kudla \cite{HK1991}.
\item[(2)]In \cite[Theorem A]{CC2017}, following different approach, we prove the analogue results for $\kappa' >\kappa$, or $\kappa \geq \kappa' $ and $N>1$, regardless the parities of $\kappa$ and $\kappa'$.
\item[(3)]In \cite[Theorem 1.1]{Xue2017}, by considering the $\SL_2$-period obtained from the restriction of nearly holomorphic Jacobi form associated to $h$ on $\frak{H}$ against $g$, Xue gave a proof of the same result in the case $N=1$. From the representation-theoretic point of view, the ${\rm SO}(2,2)$-periods considered in this article and the $\SL_2$-periods considered in \cite[Proposition 3.1]{Xue2017} are essentially the same by a seesaw identity. For instance, the corresponding $\SL_2$-period considered here is the right-hand side of the seesaw identity (\ref{P:SO(2,2)-period to SL_2 period}). However, it seems that this $\SL_2$-period is different from the $\SL_2$-period considered in \cite[Proposition 2.1]{Xue2017}.
\end{itemize}
\end{rmk}

Combining with the result of Januszewski in \cite{Januszewski2017} on period relations, we obtain a conditional result on Deligne conjecture for ${\rm Sym}^2({ g})\otimes {f}$ with abelian twists.

\begin{corollary}(Corollary \ref{C:algebraicity 2}) \label{C2}
Assume Hypothesis (H) holds and $\Lambda(\kappa+\kappa',{\rm Sym}^2({ g})\otimes { f}) \neq 0$. Let $n \in \Z$ be a critical integer for $\Lambda(s,{\rm Sym}^2({ g})\otimes {f})$ and $\chi$ be a Dirichlet character such that $(-1)^{n}\chi(-1)=1$. For $\sigma \in {\rm Aut}(\C)$, we have
$$\left (\frac{\Lambda(n,{\rm Sym}^2({ g})\otimes { f}\otimes \chi)}{G({\chi})^3(2\pi \sqrt{-1})^{3(n-\kappa-\kappa')}\langle { g},{ g}\rangle^2\Omega_f^+} \right )^{\sigma}=\frac{\Lambda(n,{\rm Sym}^2({ g}^{\sigma})\otimes {f}^{\sigma}\otimes \chi^{\sigma})}{G({\chi}^{\sigma})^3(2\pi \sqrt{-1})^{3(n-\kappa-\kappa')}\langle { g}^{\sigma},{ g}^{\sigma}\rangle^2\Omega_{f^{\sigma}}^+}.$$
Here $G(\chi)$ is the Gauss sum associated to $\chi$, and $\Omega_f^+$ is the plus period of ${ f}$ defined in \cite{Shimura1977}.
\end{corollary} 

We shall follow the idea of Ichino in \cite{Ichino2005} in the proof of Theorem \ref{T:1}. The new ingredients of this article are the calculations of various local theta lifts at places $v \mid \infty N$, the introduction of Maass differential operators, and the calculation of archimedean local trilinear period integrals. We remark that it is possible to apply the refined Gan-Gross-Prasad conjecture for Saito-Kurokawa representations proved in \cite{Qiu2014} to obtain the pullback formulae, However, the corresponding archimedean local period integrals seems too difficult to calculate directly, so we decide to use the original method of Ichino in this paper.% Another reason is that theta lifts are more suitable for considering $p$-adic deformation. For example, the $\Lambda$-adic Shintani lifts \cite{Stevens1994}.

The paper is organized as follows. In \S\,\ref{S:Notation and conventions}, we fix the notation  and conventions that will be used throughout the article and review the basics for metaplectic groups $\wt{\SL}_2$. In \S\,\ref{S:Automorphic forms and $L$-functions}, we define the adelic lifts ${\bf f}$, ${\bf g}$, ${\bf h}$, and ${\bf F}$ of $f$, $g$, $h$, and $F$ respectively to automorphic forms on $\GL_2(\A)$, $\wt{\SL}_2(\A)$, and $\GSp_4(\A)$, and fix normalizations of their corresponding Whittaker functions. We also introduce an auxiliary imaginary quadratic field $\K$ and a base change lift ${\bf g}_\K$ of ${\bf g}$ to $\GL_2(\A_\K)$. The Maass-Shimura differential operator $V_+$ on $\GL_2(\A)$ and $\wt{\SL}_2(\A)$, and the Maass differential operator $D_+$ on $\GSp_4(\A)$ are also introduced in \S\,\ref{S:Automorphic forms and $L$-functions}. In \S\,\ref{S:Weil representations and theta lifts}, we review the basic terminology for Weil representations and theta lifts. In \S\,\ref{S:seesaw I}, we establish an explicit seesaw identity for the following seesaw:
\begin{center}
\begin{tikzcd}[every arrow/.append style={dash}]
{\rm SO}(3,2)\arrow[d] & \SL_2 \times \widetilde{\SL}_2\arrow[d]\\
{\rm SO}(2,2) \times {\rm SO}(1) \arrow[ur]  & \widetilde{\SL}_2\arrow[ul]
\end{tikzcd}
\begin{tikzcd}[every arrow/.append style={dash}]
D_+^m{\bf F} \arrow[d] & \overline{\bf g}^{\sharp} \times \Theta \arrow[d]\\
(\overline{{\bf g} \times {\bf g}}) \times 1 \arrow[ur]  & V_+^{m}{\bf h} \arrow[ul]
\end{tikzcd}
\end{center}
The explicit seesaw identity follows from an explicit Jacquet-Langlands-Shimizu lifting from $\overline{{\bf g}\times{\bf g}}$ to $\overline{\bf g}^{\sharp}$ in \S\,\ref{SS:JLS}, and an explicit Saito-Kurokawa lifting from $V_+^m{\bf h}$ to $D_+^m{\bf F}$ in \S\,\ref{SS:SK}. In \S\,\ref{S:seesaw II}, we establish another explicit seesaw identity for the following seesaw:
\begin{center}
\begin{tikzcd}[every arrow/.append style={dash}]
\widetilde{\SL}_2\arrow[d] \times \widetilde{\SL}_2 &{\rm SO}(3,1)\arrow[d]\\
\SL_2 \arrow[ur]  & {\rm SO}(2,1)\times {\rm SO}(1)\arrow[ul]
\end{tikzcd}
\begin{tikzcd}[every arrow/.append style={dash}]
V_+^{m}{\bf h} \times \Theta  \arrow[d] & \chi_{-D}\times{\bf g}_{\K}^{\sharp} \arrow[d]\\
\overline{\bf g}^{\sharp}\arrow[ur]  & (V_+^{2m}{\bf f}\otimes \chi_{-D} )\times 1 \arrow[ul]
\end{tikzcd}
\end{center}
The explicit seesaw identity follows from an explicit Shintani lifting from $V_+^{2m}{\bf f}\otimes \chi_{-D}$ to $V_+^m{\bf h}$ in \S\,\ref{SS:Shintani}, and an explicit base change lifting from $\overline{\bf g}^{\sharp}$ to ${\bf g}_\K^{\sharp}$ in \S\,\ref{SS:BC}. The first seesaw translates the ${\rm SO}(2,2)$-period $\left\<D_+^m{\bf F}, {\bf g}\times {\bf g} \right\>_{{\rm SO}(2,2)}$ into the $\SL_2$-period $\left\langle {V}_+^m{\bf h}\cdot\Theta , {\bf g}^{\sharp} \right\rangle _{\SL_2}$. The $\SL_2$-period is then translated into the ${\rm SO}(2,1)$-period $\left\< {V}_+^{2m}{\bf f}, \overline{\bf g}_{\K}^{\sharp} \right\>_{{\rm SO}(2,1)}$ by the second seesaw. Thanks to the solution of the refined Gan-Gross-Prasad conjecture for $({\rm SO}_4,{\rm SO}_3)$ due to Ichino, the ${\rm SO}(2,1)$-period is then equal to a product of local trilinear period integrals and the central value of a twisted triple product $L$-function. In \S\,\ref{S:triple}, we state the explicit value of the local trilinear period integrals and postpone the proof of the archimedean case to \S\,\ref{S: Local trilinear period integral in the C times R case}. The central value of the twisted triple product $L$-function considered here is equal to 
$$\Lambda\left(\kappa+\kappa',{\rm Sym}^2(g)\otimes f\right)\Lambda\left(\kappa',f\otimes\chi_{-D}\right).$$ In \S\,\ref{S:main results}, we prove the pullback formula based on the results in \S\,\ref{S:seesaw I}-\ref{S:triple} and the Kohnen-Zagier formula. Then we deduce Deligne conjecture for $\Lambda\left(\kappa+\kappa',{\rm Sym}^2(g)\otimes f\right)$ from the pullback formula.

\subsubsection*{Acknowledgement}

The results of this paper are part of the author's Ph.D. thesis in National Taiwan University. The author would like to thank my advisor Ming-Lun Hsieh for the encouragement and help during the Ph.D. program. This work would been impossible without his guidance and insight on automorphic forms. The author also would like to thank Atsushi Ichino for the suggestions and sharing his program code for numerical examples. Finally the author thanks the referee for the comments on the previous version of this paper.

\section{Notation and conventions}\label{S:Notation and conventions}
\subsection{Notation}\label{SS:Notation}
If $R$ is a ring, we let $R^{\times}$ denote the group of units of $R$ and $R^{\times, 2} = \{ a^2 \mbox{ }\vert \mbox{ } a \in R^{\times} \}.$ If $S$ is a set, let $\I_S$ be the characteristic function of $S$.

Let $F$ be a local field of characteristic zero. When $F$ is non-archimedean, let $\varpi_F$ be a prime element, and ${\rm ord}_F$ be the valuation on $F$ normalized so that ${\rm ord}_F(\varpi_F)=1$.  Let $|\mbox{ }|_F$ be the absolute value on $F$ normalized so that $|\varpi_F|_F^{-1}$ is equal to the cardinality of $\mathcal{O}_F / \varpi_F\mathcal{O}_F$. When $F$ is archimedean, let $|\mbox{ }|=|\mbox{ }|_{\R}$ be the usual absolute value on $\R$ and $|z|_\C=z\overline{z}$ on $\C$. Define the local zeta function $\zeta_F(s)$ by
$$\zeta_F(s)=\begin{cases}(1-|\varpi_F|_F^s)^{-1} & \mbox{ if $F$ is non-archimedean},\\
\pi^{-s/2}\Gamma(s/2) & \mbox{ if $F=\R$},\\
2(2\pi)^{-s}\Gamma(s) & \mbox{ if $F=\C$}.
\end{cases}$$
Here $\Gamma(s)$ is the gamma function.

Let $F$ be a number field with the ring of integers $\mathcal{O}_F$. Let $\A_F$ be the ring of adeles of $F$ and $\A_{F,f}$ be its finite part. Let $\widehat{\mathcal{O}}_F$ be the closure of $\mathcal{O}_F$ in $\A_{F,f}$. When $F=\Q$, we denote $\A=\A_\Q$. Let $\zeta_F(s) = \prod_v \zeta_{F_v}(s)$ be the complete Dedekind zeta function of $F$, here $v$ ranges through the places of $F$.

Except in \S\,\ref{S:Weil representations and theta lifts}, let $\psi=\otimes_v\psi_v$ be an additive character of $\Q \backslash \A$ defined so that
\begin{align*}
\psi_p(x) & = e^{-2\pi \sqrt{-1}x} \mbox{ for }x \in \Z[p^{-1}],\\
\psi_{\infty}(x) & = e^{2\pi \sqrt{-1}x} \mbox{ for }x \in \R.
\end{align*}
We call $\psi$ (resp.\,$\psi_v$) the standard additive character of $\A$ (resp.\,$\Q_v$). For $a \in \Q^{\times}$ (resp.\,$a \in \Q_v^{\times}$), let $\psi^a$ (resp.\,$\psi_v^a$) be the character of $\A$ (resp.\,$\Q_v$) defined by $\psi^a(x) = \psi(ax)$ (resp.\,$\psi_v^a(x) = \psi_v(ax)$). For $a \in \Q_v^{\times}$, let $\gamma_{\Q_v}(a,\psi_p)$ be the Weil index defined in \cite[Appendix]{Rao1993} (cf.\,\cite[Lemma A.1]{Ichino2005}).

Denote $(\mbox{ },\mbox{ })_{\Q_v}$ the Hilbert symbol of $\Q_v$. For $a \in \Q^{\times}$ (resp.\,$a \in \Q_v^{\times}$), let $\chi_a$ be the quadratic character of $\A^{\times} / \Q^{\times}$ (resp.\,$\Q_v^{\times}$) defined by $\chi_a(x) = \prod_v (a,x_v)_{\Q_v}$ (resp.\,$\chi_a(x)=(a,x)_{\Q_v}$). By abuse of notation, we write $\chi_a$ for the primitive quadratic Dirichlet character associated to the Hecke character $\chi_a$.

For $n \in \Z_{>0}$, let $\GSp_{2n}$ be the similitude symplectic group defined by
$$\GSp_{2n} =\left \{ g \in \GL_{2n}\mbox{ } \left \vert \mbox{ } g\bp {\bf 0}_n & {\bf 1}_n \\ -{\bf 1}_n & {\bf 0}_n \ep {}^tg = \nu(g) \bp {\bf 0}_n & {\bf 1}_n \\ -{\bf 1}_n & {\bf 0}_n \ep ,\nu(g) \in \mathbb{G}_m    \right .\right \}.$$
The map $\nu : \GSp_{2n} \rightarrow \mathbb{G}_m$ is called the scale map. We denote by $\Sp_{2n}$ the kernel of the scale map.

For $n \in \Z_{>0}$, let $\frak{H}_n$ be the Siegel upper half space  of degree $n$ defined by
$$\frak{H}_n = \{Z \in {\rm M}_n(\C) \mbox{ }\vert \mbox{ }{}^tZ=Z,{\rm Im}(Z)>0 \}.$$ When $n=1$, we write $\frak{H}=\frak{H}_1.$

Let ${\bf B}$ (resp.\,${\bf D}$) be the standard Borel subgroup (resp.\,maximal torus) of $\GL_2$ consisting of upper triangular matrices (resp.\,diagonal matrices), and $B={\bf B} \cap \SL_2$ be the standard Borel subgroup of $\SL_2$. Let $U$ be the unipotent radical of $B$. We write
\begin{align*}
u(x)  = \bp  1 & x \\ 0 & 1 \ep &, & a(\nu) = \bp \nu & 0 \\ 0 & 1\ep &, & d(\nu) = \bp 1 & 0 \\ 0 & \nu  \ep & ,& t(a)= \bp a & 0 \\ 0 & a^{-1}\ep &, & w=\bp 0 & 1 \\ -1 & 0 \ep.
\end{align*}
Let 
\begin{align*}
{\rm SO}(2) & = \left \{ \left . k_{\theta}=\bp \cos \theta & \sin \theta \\ -\sin \theta & \cos \theta \ep \in \SL_2(\R)  \mbox{ }\right \vert \mbox{ }\theta \in \R / 2\pi \Z\right \},\\
{\rm SU}(2) & = \left \{ \left .\bp \alpha & \beta \\ -\overline{\beta} & \overline{\alpha} \ep \in \SL_2(\C) \mbox{ }\right \vert \mbox{ }\alpha,\beta \in \C,|\alpha|^2+|\beta|^2=1 \right \}.
\end{align*}

\subsection{Metaplectic groups $\wt{\SL}_2$}
Let $v$ be a place of $\Q$. Let $\wt{\SL}_2(\Q_v)$ be the $2$-fold metaplectic cover of $\SL_2(\Q_v)$. As a set, $\wt{\SL}_2(\Q_v) = \SL_2(\Q_v)\times\{\pm1\}$. The multiplication is defined by
$$(g_1,\epsilon_1)\cdot (g_2,\epsilon_2) = (g_1g_2,\epsilon_1\epsilon_2 \cdot c_v(g_1,g_2)),$$
where $c_v$ is Kubota's $2$-cocycle
\begin{align*}
c_v(g_1,g_2) &= (x(g_1g_2)x(g_1)^{-1},x(g_1g_2)x(g_2)^{-1})_{\Q_v},\\
x \left ( \bp a & b \\ c & d \ep \right ) &= \left \{\begin{array}{ll} c & \mbox{ if }c\neq 0, \\
d & \mbox{ otherwise.} \end{array} \right .
\end{align*}
By abuse of notation, we write $g$ for $(g,1) \in \wt{\SL}_2(\Q_v)$. For a subgroup $H$ of $\SL_2(\Q_v)$, we denote $\widetilde{H}$ the inverse image of $H$ in $\wt{\SL}_2(\Q_v)$. 

Let $K_0(p)$ be the standard Iwahori subgroup of $\SL_2(\Q_p)$ and
\begin{align*}
K_0(4) &= \left\{\left. \bp a & b \\ c & d \ep \in \SL_2(\Z_2)\mbox{ }\right\vert\mbox{ } c \equiv 0 \mbox{ mod }4 \right\},\\
K_1(4) &= \left\{\left. \bp a & b \\ c & d \ep \in \SL_2(\Z_2)\mbox{ }\right\vert\mbox{ } c \equiv 0 \mbox{ mod }4,\, d \equiv 1 \mbox{ mod }4 \right\}.
\end{align*}

For $v=p$, define a map $s_p : \SL_2(\Q_p)\rightarrow \{\pm1\}$ by
$$s_p\left(\bp a & b \\ c & d\ep  \right) = \begin{cases}
(c,d)_{\Q_p} & \mbox{ if }cd\neq 0\mbox{ and }{\rm ord}_{\Q_p}(c)\mbox{ is odd},\\
1 &\mbox{ otherwise}.
\end{cases}
$$
The map
\begin{align*}
\SL_2(\Q_p) \longrightarrow \wt{\SL}_2(\Q_p),\quad k \longmapsto (k,s_p(k))
\end{align*}
defines a splitting homomorphism when restricted to $\SL_2(\Z_p)$ if $p \neq 2$, and to $K_1(4)$ if $p=2$.

For $v=\infty$, we have an isomorphism
\begin{align*}
\R /4\pi\Z \longrightarrow \wt{\rm SO}(2),\quad \theta\longmapsto \tilde{k}_\theta,
\end{align*}
where 
$$\tilde{k}_\theta = \begin{cases}
(k_\theta, 1)  &\mbox{ if }-\pi < \theta \leq  \pi,\\
(k_\theta,-1)  &\mbox{ if } \pi < \theta \leq 3\pi.
\end{cases}
$$

For $v=2$, let
\begin{align*}
\epsilon_2\left ( \bp a & b \\ c & d \ep \right ) &= \left \{ \begin{array}{ll} \gamma_{\Q_2}(d,\psi_2)(c,d)_{\Q_2} & \mbox{ if }c \neq 0, \\
\gamma_{\Q_2}(d,\psi_2)^{-1} & \mbox{ otherwise.} \end{array} \right .
\end{align*}
Then we have a character of $\wt{K_0(4)}$
$$\wt{K_0(4)} \longrightarrow \C^\times,\quad (k,\epsilon)\longmapsto \epsilon\cdot\epsilon_2(k).$$

Let $\prod_{v}'\wt{\SL}_2(\Q_v)$ be the restricted direct product with respect to $\SL_2(\Z_p)$ for $p \neq 2$. The metaplectic group $\wt{\SL}_2(\A)$ is defined by
\begin{align*}
\wt{\SL}_2(\A) = \prod_{v}'\wt{\SL}_2(\Q_v) \mathlarger{\mathlarger{\mathlarger {\mathlarger{/}}}}  \left \{ (1,\epsilon_v) \in \oplus_v \{ \pm 1\} \mbox{ }\left \vert \mbox{ } \prod_{v}\epsilon_v=1\right . \right \} .
\end{align*}
By abuse of notation, we write $\prod_{v}(g_v,\epsilon_v)$ for its image in $\wt{\SL}_2(\A)$ under the natural quotient map.   We identify $\wt{\SL}_2(\Q_v)$ with a subgroup of $\wt{\SL}_2(\A)$ via the natural quotient map. We also identify $\SL_2(\Q)$ with a subgroup of $\wt{\SL}_2(\A)$ via the homomorphism
\begin{align*}
\SL_2(\Q)\longrightarrow \wt{\SL}_2(\A),\quad \gamma \longmapsto \prod_v(\gamma ,1).
\end{align*}

\subsection{Measures}\label{SS:measures}
For a connected linear algebraic group $G$ over $\Q$, we shall use the Tamagawa measure on $G(\A)$ throughout. In particular, if $G={\rm SO}(V)$ for some quadratic space $V$ over $\Q$, then 
$${\rm vol}(G(\Q)\backslash G(\A) )=2.$$
For the non-connected algebraic group ${\rm O}(V)$, we fix a Haar measure on ${\rm O}(V)(\A)$ normalized so that
$${\rm vol}({\rm O}(V)(\Q)\backslash {\rm O}(V)(\A))=1.$$

Let $dx_v$ be the Haar measure on $\Q_v$ such that ${\rm vol}(\Z_p,dx_p)=1$ if $v=p$, and $dx_{\infty}$ is the Lebesque measure if $v=\infty$. Let $d^{\times}a_{v}$ be the Haar measure on $\Q_v^{\times}$ defined by
$$d^{\times}a_v=\zeta_{\Q_v}(1)|a_v|_{\Q_v}^{-1}da_v.$$ 

Let $dg_v$ be the Haar measure on $\SL_2(\Q_v)$ defined by
$$dg_v=|a_v|_{\Q_v}^{-2}dx_xd^{\times}a_vdk_v$$
for $g_v=u(x_v)t(a_v)k_v$ with $x_v \in \Q_v$, $a_v \in \Q_v^{\times}$, and 
$$k_v \in \left \{ \begin{array}{ll}\SL_2(\Z_p) & \mbox{ if }v=p,\\ {\rm SO}(2) & \mbox{ if }v=\infty . \end{array} \right .     $$
Here $dk_p$ (resp.\,$dk_{\infty}$) is the Haar measure on $\SL_2(\Z_p)$ (resp.\,${\rm SO}(2)$) so that ${\rm vol}(\SL_2(\Z_p),dk_p)=1$ (resp.\,${\rm vol}({\rm SO}(2),dk_{\infty})=1$). Then the Tamagawa measure on $\SL_2(\A)$ is given by $$dg = \zeta_{\Q}(2)^{-1}\prod_{v}dg_v.$$

Let $d{ g}_p$ be the Haar measure on $\GL_2(\Q_p)$ defined by 
$$dg_p = \vert t_v \vert_{\Q_p}^{-1}d^{\times}z_pdx_pd^{\times}t_pdk_p$$
for $g_p = z_pu(x_p)a(t_p)k_p$ with $z_p,t_p \in \Q_p^{\times}$, $x_p \in \Q_p$, and $k_p \in \GL_2(\Z_p)$. Here $dk_p$ is the Haar measure on $\GL_2(\Z_p)$ so that ${\rm vol}(\GL_2(\Z_p),dk_p)=1.$ Let $dg_{\infty}$ be the Haar measure on $\GL_2(\R)$ defined by
$$dg_{\infty} = z_{\infty}^{-1}|t_{\infty}|_{\R}^{-1}dz_{\infty}dx_{\infty}d^{\times}t_{\infty}dk_{\infty}$$
for $g_{\infty}=z_{\infty}u(x_{\infty})a(t_{\infty})k_{\infty}$ with $z_{\infty} \in \R_{>0}$, $t_{\infty} \in \R^{\times}$, $x_{\infty} \in \R$, and $k_{\infty} \in {\rm SO}(2)$. Here $dk_{\infty}$ is the Haar measure on ${\rm SO}(2)$ so that ${
\rm vol}({\rm SO}(2),dk_{\infty})=1.$ Then the Tamagasa measure on $\GL_2(\A)$ is given by $$dg = 2\zeta_{\Q}(2)^{-1}\prod_{v}dg_v.$$

For $z \in \C$, let $|z|=\sqrt{z \overline{z}}$. Let $dz$ be the Haar measure on $\C$ defined to be twice the Lebsque measure on $\C$. Let $d^{\times}z = |z|^{-2}dz$ be the Haar measure on $\C^{\times}$. Let $dg_{\infty}$ be the Haar measure on $\SL_2(\C)$ defined by
$$dg_{\infty} = r_{\infty}^{-4}dx_{\infty}d^{\times}r_{\infty}d\theta dk_{\infty}$$
for $g_{\infty}=u(x_{\infty})t(r_{\infty}e^{\sqrt{-1}\theta})k_{\infty}$ with $x_{\infty} \in \C$, $r_{\infty} \in \R_{>0}$, $\theta \in \R/2\pi \Z$, and $k_{\infty} \in {\rm SU}(2)$. Here ${\rm vol}(\R  / 2\pi \Z,d\theta)={\rm vol}({\rm SU}(2),dk_{\infty})=1$.

\subsection{Special functions}
For $p,q \in \Z_{\geq 0}$, let ${}_pF_{q}(\alpha_1,\cdots,\alpha_p ; \beta_1 ,\cdots, \beta_q ; z)$ be the generalized hypergeometric function defined by
\begin{align*}
_{p}F_{q}(\alpha_1,\cdots,\alpha_p ; \beta_1 ,\cdots, \beta_q ; z) = \frac{\Gamma(\beta_1)\cdots\Gamma(\beta_q)}{\Gamma(\alpha_1)\cdots \Gamma(\alpha_q)}\sum_{n=0}^{\infty}\frac{\Gamma(\alpha_1+n)\Gamma(\alpha_p+n)}{n!\Gamma(\beta_1+n)\Gamma(\beta_q+n)}z^n
\end{align*}
whenever the series is converge. When $p=q+1$, the series converges for $|z|<1$.

For $n \in \Z_{\geq 0}$, let $H_n(x)$ be the Hermite polynomial defined by
$$H_n(x)=(-1)^ne^{x^2}\frac{d^n}{dx^n}(e^{-x^2}).$$

Let $K_{\nu}(z)$ be the modified Bessel function defined by
$$K_{\nu}(z) = \frac{1}{2}\int_{0}^{\infty}e^{-z(t+t^{-1})/2}t^{\nu-1}dt$$
if ${\rm Re}(z)>0$.

\section{Automorphic forms and $L$-functions}\label{S:Automorphic forms and $L$-functions}
Let $f\in S_{2\kappa'}(\Gamma_0(N))$ and $g\in S_{\kappa+1}(\Gamma_0(N))$ be normalized elliptic newforms. We assume Hypothesis (H) holds.

\subsection{Motivic $L$-functions}\label{SS:motivic L}
 Write
$$f(\tau) = \sum_{n=1}^{\infty}a_f(n)q^n ,\quad g(\tau) = \sum_{n=1}^{\infty}a_g(n)q^n.$$
For $p \nmid N$, the Hecke polynomial of ${ f}$ and ${ g}$ at $p$ has the factorization
\begin{align*}
1-a_{{ f}}(p)X+p^{2\kappa'-1}X^2&=(1-p^{\kappa'-1/2}\alpha_pX)(1-p^{\kappa'-1/2}\alpha_p^{-1}X),\\
1-a_{{g}}(p)X+p^{\kappa}X^2&=(1-p^{\kappa/2}\beta_pX)(1-p^{\kappa/2}\beta_p^{-1}X),
\end{align*}
and we let $$A_p=p^{\kappa'-1/2}\left (   \begin{array}{cc} \alpha_p & 0 \\ 0 & \alpha_p^{-1}   \end{array} \right ),\quad B_p=p^{\kappa}\left (  \begin{array}{ccc} \beta_p^2 &0 &0 \\ 0&1&0 \\0&0&\beta_p^{-2}  \end{array}   \right ).$$
For $p \mid N$, fix a $\tau_p \in \Z_p^{\times}$ such that $(p,\tau_p)_{\Q_p} = p^{-\kappa'+1}a_f(p)$. Put
$$\alpha_p = p^{-1/2}(p,\tau_p)_{\Q_p}.$$

The completed $L$-function $\Lambda(s,{\rm Sym}^2(g)\otimes f)$ is defined by
\begin{align*}
\Lambda(s,{\rm Sym}^2(g)\otimes f) = \prod_vL_v(s,{\rm Sym}^2(g)\otimes f),
\end{align*}
where
\begin{align*}
L_v(s,{\rm Sym}^2(g)\otimes f) = \begin{cases}
\det({\bf 1}_6-A_p\otimes B_p\cdot p^{-s})^{-1} &\mbox{ if }v=p \nmid N,\\
(1-a_{{ f}}(p)a_{{ g}}(p)^2p^{-s})^{-1}(1-a_{{ f}}(p)a_{{ g}}(p)^2p^{-s+1})^{-1} & \mbox{ if }v=p \mid N,\\
\Gamma_{\C}(s)\Gamma_{\C}(s-\kappa)\Gamma_{\C}(s-2\kappa'+1) & \mbox{ if }v=\infty.
\end{cases}
\end{align*}
The Euler product convergent absolutely for ${\rm Re}(s)>3/2-\kappa-\kappa'$, admits analytic continuation to $s \in \C$, and satisfying functional equation
$$\Lambda(2\kappa+2\kappa'-s,{\rm Sym}^2(g)\otimes f)=\Lambda(s,{\rm Sym}^2(g)\otimes f).$$

Let $\chi$ be a Dirichlet character. Define the completed twisted $L$-function 
$$\Lambda(s,f\otimes \chi) = \zeta_{\C}(s)\cdot\sum_{n=1}^{\infty}\frac{a_f(n)\chi(n)}{n^s}.$$

\subsection{Automorphic forms on $\GL_2(\A)$}
Let $${\bf K}_0(N\widehat{\Z}) = \prod_{p\mid N} {\bf K}_0(p) \prod_{p \nmid N}\GL_2(\Z_p)$$ be an open compact subgroup of $\GL_2(\A_f)$, where ${\bf K}_0(p)$ is the standard Iwahori subgroup of $\GL_2(\Q_p)$ for each prime $p \mid N$. Note that $\Gamma_0(N) = \SL_2(\Q) \cap {\bf K}_0(N\widehat{\Z})$.

Let $\bf f$ and $\bf g$ be cusp forms on $\GL_2(\A)$ determined by
\begin{align*}
{\bf f}(h) &= \det(h_{\infty})^{\kappa'}(c\sqrt{-1}+d)^{-2\kappa'}f(h_{\infty}(\sqrt{-1})),\\
{\bf g}(h) &= \det(h_{\infty})^{(\kappa+1)/2}(c\sqrt{-1}+d)^{-\kappa-1}g(h_{\infty}(\sqrt{-1}))
\end{align*}
for $h = \gamma h_{\infty} k \in \GL_2(\A)$ with $\gamma \in \GL_2(\Q)$, $k \in {\bf K}_0(N\widehat{\Z})$, and
$$h_{\infty} = \bp a & b \\ c & d\ep \in \GL_2^+(\R).$$
Let $\pi=\otimes_v \pi$ and $\sigma=\otimes_v \sigma_v$ be irreducible cuspidal automorphic representations of $\GL_2(\A)$ generated by $\bf f$ and $\bf g$, respectively. Let ${\bf g}^{\sharp} \in \sigma$ defined by
$${\bf g}^{\sharp} = \sigma(t(2^{-1})_2){\bf g}.$$

Let $\<f,f\>$ be the Petersson norm of $f$ defined by
\begin{align*}
\<f,f\> & = \frac{1}{\left [ \SL_2(\Z) : \Gamma_0(N)\right ]}\int_{\Gamma_0(N) \backslash \frak{H}}\vert f(\tau) \vert^2 y^{2\kappa'}\frac{d\tau}{y^2}.
\end{align*}
Similarly, we define $\<g,g\>$.
\subsubsection{Whittaker functions}
Let $W_{\bf f}$ be the Whittaker function of ${\bf f}$ with respect to $\psi$ defined by
\begin{align*}
W_{\bf f}(h) = \int_{\Q\backslash \A}{\bf f}(u(x)h)\overline{\psi(x)}dx.
\end{align*}
For each place $v$ of $\Q$, let $\mathcal{W}(\pi_v,\psi_v)$ be the space of Whittaker functions of $\pi_v$ with respect to $\psi_v$. Let $W_{{\bf f},v} \in \mathcal{W}(\pi_v,\psi_v)$ be a Whittaker function characterized as follows:
\begin{itemize}
\item If $v=p \nmid N$, then $W_{{\bf f},p}$ is $\GL_2(\Z_p)$-invariant and $W_{{\bf f},p}(1)=1$.
\item If $v=p \mid N$, then $W_{{\bf f},p}$ is ${\bf K}_0(p)$-invariant and $W_{{\bf f},p}(1)=1$.
\item If $v=\infty$, then 
$$W_{{\bf f},\infty}(z a(y)k_{\theta}) = e^{2\sqrt{-1}\kappa'\theta}y^{\kappa'}e^{-2\pi y} \I_{\R_{>0}}(y)$$
for $y,z \in \R^{\times}$ and $k_{\theta} \in {\rm SO}(2)$.
\end{itemize}
Note that we have a decomposition
\begin{align*}
W_{\bf f} = \prod_v W_{{\bf f},v}.
\end{align*}
Similarly, we define Whittaker functions $W_{\bf g}$, $W_{{\bf g},v}$ for places $v$ of $\Q$.

\subsubsection{Maass-Shimura differential operator}
Let $V_+$ be the weight raising operator defined by
\begin{align}\label{differential operator}
V_+=-\frac{1}{8\pi}\cdot\left[\bp 1 & 0\\ 0 & -1 \ep \otimes 1 + \bp 0 & 1 \\ 1 & 0\ep\otimes \sqrt{-1}\right] \in \frak{sl}_2(\R)\otimes_{\R}\C.
\end{align}
We identify $V_+$ as an element in the universal enveloping algebra of the Lie algebra $\frak{gl}_2(\R)\otimes_{\R}\C$. We have (cf.\,\cite[Lemma 3.3]{CC2017})
\begin{align}\label{E:Whittaker real}
V_+^{2m}W_{{\bf f},\infty}(a(y)k_{\theta}) = e^{2\sqrt{-1}\kappa\theta}y^{\kappa'}e^{-2\pi y}\sum_{j=0}^{2m}(-4\pi)^{j-2m}y^{j}\frac{\Gamma(2\kappa'+2m)}{\Gamma(2\kappa'+j)}{2m \choose j}\I_{\R_{>0}}(y)
\end{align}
for $y \in \R^{\times}$ and $k_{\theta} \in {\rm SO}(2)$.

\subsection{Automorphic forms on $\GL_2(\A_{\K})$}
\label{SS:Automorphic forms on GL_2 over K}
Let $\mathcal{K}$ be an imaginary quadratic field with ring of integers $\mathcal{O}$ and discriminant $-D<0$. Let $\psi_\K(x) = \psi\circ{\rm tr}_{\K/\Q}(\delta^{-1}x)$ be a non-trivial additive character of $\K\backslash\A_\K$. Let $\tau$ be the nontrivial automorphism of $\K$ over $\Q$. Put $\delta = \sqrt{-D}.$ We assume $(D,N)=1$.

Let 
$${\mathbb K}_0(N\widehat{\mathcal O}) = \prod_{p \mid N}{\mathbb K}_0(p)\prod_{p \nmid N}\GL_2(\mathcal{O}_p)$$
be an open compact subgroup of $\GL_2(\A_{\K,f})$, where ${\mathbb K}_0(p)$ is the standard Iwahori subgroup of $\GL_2(\K_p)$ for each prime $p \mid N$.

Let $\sigma_{\K}=\otimes_v\sigma_{\K,v}$ be the base change lift of $\sigma$ to $\GL_2(\A_\K)$. Note that $\sigma$ is not dihedral with respect to $\K$ since $(D,N)=1$. Therefore, $\sigma_\K$ is cuspidal. Consider a non-zero cusp form ${\bf g}_\K={\bf g}_{\K,\infty}{\bf g}_\K^{(\infty)} \in \sigma_{\K}=\sigma_{\K,\infty}\otimes\sigma_\K^{(\infty)}$ satisfying the following conditions:
\begin{itemize}
\item ${\bf g}_\K^{(\infty)}$ is ${\mathbb K}_0(N\mathcal{O})$-invariant.
\item ${\bf g}_{\K,\infty}$ is in the minimal ${\rm SU}(2)$-type of $\sigma_{\K,\infty}$ and $X\cdot {\bf g}_{\K,\infty}=0$. Here $X \in \frak{gl}_2(\C)\otimes_{\R}\C$ is defined by
$$X= \frac{1}{2}\bp 0 & -1 \\ 1 & 0 \ep\otimes 1 + \frac{1}{2}\bp  0 & \sqrt{-1} \\ \sqrt{-1} & 0 \ep\otimes \sqrt{-1}.$$
\end{itemize}
Note that these conditions uniquely determine ${\bf g}_\K \in \sigma_\K$ up to scalars. 

Define ${\bf g}_\K^\sharp \in \sigma_\K$ by
$${\bf g}_\K^{\sharp} = \sigma_\K({\bf t}_{\infty}){\bf g}_\K,$$
where
$${\bf t}_{\infty}=\displaystyle{\frac{1}{\sqrt{2}}\begin{pmatrix} 1 & -\sqrt{-1} \\ -\sqrt{-1} & 1 \end{pmatrix}} \in {\rm SU}(2).$$

\subsubsection{Whittaker functions}
Let $W_{{\bf g}_\K}$ be the Whittaker function of ${\bf g}_\K$ with respect to $\psi_\K$ defined by
\begin{align*}
W_{{\bf g}_\K}(h) = \int_{\K\backslash \A_\K}{\bf g}_\K(u(x)h)\overline{\psi_\K(x)}dx.
\end{align*}
We normalize ${\bf g}_\K$ so that
$$W_{{\bf g}_\K}(1) = (-\sqrt{-1})^{\kappa}D^{(-\kappa-1)/2}K_{\kappa}(4\pi D^{-1/2}).$$
For each place $v$ of $\Q$, let $\mathcal{W}(\sigma_{\K,v},\psi_{\K,v})$ be the space of Whittaker functions of $\sigma_{\K,v}$ with respect to $\psi_{\K,v}$. Let $W_{{\bf g}_\K,v} \in \mathcal{W}(\sigma_{\K,v},\psi_{\K,v})$ be a Whittaker function characterized as follows:
\begin{itemize}
\item If $v=p \nmid N$, then $W_{{\bf g}_\K,p}$ is $\GL_2(\mathcal{O}_p)$-invariant and $W_{{\bf g}_\K,p}(1)=1$.
\item If $v=p \mid N$, then $W_{{\bf g}_\K,p}$ is $\mathbb{K}_0(p)$-invariant and $W_{{\bf g}_\K,p}(1)=1$.
\item If $v=\infty$, then 
\begin{align}\label{E:Whittaker complex case}
W_{{\bf g}_\K,\infty}(a(\delta)zt(a)k)=a^{2\kappa+2}\sum_{n=0}^{2\kappa}{2\kappa \choose n}(\sqrt{-1})^{n}\alpha^{2\kappa-n}\overline{\beta}^{n}K_{\kappa-n}(4\pi a^2)
\end{align}
for $z \in \C^{\times}$, $a \in \R_{>0}$ and $k=\begin{pmatrix} \alpha&\beta \\ -\overline{\beta} & \overline{\alpha}  \end{pmatrix} \in {\rm SU}(2).$ The formula can be extracted from the calculation in \cite[\S\,6]{JLbook}
\end{itemize}
Note that we have a decomposition
\begin{align*}
W_{{\bf g}_\K} = \prod_v W_{{\bf g}_\K,v}.
\end{align*}

\subsection{Representations of $\widetilde{\SL}_2(\Q_v)$}\label{S:appendix}

Let $v$ be a place of $\Q$. Let $\psi_v$ be the standard additive character of $\Q_v$.
\subsubsection{Representations}
Let $\chi : \Q_v^\times \rightarrow \C^\times$ be a continuous character. Denote
$${\rm Ind}_{\widetilde{B}(\Q_v)}^{\widetilde{\SL}_2(\Q_v)}(\chi)$$
a principal series representation acting via right translation on the space consisting of smooth $\widetilde{K}$-finite functions $f : \widetilde{\SL}_2(\Q_v) \rightarrow \C$ satisfying
\begin{align*}
f((u(x)t(a),\epsilon)g)=\epsilon \gamma_{\Q_v}(a,\psi_v)\chi(a)|a|_{\Q_v}f(g)
\end{align*}
for $x \in \Q_v$, $a \in \Q_v^{\times}$, $\epsilon \in \{\pm 1 \}$, and $g \in \wt{\SL}_2(\Q_v)$. Here $K= \SL_2(\Z_p)$ if $v=p$, and $K={\rm SO}(2)$ if $v=\infty$.

For $v=p$, the principal series representation is irreducible if and only $\chi^2 \neq |\mbox{ }|_{\Q_p}^{\pm}.$ If $\chi^2 = |\mbox{ }|_{\Q_p}$, then the principal series representation has a unique irreducible sub-representation. We denote this representation by ${\rm St}(\chi)$ and call it a Steinberg representation of $\wt{\SL}_2(\Q_p)$. Note that if $\chi$ is an unramified character and $p \neq 2$, then the Steinberg representation ${\rm St}(\chi)$ has a unique $K_0(p)$-invariant vector up to scalars (cf.\,\cite[Lemma 8.3]{BM2007}).

For $v=\infty$, let $\chi = {\rm sgn}^{\kappa'-1}|\mbox{ }|_\R^{\kappa'-1/2}$ for some $\kappa' \in \Z_{>0}$. Then the principal series representation has a unique irreducible sub-representation. We denote this representation by ${\rm DS}(\kappa'+1/2)$ and call it a holomorphic discrete series representation of weight $\kappa'+1/2$. Note that ${\rm DS}(\kappa'+1/2)$ has minimal $\wt{{\rm SO}}(2)$-type $\kappa'+1/2$.

\subsubsection{Whittaker functions}
Let $\tilde{\pi}_v$ be a genuine irreducible admissible representation of $\wt{\SL}_2(\Q_v)$. We assume $\tilde{\pi}_v$ is a representation considered in the previous section. For $\xi \in \Q_v^{\times}$, let $\mathcal{W}(\tilde{\pi}_v,\psi_v^{\xi})$ be the space of Whittaker functions of $\tilde{\pi}_v$ with respect to $\psi_v^{\xi}$. Note that the space might be zero. Let $\Q_v(\tilde{\pi}_v)$ be a subset of $\Q_v^{\times}$ defined by
\begin{align*}
\Q_v(\tilde{\pi}_v) = \begin{cases}
\Q_p^{\times} & \mbox{ if }v=p\mbox{ and }\tilde{\pi}_p\mbox{ is a principal series representation},\\
\Q_p^{\times}\setminus (-\tau\Q_p^{\times,2}) & \mbox{ if }v=p\mbox{ and }\tilde{\pi}_p={\rm St}(\chi_{\tau,p}|\mbox{ }|_{\Q_p}^{1/2})\mbox{ for some }\tau \in \Q_p^{\times},\\
\R_{>0} & \mbox{ if }v=\infty\mbox{ and }\tilde{\pi}_\infty\mbox{ is a holomorphic discrete series representation}.
\end{cases}
\end{align*}
By the results of Waldspurger in \cite{Wald1980} and \cite{Wald1991}, $\xi \in \Q_v(\tilde{\pi}_v)$ if and only if $\mathcal{W}(\tilde{\pi}_v,\psi_v^{\xi})$ is non-zero. 

Let $v=p\neq 2$, $\tilde{\pi}_p = {\rm Ind}_{\widetilde{B}(\Q_p)}^{\widetilde{\SL}_2(\Q_p)}(\chi)$ for some unramified character $\chi$. For $\xi \in \Q_p(\tilde{\pi}_p)$ with $m={\rm ord}_{\Q_p}(\xi)$, let $\Psi_p(\xi;X) \in \C[X+X^{-1}]$ defined by
\begin{align*}
\begin{split}
\Psi_p(\xi;X)=\left \{ \begin{array}{ll}
\displaystyle{\frac{X^{m/2+1}-X^{-m/2-1}}{X-X^{-1}}-p^{-1/2}(p,-\xi)_{\Q_p}\frac{X^{m/2}-X^{-m/2}}{X-X^{-1}} }&\mbox{ if }m\geq 0\mbox{ is even},\\
\displaystyle{\frac{X^{(m-1)/2+1}-X^{-(m-1)/2-1}}{X-X^{-1}}} &\mbox{ if }m\geq 0  \mbox{ is odd},\\
0 & \mbox{ if }m<0.
         \end{array} \right .
\end{split}
\end{align*}
Up to scalars, there is a unique $\SL_2(\Z_p)$-invariant Whittaker function $W_{\xi,p} \in \mathcal{W}(\tilde{\pi}_p,\psi_p^{\xi})$. By \cite[Lemma A.3]{Ichino2005}, we can normalize $W_{\xi,p}$ so that
\begin{align}\label{E:Whittaker function p-adic}
W_{\xi,p}(t(p^n)) = p^{-n}\gamma_{\Q_p}(p^n,\psi_p)\Psi_p(p^{2n}\xi;\chi(p)).
\end{align}

Let $v=2$, $\tilde{\pi}_2 = {\rm Ind}_{\widetilde{B}(\Q_2)}^{\widetilde{\SL}_2(\Q_2)}(\chi)$ for some unramified character $\chi$. For $\xi=2^m u \in \Q_2(\tilde{\pi}_2)$ with $m={\rm ord}_{\Q_2}(\xi)$, let $\Psi_2(\xi;X) \in \C[X+X^{-1}]$ defined by
\begin{align*}
\begin{split}
&\Psi_2(\xi;X)\\
 &= \left \{\begin{array}{ll} \displaystyle{\frac{X^{m/2+1}-X^{-m/2-1}}{X-X^{-1}}-2^{-1/2}(2,\xi)_{\Q_2}\frac{X^{m/2}-X^{-m/2}}{X-X^{-1}}}  & \mbox{ if }m \geq 0\mbox{ is even, and }u \equiv -1 \mbox{ mod }4,\\
\displaystyle{\frac{X^{m/2}-X^{-m/2}}{X-X^{-1}}}& \mbox{ if }m \geq 0\mbox{ is even, and }u \equiv 1 \mbox{ mod }4,\\
\displaystyle{\frac{X^{(m-1)/2}-X^{-(m-1)/2}}{X-X^{-1}}}& \mbox{ if }m \geq 0\mbox{ is odd},\\
0 & \mbox{ if }m<0. \end{array}
\right .
\end{split}
\end{align*}
Let $\mathsf{U}$ and $\mathsf{W}$ be operators on $\tilde{\pi}_2$ defined by
\begin{align*}
\tilde{\pi}_2(\mathsf {U})f=\int_{\Z_2}\tilde{\pi}_2(u(x)t(2))fdx ,\quad \tilde{\pi}_2(\mathsf{W})f=\tilde{\pi}_2(w^{-1}t(2))f.
\end{align*}
By \cite[Lemma 3.6]{Ichino2005}, there exists a unique Whittaker function $W_{\xi,2} \in \mathcal{W}(\tilde{\pi}_2,\psi_2^{\xi})$ up to scalars such that
\begin{align*}
\tilde{\pi}_2(k)W_{\xi,2} = \epsilon_2(k)^{-1}W_{\xi,2},\quad\tilde{\pi}_2(\mathsf{W}\mathsf{U})W_{\xi,2} = 2^{-1/2}\zeta_8^{-1}W_{\xi,2}
\end{align*}
for $k \in K_0(4)$. By \cite[Lemma A.3]{Ichino2005}, we can normalize $W_{\xi,2}$ so that
\begin{align}\label{E:Whittaker function 2-adic}
W_{\xi,2}(t(2^n)) = 2^{-n}\Psi_2(2^{2n}\xi;\chi(2)).
\end{align}

Let $v=\infty$, $\tilde{\pi}_\infty = {\rm DS}(\kappa'+1/2)$ for some $\kappa' \in \Z_{>0}$. For $\xi \in \R(\tilde{\pi}_\infty)$, there exists a unique Whittaker function $W_{\xi,\infty} \in \mathcal{W}(\tilde{\pi}_\infty,\psi_\infty^{\xi})$ of $\wt{\rm SO}(2)$-type $\kappa'+1/2$ up to scalars. We deduce from \cite[p.\,24]{Wald1980} that $W_{\xi,\infty}$ can be normalized so that
\begin{align}\label{E:Whittaker function infty}
W_{\xi,\infty}(t(a)) = a^{\kappa'+1/2}e^{-2\pi\xi a^2}
\end{align}
for $a \in \R_{>0}$.

Let $v=p \neq 2$, $\tilde{\pi}_p = {\rm St}(\chi)$ for some unramified character $\chi$. For $\xi \in \Q_p(\tilde{\pi}_p)$ with $m={\rm ord}_{\Q_p}(\xi)$, let $\Psi_{p}(\xi;X) \in \C[X]$ defined by
\begin{align*}
\begin{split}
\Psi_p(\xi;X) = \left \{\begin{array}{ll} X^{m/2} & \mbox{ if }m\geq 0\mbox{ is even}, \\
X^{(m-1)/2} &   \mbox{ if }m\geq 0\mbox{ is odd},\\
0 & \mbox{ otherwise.} \end{array}       \right .
\end{split}
\end{align*}
Up to scalars, there exists a unique $K_0(p)$-invariant Whittaker function $W_{\xi,p} \in \mathcal{W}(\tilde{\pi}_p,\psi_p^{\xi})$. By \cite[(2.2.2) and Lemma 2.2.6]{ChenThesis}, we can normalize $W_{\xi,p}$ so that
\begin{align}\label{L:Whittaker function at 1 for unramified special}
W_{\xi,p}(t(p^n)) = p^{-n}\gamma_{\Q_p}(p^n,\psi_p)\Psi_p(p^{2n}\xi;\chi(p)).
\end{align}

\begin{lemma}\label{L:Whittaker function at w for Steinberg}
Let $\tilde{\pi}_p = {\rm St}(\chi)$ for some unramified character $\chi$. Let $\xi \in \Q_p(\tilde{\pi}_p)$ with $m = {\rm ord}_{\Q_p}(\xi)$. 
\\If $m \geq 0$, then 
$$W_{\xi,p}(w)=-p^{-1}\Psi_p(\xi;\chi(p)).$$
If $m=-1$, then $$W_{\xi,p}(w)=-p^{-1}\chi(p)^{-1}.$$
If $m \leq -2$, then $$W_{\xi,p}(w)=0.$$
\end{lemma}

\begin{proof}
The assertions are proved in \cite[Lemma 2.2.7]{ChenThesis}. We give a sketch of the proof for the reader's convenience. We have the double coset decomposition
\begin{align*}
\wt{B(\Q_p)} \backslash \wt{\SL_2(\Q_p)} / K_0(p) = \wt{B(\Q_p)}K_0(p) \sqcup \wt{B(\Q_p)}{w}K_0(p).
\end{align*} 
The space of $K_0(p)$-invariant sections in ${\rm Ind}_{\widetilde{B}(\Q_v)}^{\widetilde{\SL}_2(\Q_v)}(\chi)$ has a basis $\{F[1,1],F[1,p]\}$ defined by 
\begin{align*}
F[1,1](1)=0,\quad F[1,1](w) = 1 ,\\
F[1,p](1)=1,\quad F[1,p](w) = 0.
\end{align*}
Let 
$$f=F[1,1]-p\cdot F[1,p].$$
For $\xi \in \Q_p^\times$ with $m={\rm ord}_{\Q_p}(\xi)$, let $W_{f,\xi}$ be the Whittaker function defined by the integral
$$W_{f,\xi}(g) = \int_{\Q_p}f(w^{-1}{u}(x)g)\overline{\psi_p(\xi x)}dx$$
for $g \in \widetilde{\SL}_2(\Q_p)$. Put 
\begin{align*}
c_p(\xi)= \left \{ \begin{array}{ll} 2^{-1}\chi(p)^{-m/2} & \mbox{ if }m\mbox{ is even},\\
(1+p)^{-1}\chi(p)^{-(m+3)/2} &  \mbox{ if }m\mbox{ is odd}. \end{array}   \right . 
\end{align*}
Then we have
$$W_{f,\xi} = \begin{cases}
c_p(\xi)^{-1}W_{\xi,p} &\mbox{ if $\xi \in \Q_p(\tilde{\pi}_p)$},\\
0 & \mbox{ if $\xi \notin \Q_p(\tilde{\pi}_p)$}.
\end{cases}
$$
The calculation for $W_{f,\xi}$ is similar to the one in \cite[Appendix\,A.4]{Ichino2005}, we leave the detail to the readers.
\end{proof}

\subsection{Automorphic forms on $\widetilde{\SL}_2(\A)$}
We keep the notation of \S\,\ref{S:appendix}.

For $p \nmid N$, fix $s_p \in \C$ such that $\alpha_p = p^{-s_p}$. Note that ${\rm Re}(s_p)=0$ by the Ramanujan conjecture. For $p \mid N$, recall we have fixed a $\tau_p \in \Z_p^{\times}$ such that $(p,\tau_p)_{\Q_p} = p^{-\kappa'+1}a_f(p)$. Let $\tilde{\pi}_v$ be a genuine irreducible admissible representation of $\widetilde{\SL}_2(\Q_v)$ defined by
\begin{align*}
\tilde{\pi}_v = \begin{cases}
{\rm Ind}_{\widetilde{B}(\Q_p)}^{\widetilde{\SL}_2(\Q_p)}(|\mbox{ }|_{\Q_p}^{s_p}) & \mbox{ if }v=p \nmid N,\\
{\rm St}(\chi_{\tau_p,p}|\mbox{ }|_{\Q_p}^{1/2}) & \mbox{ if }v=p \mid N,\\
{\rm DS}(\kappa'+1/2) & \mbox{ if }v=\infty. 
\end{cases}
\end{align*}
Let $\tilde{\pi} = \otimes_v \tilde{\pi}_v$. By the results of Wladspurger in \cite{Wald1980}, $\tilde{\pi}$ is an irreducible genuine cuspidal automorphic representation of $\widetilde{\SL}_2(\A)$. Consider a non-zero cusp form ${\bf h} \in \tilde{\pi}$ satisfying the following conditions:
\begin{itemize}
\item $\tilde{\pi}((k,s_p(k))){\bf h} = {\bf h}$ for $k \in \SL_2(\Z_p)$, if $p \nmid 2N$.
\item $\tilde{\pi}((k,s_p(k))){\bf h} = {\bf h}$ for $k \in K_0(p)$, if $p \mid N$.
\item $\tilde{\pi}(k){\bf h} = \epsilon_2(k)^{-1}{\bf h}$ for $k \in K_0(4)$.
\item $\tilde{\pi}(\mathsf{W}\mathsf{U}){\bf h} = 2^{-1/2}\zeta_8^{-1}{\bf h}$.
\item $\tilde{\pi}(\tilde{k}_{\theta}){\bf h} = e^{\sqrt{-1}(\kappa'+1/2)\theta}{\bf h},$ for $\tilde{k}_{\theta} \in \widetilde{{\rm SO}}(2)$.
\end{itemize}
Note that these conditions uniquely determine ${\bf h} \in \tilde{\pi}$ up to scalars.

Recall $S_{\kappa'+1/2}^+(\Gamma_0(4N))$ is the Kohnen's plus space, consisting of cusp forms $h'(\tau) = \sum_{n=1}^{\infty}c_{h'}(n)q^n$ of level $\Gamma_0(4N)$ and weight $\kappa'+1/2$ on $\frak{H}$ such that
$$c_{h'}(n)=0 \mbox{ for all }(-1)^{\kappa'}n \equiv 2,3 \mbox{ mod }4.$$ Let $h \in S_{\kappa'+1/2}^+(\Gamma_0(4N))$ be the half-integral weight modular form associated to ${\bf h}$ defined by
\begin{align*}
h(\tau) = y^{\kappa'+1/2}{\bf h}(g_{\infty})
\end{align*}
for $\tau =x+\sqrt{-1}y \in \frak{H}$ and $g_{\infty} = u(x)t(y^{1/2}) \in \SL_2(\R)$. Then $h$ is a newform in the sense of Kohnen in \cite{Kohnen1982}. We call it the newform associated to $f$. Note that $h$ is well-defined up to scalars.

Let $\<h,h\>$ be the Petersson norm of $h$ defined by
\begin{align*}
\<h,h\> & = \frac{1}{6\left [ \SL_2(\Z) : \Gamma_0(N)\right ]}\int_{\Gamma_0(4N) \backslash \frak{H}}\vert h(\tau) \vert^2 y^{\kappa'+1/2}\frac{d\tau}{y^2}.
\end{align*}

\subsubsection{Whittaker functions}
Let $\xi \in \Q^{\times}$, the Whittaker function $W_{{\bf h},\xi}$ of ${\bf h}$ with respect to $\psi^{\xi}$ is defined by
$$W_{{\bf h},\xi}(g) = \int_{\Q\backslash\A}{\bf h}(u(x)g) \overline{\psi(\xi x)}dx.$$
By the results of Waldspurger in \cite{Wald1980} and \cite{Wald1991}, $W_{{\bf h},\xi}$ is non-zero if and only if the following conditions are satisfied:
\begin{itemize}
\item $\xi \in \bigcap_v \Q_v(\tilde{\pi}_v)$.
\item $\Lambda(\kappa',f\otimes \chi_{-\xi}) \neq 0$.
\end{itemize}
Let $\xi \in \bigcap_v \Q_v(\tilde{\pi}_v)\bigcap \Q^{\times}$. Let $W_{\xi,v} \in \mathcal{W}(\tilde{\pi}_v,\psi_v^\xi)$ be the Whittaker functions defined in (\ref{E:Whittaker function p-adic})-(\ref{L:Whittaker function at 1 for unramified special}). 
Write $\xi = \frak{d}_{\xi}\frak{f}_{\xi}^2$ with $\frak{d}_{\xi} \in \Z_{>0}$, $\frak{f}_{\xi} \in \Q_{>0}$ so that $-\frak{d}_{\xi}$ is the fundamental discriminant of $\Q(\sqrt{-\xi}) / \Q$. Note that there is a well-known relation between $c_h(\xi)$ and $c_h(\frak{d}_\xi)$ given by
$$c_{h}(\xi)=c_{h}(\frak{d}_{\xi})\frak{f}_{\xi}^{\kappa'-1/2}\prod_{p}\Psi_p(\xi;\alpha_p).$$
We have a decomposition
\begin{align*}
W_{{\bf h},\xi} =  c_h(\frak{d}_{\xi})\frak{f}_{\xi}^{\kappa'-1/2}\prod_{v}W_{\xi,v}.
\end{align*}

\subsubsection{Maass-Shimura differential operator}
Let $V_+$ be the weight raising operator defined as in (\ref{differential operator}). We identify $V_+$ as an element in the universal enveloping algebra of the complexified Lie algebra of $\widetilde{\SL}_2(\R)$. For $\xi \in \R_{>0}$, we have
\begin{align}\label{E:metaplectic weight raising Whittaker  function archimedean case}
{V}_+^mW_{\xi,\infty}(t(a)\tilde{k}_\theta)=e^{\sqrt{-1}(\kappa+1/2)\theta}a^{\kappa'+1/2}e^{-2\pi\xi a^2}\sum_{j=0}^{m}(-4\pi)^{j-m}a^{2j}\frac{\Gamma(\kappa'+1/2+m)}{\Gamma(\kappa'+1/2+j)}{m \choose j} 
\end{align}
for $a \in \R_{>0}$ and $\tilde{k}_\theta \in \widetilde{\rm SO}(2)$.

\subsection{Automorphic forms on $\GSp_4(\A)$}
Let $${\bf K}_0^{(2)}(N\widehat{\Z}) = \prod_{p\mid N} {\bf K}_0^{(2)}(p) \prod_{p \nmid N}\GSp_4(\Z_p)$$ be an open compact subgroup of $\GSp_4(\A_f)$, where ${\bf K}_0^{(2)}(p)$ is the standard Siegel congruence subgroup of $\GSp_4(\Q_p)$ with level $p$ for each prime $p \mid N$. Note that $\Gamma_0^{(2)}(N) = \Sp_4(\Q) \cap {\bf K}_0^{(2)}(N\widehat{\Z})$.

Let $h \in S_{\kappa'+1/2}^+(\Gamma_0(4N))$ be a newform associated to $f$.
The Saito-Kurokawa lifts of $h$ is a Siegel cusp form $F \in S_{\kappa'+1}(\Gamma_0^{(2)}(N))$ defined by
\begin{align*}
F(Z) = \sum_{B}A(B) e^{2\pi\sqrt{-1}\,{\rm tr}(BZ)},
\end{align*}
where $B$ runs over all positive definite half-integral symmetric matrices of size $2$, and 
$$A(B)=\sum_{d \mid (b_1,b_2,b_3),(d,N)=1}d^{\kappa'}c_h\left(\frac{4b_1b_3-b_2^2}{d^2} \right)$$
for $B= \bp b_1 & b_2/2 \\ b_2/2 & b_3\ep$ with $b_1,b_2,b_3\in\Z$, $b_1,b_3 >0$, and $4b_1b_3-b_2^2>0$. Let ${\bf F}$ be a cusp form on $\GSp_4(\A)$ associated to $F$ determined by
$${\bf F}(h) = \det(h_{\infty})^{(\kappa'+1)/2}\det(C\sqrt{-1}+D)^{-\kappa'-1}F(h_{\infty}(\sqrt{-1}))$$
for $h=\gamma_{\infty}h_{\infty}k$ with $\gamma \in \GSp_4(\Q)$, $k \in {\bf K}_0^{(2)}(N\widehat{\Z})$, and 
$$h_{\infty} = \bp A&B\\C&D\ep \in \GSp_4^+(\R).$$

\subsubsection{Maass differential operator}
For $Z \in \frak{H}_2$, write
$Z= \begin{pmatrix} \tau_1 & z \\ z & \tau_2   \end{pmatrix},$ $\tau_i=x_i+\sqrt{-1}y_i$ for $i=1,2$ and $z=u+\sqrt{-1}v$. Let $Y=\begin{pmatrix}y_1 & v \\ v & y_2 \end{pmatrix}$ be the imaginary part of $Z$. For $r \in \N$, let $\Delta_r$ be the weight $r$ Maass differential operator (cf.\,\cite{MaassLNM}) defined by
\begin{align}\label{E:Maass differential operator}
\Delta_r&=\frac{1}{32\pi^2} \left[ r(2r-1)\det (Y)^{-1}-8 \frac{\partial^2}{\partial \tau_1 \partial \tau_2}+2 \frac{\partial^2}{\partial z^2} +(4r-2) \sqrt{-1}\det (Y)^{-1} \left(y_1 \frac{\partial}{\partial \tau_1}+y_2 \frac{\partial}{\partial \tau_2}+v \frac{\partial}{\partial z}\right) \right]. 
\end{align}
We write $\Delta_{\kappa'+1}^m = \Delta_{\kappa-1}\circ \cdots \circ \Delta_{\kappa'+1}$. Then $\Delta_{\kappa'+1}^mF$ is a nearly holomorphic Siegel modular form of level $\Gamma_0^{(2)}(N)$ and weight $\kappa+1$. An induction argument shows that the Fourier expansion of $\Delta_{\kappa'+1}^m {F}$ is given by
\begin{align}\label{E:Fourier expansion of nearly holomorphic Siegel modular form}
\begin{split}
\Delta_{\kappa'+1}^{m}{F}(Z)&=\sum_{B}A(B) e^{2\pi \sqrt{-1}\, {\rm tr}(BZ)}\\
&\times \sum_{j=0}^{m}(-4\pi)^{j-m}\frac{\Gamma(\kappa'+m+1/2)}{\Gamma(\kappa'+j+1/2)} {m \choose j}\det(B)^j\det(Y)^{j-m} \\
&\times  \sum_{i=0}^{m-j}\frac{(2m-2j-i)!}{i! (m-j-i)!}(4\pi )^{i+j-m}\sum_{l=0}^{i}\frac{(\kappa+1)! (-4\pi)^{-l}}{(\kappa+1-l)!}{i \choose l}{\rm tr}(BY)^{i-l},
\end{split}
\end{align}
for $Z=X+\sqrt{-1}Y \in \frak{H}_2$, here $B$ runs over all positive definite half-integral symmetric matrices of size $2$.

Define $X_+,P_{0+},P_{1+}\in \frak{sp}_4(\R)\otimes_\R{\C}$ by
\begin{align*}
X_+&=\frac{1}{2}\begin{pmatrix} 1 & 0 & 0 & 0 \\ 
0 & 0 & 0 & 0 \\
0 & 0 & -1 & 0 \\
0 & 0 & 0 & 0 \end{pmatrix}\otimes 1 +\frac{1}{2}\begin{pmatrix} 0 & 0 & 1 & 0 \\ 
0 & 0 & 0 & 0 \\
1 & 0 & 0 & 0 \\
0 & 0 & 0 & 0 \end{pmatrix}\otimes \sqrt{-1}  ,\\
P_{0+}&=\frac{1}{2}\begin{pmatrix} 0 & 0 & 0 & 0 \\ 
0 & 1 & 0 & 0 \\
0 & 0 & 0 & 0 \\
0 & 0 & 0 & -1 \end{pmatrix}\otimes 1 + \frac{1}{2}\begin{pmatrix} 0 & 0 & 0 & 0 \\ 
0 & 0 & 0 & 1 \\
0 & 0 & 0 & 0 \\
0 & 1 & 0 & 0 \end{pmatrix}\otimes \sqrt{-1},\\
P_{1+}&=\frac{1}{2}\begin{pmatrix} 0 & 1 & 0 & 0  \\ 
1 & 0 & 0 & 0 \\
0 & 0 & 0 & -1 \\
0 & 0 & -1 & 0 \end{pmatrix}\otimes 1 +\frac{1}{2}\begin{pmatrix} 0 & 0 & 0 & 1  \\ 
0 & 0 & 1 & 0 \\
0 & 1 & 0 & 0 \\
1 & 0 & 0 & 0 \end{pmatrix}\otimes \sqrt{-1}.
\end{align*}
Let ${D}_+$ be an element in the universal enveloping algebra of $\frak{sp}_4(\R)\otimes_\R{\C}$ defined by
\begin{align}\label{E:Lie algebra Maass differential operator}
{D}_+=-\frac{1}{64\pi^2} (P_{1+}^2 - 4X_+ P_{0+}).
\end{align}
Then $D_+$ is a weight raising operator sending weight $(k,l)$ to weight $(k+2,l+2)$.
Note that ${D}_+^m{\bf F}$ is the cusp form on $\GSp_4(\A)$ associated to $\Delta_{\kappa'+1}^m {F}$ (cf.\,\cite[Remark 3.12]{PSS2016}).

\section{Weil representations and theta lifts}\label{S:Weil representations and theta lifts}

\subsection{Quadratic spaces}
Let $F$ be a filed of characteristic not $2$. Let $(V,Q)$ be a non-degenerate quadratic space of dimension $m$ over $F$. The associated non-degenerate symmetric bilinear form $(\mbox{ },\mbox{ })$ is defined by
$$(x,y) = Q[x+y]-Q[x]-Q[y]$$
for $x,y \in V$. Define the orthogonal similitude group ${\rm GO}(V)$ by
$${\rm GO}(V) = \{ h \in \GL(V)\mbox{ }\vert \mbox{ }(hx,hy)=\nu(h)(x,y) \mbox{ for }x,y \in V \},$$ 
here $\nu : {\rm GO}(V) \rightarrow {\mathbb G}_m$ is the scale map. When $m$ is even, let
$${\rm GSO}(V) = \{h \in {\rm GO}(V) \mbox{ }\vert \mbox{ }\det(h)=\nu(h)^{m/2} \}.$$
Let ${\rm O}(V)$ and ${\rm SO}(V)$ be the orthogonal group and special orthogonal group defined by
\begin{align*}
{\rm O}(V) &= \{ h \in {\rm GO}(V) \mbox{ }\vert \mbox{ } \nu(h)=1 \},\\
{\rm SO}(V) &= \{ h \in {\rm GO(V)} \mbox{ }\vert \mbox{ } \det(h)=\nu(h)=1 \}.
\end{align*}
\subsection{Weil representations}\label{SS:Weil rep}
Let $F$ be a local field of characteristic zero. Let $\psi$ be a non-trivial additive character of $F$. Let $(V,Q)$ be a quadratic space over $F$ with dimension $m$, and $\widetilde{\Sp}_{2n}(F)$ be the $2$-fold metaplectic cover of $\Sp_{2n}(F)$. As a set, $\widetilde{\Sp}_{2n}(F) = \Sp_{2n}\times \{\pm1\}.$ The multiplication is given by Rao's $2$-cocycle defined in \cite[Theorem 5.3]{Rao1993}. By abuse of notation, we write $g$ for the element $(g,1) \in \widetilde{\Sp}_{2n}(F)$.

Let $\mathcal{S}(V^n(F))$ be the space of Bruhat-Schwartz functions on $V^n(F).$ Let $\omega$ be the Weil representation of $\wt{\Sp}_{2n}(F)\times {\rm O}(V)(F)$ on $\mathcal{S}(V^n(F))$ with respect to $\psi$ (cf.\,\cite[\S\,5]{Kudla1994} and \cite[\S\,4.2]{Ichino2005}). When $F$ is archimedean, let $S(V^n(F))$ be the subspace of $\mathcal{S}(V^n(F))$ consisting of functions which correspond to polynomials in the Fock model. When $F$ is non-archimedean, let $S(V^n(F)) = \mathcal{S}(V^n(F))$.

When $m$ is even, we may regard $\omega$ as a representation of $\Sp_{2n}(F)\times {\rm O}(V)(F)$. The Weil representation $\omega$ can be extend to $R(F)$. Here 
$$R={\rm G}(\Sp_{2n} \times {\rm O}(V)) = \{ (g,h) \in \GSp_{2n} \times {\rm GO}(V) \mbox{ }\vert \mbox{ }\nu(g)=\nu(h) \}.$$
The action is given by
\begin{align*}
\omega(g,h)\varphi = \omega\left (g \bp {\bf 1}_n & 0 \\ 0 & \nu(g)^{-1}{\bf 1}_n\ep,1 \right )L(h)\varphi
\end{align*}
for $(g,h) \in R(F)$ and $\varphi \in \mathcal{S}(V^n(F))$. Here
$$L(h)\varphi(x)=\vert \nu(h)  \vert_F^{-mn/4}\varphi(h^{-1}x).$$
\subsubsection{Change of polarizations}\label{SS:Change of polarizations}
Let $n=1$. Let ${\bf e} \in V$ be an isotropic vector. Fix an isotropic vector ${\bf e}^* \in V$ such that $({\bf e},{\bf e}^*)=r \in F^{\times}$. Let $(V_1,Q_1)$ be a non-degenerated quadratic space over $F$ defined by
$$V_1=(F {\bf e}+F{\bf e}^*)^{\perp},\quad Q_1=Q\vert_{V_1}.$$
Let 
\begin{align*}
\mathcal{S}(V(F)) \longrightarrow  \mathcal{S}(V_1(F))\otimes \mathcal{S}(F^2),\quad
\varphi \longmapsto \hat{\varphi}
\end{align*}
be the partial Fourier transform defined by
\begin{align*}
\hat{\varphi}(x;y) = |r|_F^{1/2}\int_{F}\varphi(z{\bf e}+x_1+y_1{\bf e}^*)\psi(ry_2z)dz
\end{align*}
for $x_1 \in V_1(F)$, $y=(y_1,y_2)\in F^2$. Here $dz$ is the self-dual measure on $F$ with respect to $\psi$. Let $\hat{\omega}$ be the representation of $\widetilde{\SL}_2(F)\times {\rm O}(V)(F)$ on $\mathcal{S}(V_1(F))\otimes \mathcal{S}(F^2)$ defined by
$$\hat{\omega}(g,h)\hat{\varphi} = (\omega(g,h)\varphi)^{\mathlarger{\hat{}}}.$$ 
When $m$ is even, similarly we extend $\hat{\omega}$ to a representation of $R(F)$. If $\hat{\varphi}=\varphi_1\otimes \varphi_2$ with $\varphi_1 \in \mathcal{S}(V_1(F))$ and $\varphi_2 \in \mathcal{S}(F^2)$, then 
\begin{align*}
\hat{\omega}((g,\epsilon),1)\hat{\varphi}(x_1;y)=\omega((g,\epsilon),1)\varphi_1(x_1)\cdot \varphi_2(yg)
\end{align*}
for $(g,\epsilon) \in \wt{\SL}_2(F)$.

\subsection{Theta lifts}\label{SS:Theta lifts}
Let $F$ be a number field. Let $\psi=\otimes_v\psi_v$ be a non-trivial additive character of $F \backslash \A_F$. Let $(V,Q)$ be a quadratic space over $F$ with dimension $m$, and $\widetilde{\Sp}_{2n}(\A_F)$ be the $2$-fold metaplectic cover of $\Sp_{2n}(\A_F)$.

Let $\mathcal{S}(V^n(\A_F))$ the space of Bruhat-Schwartz functions on $V^n(\A_F).$ For each place $v$ of $F$, let $\omega_v$ be the Weil representation of $\wt{\Sp}_{2n}(F_v)\times {\rm O}(V)(F_v)$ on $\mathcal{S}(V^n(F_v))$. Let $\omega=\otimes_v\omega_v$ be the global Weil representation of 
 $\wt{\Sp}_{2n}(\A_F)\times {\rm O}(V)(\A_F)$ on $\mathcal{S}(V^n(\A_F))$. Let $S(V^n(\A_F)) = \otimes_v S(V^n(F_v))$. When $m$ is even, the Weil representation $\omega$ can be extend to $R(\A_F)$ in a way similar to the local case. 

For $\varphi \in \mathcal{S}(V^n(\A_F))$, define a theta function
$$\theta(g,h;\varphi) = \sum_{x \in V^n(F)}\omega(g,h)\varphi(x)$$
for $(g,h) \in \wt{\Sp}_{2n}(\A_F)\times {\rm O}(V)(\A_F)$. This is a slowly increasing function on $\widetilde{\Sp}_{2n}(\A_F)\times {\rm O}(V)(\A_F)$, and define an automorphic form if $\varphi \in S(V^n(\A_F)).$ When $m$ is even, we regard it as a function on $\wt{\Sp}_{2n}(\A_F)\times {\rm O}(V)(\A_F)$.

For example, when $F=\Q$ and $\psi$ is the standard additive character, let $(V^{(0)},Q^{(0)})$ be a quadratic space over $\Q$ defined by
$$V^{(0)} = \Q,\quad Q^{(0)}[x] = x^2.$$
Let $\varphi^{(0)}=\otimes_v\varphi_v^{(0)} \in S(\A)$ defined as follows:
\begin{itemize}
\item If $v=p$, then $\varphi_p^{(0)}=\I_{\Z_p}.$
\item If $v=\infty$, then $\varphi_\infty^{(0)}(x)=e^{-2\pi x^2}$.
\end{itemize}
Let $\Theta$ be an automorphic form on $\wt{\SL}_2(\A)$ defined by
$$\Theta(g) = \theta(g,1;\varphi^{(0)}).$$

For a cusp form $f$ on $\wt{\Sp}_{2n}(\A_F)$ with $f((g,\epsilon)) = \epsilon^mf(g)$, and $\varphi \in S(V^n(\A))$. Let $\theta(f,\varphi)$ be an automorphic form on ${\rm O}(V)(\A_F)$ defined by
$$\theta(h;f,\varphi) = \int_{\Sp_{2n}(F)\backslash\Sp_{2n}(\A_F)}\theta(g,h;\varphi)f(g)dg$$
for $h \in {\rm O}(V)(\A_F)$. For a cusp form $f$ on ${\rm O}(V)(\A_F)$ or ${\rm SO}(V)(\A_F)$, and $\varphi \in S(V^n(\A))$. Similarly we define an automorphic form $\theta(f,\varphi)$ on $\wt{\Sp}_{2n}(\A_F)$ with $\theta((g,\epsilon);f,\varphi) = \epsilon^m \theta(g;f,\varphi)$.

Assume $m$ is even. For a cusp form $f$ on $\GSp_{2n}(\A_F)$ and $\varphi \in S(V^n(\A))$. Let $\theta(f,\varphi)$ be an automorphic form on ${\rm GO}(V)(\A_F)$ defined by
$$\theta(h;f,\varphi) = \int_{\Sp_{2n}(F)\backslash\Sp_{2n}(\A_F)}\theta(gg',h;\varphi)f(gg')dg$$
for $(g',h) \in R(\A_F)$. For a cusp form $f$ on ${\rm GO}(V)(\A_F)$ or ${\rm GSO}(V)(\A_F)$, and $\varphi \in S(V^n(\A))$. Similarly we define an automorphic form $\theta(f,\varphi)$ on $\GSp_{2n}(\A_F)^+$. Here
$$\GSp_{2n}(\A_F)^+ = \{g \in \GSp_{2n}(\A_F) \mbox{ }\vert\mbox{ }\nu(g) \in \nu({\rm GO}(V)(\A_F))\}.$$

Let $f_1$ and $f_2$ be two genuine automorphic forms on $\wt{\Sp}_{2n}(\A_F)$, or two automorphic forms on ${\Sp}_{2n}(\A_F)$. Assume one of them is a cusp form. Define a Hermitian pairing
$$\<f_1,f_2\>_{\Sp_{2n}} = \int_{\Sp_{2n}(F)\backslash{\Sp}_{2n}(\A_F)}f_1(g)\overline{f_2(g)}dg.$$
Similarly we define Hermitian pairings $\<\cdot,\cdot\>_{{\rm O}(V)}$ and $\<\cdot,\cdot\>_{{\rm SO}(V)}$.

\section{Seesaw identity (I)}\label{S:seesaw I}

We keep the notation of \S\,\ref{S:Automorphic forms and $L$-functions}. Let $(V^{(0)},Q^{(0)})$ be the quadratic space defined in \S\,\ref{SS:Theta lifts}. The aim of this section is to prove an explicit seesaw identity in Proposition \ref{P:SO(2,2)-period to SL_2 period}. 
\subsection{Main results}\label{SS:5.1}
\subsubsection{Setting}
Let $(V^{(1)},Q^{(1)})$ and $(V^{(2)},Q^{(2)})$ be quadratic spaces over $\Q$ defined by
\begin{align*}
V^{(1)} &= \left \{\left . \bp x_1 & x_2 \\ x_3 & x_4 \ep \mbox{ } \right \vert \mbox{ } x_1,x_2,x_3,x_4 \in \mathbb{G}_a\right \},\quad Q^{(1)}[x]=\det(x),\\
V^{(2)} &= \left \{ x \in {\rm M}_4  \mbox{ }\left \vert \mbox{ } x \bp {\bf 0}_2 & { \bf 1}_2 \\ -{ \bf 1}_2 & {\bf 0}_2\ep=\bp {\bf 0}_2 & { \bf 1}_2 \\ -{ \bf 1}_2 & {\bf 0}_2\ep {^t}x \right .\right \}, \quad  Q^{(2)}[x]= 4^{-1}{\rm tr}(x^2).\\
\end{align*}
We have an isomorphism of quadratic spaces 
\begin{align*}
V^{(0)} \oplus V^{(1)} &\longrightarrow V^{(2)}\\
\left ( x, \bp x_1 & x_2 \\ x_3 & x_4 \ep \right ) &\longmapsto \bp x & x_1 & 0 & x_2 \\ x_4 & -x & -x_2 & 0 \\ 0 & x_3 & x & x_4 \\ -x_3 & 0 & x_1 & -x \ep .
\end{align*}
Recall we have exact sequences
\begin{align*}
&1 \longrightarrow   {\mathbb G}_m \stackrel{\iota^{(1)}}{\longrightarrow} \GL_2\times \GL_2 \stackrel{\rho^{(1)}}{\longrightarrow} \text{GSO}( V^{(1)}  ) \longrightarrow 1,\\
&1 \longrightarrow {\mathbb G}_m \stackrel{\iota^{(2)} }{\longrightarrow} \GSp_4 \stackrel{\rho^{(2)} }{\longrightarrow} {\rm SO}(V^{(2)} ) \longrightarrow 1,
\end{align*}
where 
\begin{align*}
\iota^{(1)}(a) &= (a {\bf 1}_2,a {\bf 1}_2), \quad \iota^{(2)}(a) = a {\bf 1}_4,\\
\rho^{(1)}(h_1,h_2)x &= h_1xh_2^{-1},\quad \rho^{(2)}(h)y = hyh^{-1},
\end{align*}
for $a \in \mathbb{G}_m$, $(h_1,h_2)\in\GL_2\times\GL_2$, $h \in \GSp_4$, and $x \in V^{(1)}$, $y \in V^{(2)}$. Note that the natural embedding ${\rm SO}(V^{(1)}) \hookrightarrow {\rm SO}(V^{(2)})$ is compatible with the embedding
\begin{align*}
{\rm G}(\SL_2 \times \SL_2) &\longrightarrow {\rm GSp}_4\\
\left ( \begin{pmatrix} a_1 & b_1 \\ c_1 & d_1 \end{pmatrix},\begin{pmatrix} a_2 & b_2 \\ c_2 & d_2 \end{pmatrix}\right ) & \longmapsto \left ( \begin{array}{cccc} a_1 & 0 & b_1 & 0 \\ 0 & a_2 & 0 & b_2 \\   c_1 & 0 & d_1 & 0 \\ 0 & c_2 & 0 & d_2       \end{array} \right ).
\end{align*}
\subsubsection{Schwartz functions}
Define $\varphi^{(1)} = \otimes_v \varphi_v^{(1)} \in S(V^{(1)}(\A))$ as follows:
\begin{itemize}
\item If $v=p \nmid N$, then $\varphi_p^{(1)}$ is the function $\varphi_v^{(3)}$ defined in \cite[\S\,6.3]{Ichino2005}.
\item If $v=p \mid N$, then $$\varphi_p^{(1)} (x) = \I_{\Z_p}(x_1)\I_{\Z_p}(x_2)\I_{p\Z_p}(x_3)\I_{\Z_p}(x_4).$$
\item If $v=\infty$, then $$\varphi_{\infty}^{(1)}(x) = ({ x}_1+\sqrt{-1}{ x}_2+\sqrt{-1}{x}_3-{x}_4)^{\kappa+1}e^{-\pi\,{\rm tr}({x}{}^t{x})}.$$
\end{itemize}
Define $\varphi^{(2)} = \otimes_v \varphi_v^{(2)} \in S(V^{(2)}(\A))$ by 
$$\varphi_v^{(2)} = \varphi_v^{(0)}\otimes\varphi_v^{(1)}.$$

For $v=p \mid N$, we have
\begin{align}\label{E:equivariant property for JLS}
\omega_p (k,\rho^{(1)}(k_1,k_2)) \varphi_{p}^{(1)}=\varphi_{p}^{(1)}
\end{align}
for $k,k_1,k_2 \in {\bf K}_0(p)$ such that $\det(k)=\det(k_1k_2^{-1})$.
\begin{align}\label{E:equivariant property of test vector case c=1 SK lift}
\omega_p((k,s_p(k)),\rho^{(2)}(k'))\varphi_p^{(2)}=\varphi_p^{(2)}
\end{align}
for $k \in K_0(p) $, $k' \in {\bf K}_0^{(2)}(p).$

For $v=\infty$, we have
\begin{align}\label{E:equivariant property of test vector archimedean case SK lift}
\omega_{{\infty}}(\tilde{k}_{\theta},\rho^{(2)}(k'))\varphi_{\infty} =e^{-\sqrt{-1}(\kappa+1/2)\theta}\det(\alpha+\sqrt{-1}\beta)^{\kappa+1}\varphi_{\infty}
\end{align}
for $\tilde{k}_{\theta} \in \widetilde{{\rm SO}}(2)$, 
$$k' = \bp \alpha & \beta \\ -\beta & \alpha \ep  \in \Sp_4(\R).$$

\subsubsection{Seesaw identity}
Via the homomorphisms $\rho^{(1)}$ and $\rho^{(2)}$, we identify ${\bf g}\times {\bf g}$ and ${\bf F}$ as automorphic forms on ${\rm GSO}(V^{(1)})(\A)$ and ${\rm SO}(V^{(2)})(\A)$, respectively. We have a seesaw identity
\begin{align}\label{E:seesaw 1}
\left\<\theta({V}_+^m{\bf h},\varphi^{(2)}), {\bf g}\times {\bf g} \right\>_{{\rm SO}(V^{(1)})}= \left\langle {V}_+^m{\bf h} , \overline{\theta(\overline{{\bf g}\times {\bf g}},\varphi^{(1)})\cdot\Theta } \right\rangle _{\SL_2}.
\end{align}
On the other hand, we have two explicit theta lifts
\begin{align}\label{E:theta lifts 1 2}
\begin{split}
\theta(\overline{{\bf g}\times {\bf g}},\varphi^{(1)}) &= 2^{\kappa-1}\zeta_{\Q}(2)^{-2}\left [ \SL_2(\Z) : \Gamma_0(N)\right ]^{-1}   \langle { g},{ g}  \rangle \cdot\overline{{\bf g}}^{\sharp},\\
\theta({V}_+^{m}{\bf h},\varphi^{(2)} )&=2^{2m-2}\zeta_{\Q}(2)^{-1}\left [ \SL_2(\Z) : \Gamma_0(N)\right ]^{-1}\cdot {D}_+^m {\bf F}.
\end{split}
\end{align}
The identities in (\ref{E:theta lifts 1 2}) will be proved in Lemmas \ref{L:theta 1} and \ref{L:theta 2} below. By the seesaw identity (\ref{E:seesaw 1}) and (\ref{E:theta lifts 1 2}), we obtain the following

\begin{prop}\label{P:SO(2,2)-period to SL_2 period}
We have 
\begin{align*}
\langle   \Delta_{\kappa'+1}^m{F}|_{\frak{H}\times \frak{H}} ,{ g}\times { g}     \rangle = 2^{\kappa'+2} \zeta_{\Q}(2) \langle { g},{ g}\rangle  \cdot\langle {V}_+^m{\bf h}\cdot {\Theta},{\bf g}^{\sharp}   \rangle_{\SL_2}.
\end{align*}
\end{prop}

The rest of this section are devoted to prove (\ref{E:theta lifts 1 2}).
\subsection{Jacquet-Langlands-Shimizu lifts}\label{SS:JLS}
Let $(V,Q)=(V^{(1)},Q^{(1)})$.

For each place $v$ of $\Q$, let 
\begin{align*}
\mathcal{S}(V(\Q_v)) \longrightarrow \mathcal{S}(V(\Q_v)),\quad
\varphi  \longmapsto \hat{\varphi}
\end{align*}
be the partial Fourier transform defined by
$$\hat{\varphi} \left ( \bp x_1 & x_2 \\ x_3 & x_4 \ep \right ) = \int_{\Q_v}\varphi \left ( \bp x_1 & y_2 \\ x_3 & y_4 \ep \right )\psi_v(-x_4y_2+x_2y_4)dy_2dy_4.$$

For each place $v$ of $\Q$, let $\varphi_v = \varphi_v^{(1)}$ and define $\mathcal{W}_v \in \C$ by
\begin{align*}
\mathcal{W}_v & =\int_{\SL_2(\Q_v)}\hat{\varphi}_v(g) W_{{\bf g},v}(g)dg.
\end{align*}

\begin{lemma}\label{L:3.2}
Let $v=p \mid N$. We have
$$\mathcal{W}_p=(1+p)^{-1}.$$
\end{lemma}

\begin{proof}
Note that $\hat{\varphi}_p=\varphi_p$.

For $k \in \SL_2(\Z_p)$, let 
$$\mathcal{J}(k) = \int_{\Q_p^{\times}}\int_{\Q_p}\hat{\varphi}_p(u(x)t(a)k)W_{{\bf g},p}(u(x)t(a)k)|a|_{\Q_p}^{-2}dxd^{\times}a.$$
It is easy to verify that 
$$u(x)t(a)k \in \bp \Z_p & \Z_p \\ p\Z_p & \Z_p \ep$$
if and only if
$x \in \Z_p$, $a \in \Z_p^{\times}$, and $k \in {\bf K}_0(p)$. 
Therefore, 
$$\mathcal{J}(k) = \left \{ \begin{array}{ll} 1 &\mbox{ if }k \in {\bf K}_0(p),\\
0 & \mbox{ otherwise.} \end{array}\right. $$
We conclude from the above calculation and (\ref{E:equivariant property for JLS}) that
\begin{align*}
\mathcal{W}_p=(1+p)^{-1}\sum_{k \in \SL_2(\Z_p) / K_0(p) }\mathcal{J}(k)=(1+p)^{-1}.
\end{align*}
This completes the proof.
\end{proof}

\begin{lemma}\label{L:theta 1}
We have
$$\theta(\overline{{\bf g}\times {\bf g}},\varphi^{(1)}) = 2^{\kappa-1}\zeta_{\Q}(2)^{-2}\left [ \SL_2(\Z) : \Gamma_0(N)\right ]^{-1}   \langle { g},{ g}  \rangle \cdot\overline{{\bf g}}^{\sharp}.$$
\end{lemma}

\begin{proof}
Following the proof of \cite[Lemma 6.3]{Ichino2005}, with the
equivariant properties of $\varphi^{(1)}$ in (\ref{E:equivariant property for JLS}) and \cite[(6.1) and (6.2)]{Ichino2005}, we have
$$\theta(\overline{{\bf g}\times {\bf g}},\varphi^{(1)}) = 2^{-2}\zeta_{\Q}(2)^{-2}  \langle { g},{ g}  \rangle \prod_{v}\mathcal{W}_v\cdot\overline{{\bf g}}^{\sharp}.$$
By Lemma \ref{L:3.2} and \cite[Proposition 5.2]{IchinoIkeda2008}, 
$$\prod_v\mathcal{W}_v = 2^{\kappa+1}\left [ \SL_2(\Z) : \Gamma_0(N)\right ]^{-1}.$$
This completes the proof.
\end{proof}

\subsection{Saito-Kurokawa lifts}\label{SS:SK}

Let $(V,Q)=(V^{(2)},Q^{(2)})$. We identify $V$ with the space of column vectors $\Q^5$ via
\begin{align*}
\Q^5 &\longrightarrow V^{(2)},\quad \bp x_1 \\ x_2 \\ x_3 \\ x_4 \\ x_5 \ep \longmapsto \bp x_3 & x_2 & 0 & x_1 \\ x_4 & -x_3 & -x_1 & 0 \\ 0 & x_5 & x_3 & x_4 \\ -x_5 & 0 & x_2 & -x_3 \ep.
\end{align*}

Let $${\bf e} = \bp 1 \\ 0 \\ 0 \\ 0 \\ 0 \ep,\quad {\bf e}^* = \bp 0 \\ 0 \\ 0 \\ 0 \\ 1 \ep.$$
Then $V_1=(\Q {\bf e} + \Q {\bf e}^*)^{\perp}= \{x\in V \mbox{ }\vert \mbox{ } x_1=x_5=0 \}$. For each place $v$ of $\Q$, let $\omega_v$ be the Weil representation of $\widetilde{\SL}_2(\Q_v)\times {\rm O}(V)(\Q_v)$ on $\mathcal{S}(V(\Q_v))$, and $\hat{\omega}_v$ be the representation of $\widetilde{\SL}_2(\Q_v)\times {\rm O}(V)(\Q_v)$ on $\mathcal{S}(V_1(\Q_v))\otimes \mathcal{S}(\Q_v^2)$ defined in \S\,\ref{SS:Change of polarizations}. 

Let 
$$B=\begin{pmatrix} b_1 & b_2/2 \\ b_2/2 & b_3 \end{pmatrix} \in {\rm Sym}_2(\Q).$$ Put $\xi = \det(B)$ and 
$$\beta = \bp b_3 \\ b_2/2 \\ -b_1  \ep \in V_1(\Q) .$$
 Assume $\xi \in \bigcap_v \Q_v(\tilde{\pi}_v)\bigcap \Q^{\times}$. For each place $v$ of $\Q$, let $\varphi_v= \varphi_v^{(2)}$ and define a function $\mathcal{W}_{B,v} :  \GSp_4(\Q_v)\rightarrow \C$ by
\begin{align*}
\mathcal{W}_{B,v}(h) =\begin{cases} \displaystyle{\int_{U(\Q_p) \backslash \SL_2(\Q_p)}\hat{\omega}_p(g,\rho^{(2)}(h))\hat{\varphi}_p(\beta;0,1)W_{\xi,p}(g)dg} & \mbox{ if }v=p,\\
\displaystyle{\int_{U(\R) \backslash \SL_2(\R)}\hat{\omega}_\infty(g,\rho^{(2)}(h))\hat{\varphi}_\infty(\beta;0,1)V_+^{m}W_{\xi,\infty}(g)dg} & \mbox{ if }v=\infty.
\end{cases}
\end{align*}
By abuse of notation, we write $h$ for $\rho^{(2)}(h) \in {\rm SO}(V)$ .

\begin{lemma}\label{L:SK p}
Let $v=p\mid N$. We have
$$\mathcal{W}_{B,p}(1)=\left \{ \begin{array}{ll} (1+p)^{-1}\Psi_p(\xi;\alpha_p) & \mbox{ if }b_1,b_2,b_3 \in \Z_p, \\
0 & \mbox{ otherwise}. \end{array} \right .$$ 
\end{lemma}

\begin{proof}
Note that 
\begin{align}\label{E:3.4}
\begin{split}
\hat{\varphi}_p (x;y_1,y_2)&=\I_{\Z_p}(x_1) \I_{\Z_p}(x_2) \I_{\Z_p}(x_3) \I_{p\Z_p}(y_1) \I_{\Z_p}(y_2),\\
\hat{\omega}_p(w,1)\hat{\varphi}_p (x;y_1,y_2)&=\I_{\Z_p}(x_1) \I_{\Z_p}(x_2) \I_{\Z_p}(x_3) \I_{\Z_p}(y_1) \I_{p\Z_p}(y_2).
\end{split}
\end{align}
For $k \in \text{SL}_2(\Z_p)$, let
$$\mathcal{J}(k)=\sum_{n \in \Z}p^{{n}/{2}}\gamma_{\Q_p}(p^n,\psi_p)^{-1}\hat{\omega}_p\left (k,1 \right )\hat{\varphi}_p (p^n\beta;0,p^{-n})W_{\xi,p}(t(p^n)k).$$
By (\ref{E:3.4}), (\ref{L:Whittaker function at 1 for unramified special}), and Lemma \ref{L:Whittaker function at w for Steinberg}
\begin{align*}
\mathcal{J}(1) &= \begin{cases}  
   \sum_{n=0}^{\min({\rm ord}_{\Q_p}(b_i))}p^{{n}/{2}}\Psi_p(p^{-2n}\xi;\alpha_p) & \text{ if }\min({\rm ord}_{\Q_p}(b_i)) \geq 0, \\
0 & \text{ otherwise}. \end{cases}
  \\
\mathcal{J}(w) &= \begin{cases}     -p^{-1}\sum_{n=1}^{\min({\rm ord}_{\Q_p}(b_i))}p^{{n}/{2}}\Psi_p(p^{-2n}\xi;\alpha_p) & \text{ if }\min({\rm ord}_{\Q_p}(b_i)) \geq 1, \\
0 & \text{ otherwise}. \end{cases}
\end{align*}

Let $k = \bp 1 & 0 \\ a & 1 \ep$ with $a \in \Z_p^{\times}$. Note that 
$$k = u(a^{-1})w\bp -a & -1 \\ 0 & -a^{-1} \ep.$$
Thus $\mathcal{J}(k)=\mathcal{J}(w)$ by (\ref{E:equivariant property of test vector case c=1 SK lift}). 
Therefore, we conclude from the above calculation and (\ref{E:equivariant property of test vector case c=1 SK lift}) that
\begin{align*}
\mathcal{W}_{B,p}(1)&=(1+p)^{-1}\sum_{\text{SL}_2(\Z_p) / K_0(p)  } \mathcal{J}(k)\\
&=(1+p)^{-1} \left [ \mathcal{J}(1)+p\mathcal{J}(w) \right ]\\
&=\left \{ \begin{array}{ll} (1+p)^{-1}\Psi_p(\xi;\alpha_p) & \mbox{ if }b_1,b_2,b_3 \in \Z_p, \\
0 & \mbox{ otherwise}. \end{array} \right .
\end{align*}
This completes the proof.
\end{proof}

\begin{lemma}\label{L:archimedean local integral for Shimura lifts and Saito-Kurokawa lifts}
Let $r \in \R^{\times}$, $m,n \in \Z_{\geq 0}$. Put
\begin{align*}
J(m,n;r)&=\int_{0}^{\infty}a^{n-2m-2}H_n(\sqrt{\pi}(ra+a^{-1}))e^{-\pi(ra+a^{-1})^2}da,\\
I(m,n;r)&=\sum_{j=0}^{m}\frac{(2m-j)!}{j!(m-j)!}(4\pi)^{-m+j}\sum_{i=0}^{j}\frac{n!(-4\pi)^{-i}}{(n-i)!}{j \choose i}r^{j-i}.
\end{align*}

If $r>0$, then
$$J(m,n;r)=2^{2n-1}\pi^{n/2}e^{-4\pi r}I(m,n;r).$$

If $r<0$ and $n>m$, then
$$J(m,n;r)=0.$$
\end{lemma}
\begin{proof}
In a small neighborhood of $x=0$,
\begin{align*}
&\sum_{n=0}^{\infty}\frac{1}{n!}(-\sqrt{\pi}x)^nJ(m,n;r)\\&=\int_{0}^{\infty}a^{-2m-2}e^{-\pi(ax+ra+a^{-1})^2}da\\
%&=e^{-2\pi(r+x)}\int_{0}^{\infty}a^{-2m-2}e^{-\pi(x+r)^2a^2-\pi a^{-2}}da\\
%&=2^{-1}e^{-2\pi(r+x)}\int_{0}^{\infty}t^{-m-3/2}e^{-\pi(r+x)^2t-\pi t^{-1}}dt\\
%&=2^{-1}e^{-2\pi(r+x)}|r+x|_{\R}^{m+1/2}\int_{0}^{\infty}t^{-m-3/2}e^{-\pi|r+x|_{\R}(t+t^{-1})}dt\\
&=e^{-2\pi (r+x)}|r+x|_{\R}^{m+1/2}K_{m+1/2}(2\pi|r+x|_{\R})\\
&=2^{-1}e^{-2\pi (r+x)-2\pi|r+x|_{\R}}|r+x|_{\R}^{m}\sum_{j=0}^{m}\frac{(m+j)!}{j! (m-j)!}(4\pi)^{-j}|r+x|_{\R}^{-j},
\end{align*}
where the last equality follows from \cite[8.468]{Table2000}.
\\If $r<0$, then $r+x<0$ in a neighborhood of $x=0$ and we have
$$\sum_{n=0}^{\infty}\frac{1}{n!}(-\sqrt{\pi}x)^nJ(m,n;r)=2^{-1}\sum_{j=0}^{m}\frac{(m+j)!}{j! (m-j)!}(4\pi)^{-j}(-1)^{m-j}(r+x)^{m-j}.$$
\\If $r>0$, then $r+x>0$ in a neighborhood of $x=0$ and we have
$$\sum_{n=0}^{\infty}\frac{1}{n!}(-\sqrt{\pi}x)^nJ(m,n;r)=2^{-1}e^{-4\pi(r+x)}\sum_{j=0}^{m}\frac{(2m-j)!}{j!(m-j)!}(4\pi)^{-m+j}(r+x)^j.$$
This completes the proof.
\end{proof}

\begin{lemma}\label{L:SK infinity}
Let $v=\infty$. Let $A \in \GL_2^{+}(\R)$ and $X \in {\rm Sym}_2(\R)$. Put $Y=A{{}^{t}A}$ and $Z=X+\sqrt{-1}Y$. 

If $B>0$, then
\begin{align*}
&\mathcal{W}_{B,\infty}\left (    \left (  \begin{array}{cc} {\bf 1}_2&X \\ {\bf 0}_2 &{\bf 1}_2   \end{array} \right )  \left (  \begin{array}{cc} A&{\bf 0}_2 \\ {\bf 0}_2 &{{}^{t}A^{-1}}  \end{array}  \right ) \right )\\
&=    2^{\kappa+1}\det(Y)^{(\kappa+1)/2}e^{2\pi \sqrt{-1} {\rm tr}(BZ)}\\
&\times \sum_{j=0}^{m}(-4\pi)^{j-m}\frac{\Gamma(\kappa-m+1/2)}{\Gamma(\kappa-2m+1/2+j)} {m \choose j}\det(B)^j\det(Y)^{j-m} \\
&\mbox{ }\mbox{ }\mbox{ }\times  \sum_{i=0}^{m-j}\frac{(2m-2j-i)!}{i! (m-j-i)!}(4\pi )^{i+j-m}\sum_{l=0}^{i}\frac{(\kappa+1)! (-4\pi)^{-l}}{(\kappa+1-l)!}{i \choose l}{\rm tr}(BY)^{i-l}.
\end{align*}

If $B<0$, then
$$\mathcal{W}_{B,\infty}\left (    \left (  \begin{array}{cc} {\bf 1}_2&X \\ {\bf 0}_2 &{\bf 1}_2   \end{array} \right )  \left (  \begin{array}{cc} A&{\bf 0}_2 \\ {\bf 0}_2 &{{}^{t}A^{-1}}  \end{array}  \right ) \right )=0.$$
\end{lemma}

\begin{proof}
As in the proof of \cite[Lemma 7.6]{Ichino2005}, we have
\begin{align*}
&\hat{w}_\infty\left (t(a)\tilde{k}_{\theta}, \left (  \begin{array}{cc} {\bf 1}_2&X \\ {\bf 0}_2 &{\bf 1}_2   \end{array} \right )  \left (  \begin{array}{cc} A&{\bf 0}_2 \\ {\bf 0}_2 &{{}^{t}A^{-1}}  \end{array}\right )  \right )\hat{\varphi}_{\infty} (\beta;0,1)\\
&=e^{-\sqrt{-1}(\kappa+1/2)\theta}(2\sqrt{\pi})^{-\kappa-1}\det (A)\gamma_{\R}(a,\psi_{\infty})^{-1}|a|_{\R}^{3/2}H_{\kappa+1}(\sqrt{\pi}(\det (A)^{-1}{\rm tr}(BY)a+\det (A)a^{-1}))\\
&\times e^{2\pi \sqrt{-1}{\rm tr}(BX)}e^{2\pi{\rm tr}(BY)}e^{-\pi(\det (A)^{-1}{\rm tr}(BY)a+\det (A)a^{-1}))^2}e^{2\pi\xi a^2},
\end{align*}
for $a \in \R_{>0}$, and $\theta \in \R / 4\pi \Z$.
Therefore, by (\ref{E:metaplectic weight raising Whittaker  function archimedean case}) and Lemma \ref{L:archimedean local integral for Shimura lifts and Saito-Kurokawa lifts}
\begin{align*}
&\mathcal{W}_{B,\infty}\left (    \left (  \begin{array}{cc} {\bf 1}_2&X \\ {\bf 0}_2 &{\bf 1}_2   \end{array} \right )  \left (  \begin{array}{cc} A&{\bf 0}_2 \\ {\bf 0}_2 &{{}^{t}A^{-1}}  \end{array}  \right ) \right )\\
&=2 (2\sqrt{\pi})^{-\kappa-1}e^{2\pi \sqrt{-1}{\rm tr}(BX)}e^{2\pi{\rm tr}(BY)}\det (Y)^{(\kappa+1)/2}\\
&\times \sum_{j=0}^{m}(-4\pi)^{j-m}\frac{\Gamma(\kappa'+1/2+m)}{\Gamma(\kappa'+1/2+j)}{m \choose j}\xi^j\det (Y)^{j-m} J(m-j,\kappa+1;{\rm tr}(BY))\\
&=2^{\kappa+1}\det (Y)^{(\kappa+1)/2} e^{2\pi \sqrt{-1}{\rm tr}(BZ)}\\
&\times \sum_{j=0}^{m}(-4\pi)^{j-m}\frac{\Gamma(\kappa'+1/2+m)}{\Gamma(\kappa'+1/2+j)}{m \choose j}\xi^j\det (Y)^{j-m} I(m-j,\kappa+1;{\rm tr}(BY)).
\end{align*} 
Note that ${\rm tr}(BY)>0$ if and only if $B>0$. This completes the proof.
\end{proof}

\begin{lemma}\label{L:theta 2}
We have
$$\theta({V}_+^{m}{\bf h},\varphi^{(2)} )=2^{2m-2}\zeta_{\Q}(2)^{-1}\left [ \SL_2(\Z) : \Gamma_0(N)\right ]^{-1}\cdot {D}_+^m {\bf F}.$$
\end{lemma}

\begin{proof}
For an automorphic form $\tilde{\bf F}$ on $\GSp_4(\A)$ and $B \in {\rm Sym}_2(\Q)$, the $B$-th Fourier coefficient of $\tilde{\bf F}$ is defined by
$${W}_{\tilde{\bf F},B}(h)=\int_{{\rm Sym}_2(\Q)\backslash{\rm Sym}_2(\A)} \tilde{\bf F}\left ( \begin{pmatrix}{\bf 1}_2 & X \\ {\bf 0}_2 & {\bf 1}_2 \end{pmatrix} h\right ) \overline{\psi({\rm tr}(BX))}dX.$$

By (\ref{E:equivariant property of test vector case c=1 SK lift}) and (\ref{E:equivariant property of test vector archimedean case SK lift}), both $D_+^m\bf F$ and $\theta({V}_+^{m}{\bf h},\varphi^{(2)})$ are right ${\bf K}_0^{(2)}(N\widehat{\Z})$-invariant, and 
\begin{align*}
{D_+^m\bf F}(hk) &= \det(\alpha+\sqrt{-1}\beta)^{\kappa+1}D_+^m{\bf F}(h),\\
\theta({V}_+^{m}{\bf h},\varphi^{(2)})(hk) &= \det(\alpha+\sqrt{-1}\beta)^{\kappa+1}\theta({V}_+^{m}{\bf h},\varphi^{(2)})(h)
\end{align*}
for $h \in \GSp_4(\A)$ and 
$$k=\bp \alpha & \beta \\ -\beta & \alpha \ep \in \Sp_4(\R).$$
Therefore, it suffices to show that
$${W}_{\theta({V}_+^{m}{\bf h},\varphi^{(2)}),B}(h_{\infty})  =2^{2m-2}\zeta_{\Q}(2)^{-1}\left [ \SL_2(\Z) : \Gamma_0(N)\right ]^{-1}\cdot{W}_{D_+^m{\bf F},B}(h_{\infty})$$
for $B \in {\rm Sym}_2(\Q)$, $h_{\infty}=\left (  \begin{array}{cc} {\bf 1}_2&X \\ {\bf 0}_2 &{\bf 1}_2   \end{array} \right )  \left (  \begin{array}{cc} A&{\bf 0}_2 \\ {\bf 0}_2 &{{}^{t}A^{-1}}  \end{array}  \right ) \in \Sp_4(\R)$ with $X \in {\rm Sym}_2(\R)$ and $A \in \GL_2^+(\R)$.
Put $B= \bp b_1 & b_2/2 \\ b_2/2 & b_3\ep \in {\rm Sym}_2(\Q)$ and $\xi = \det(B) \in \Q$. By \cite[Lemma 4.2]{Ichino2005}, 
\begin{align*}
W_{\theta({V}_+^{m}{\bf h},\varphi^{(2)}),B} = \begin{cases}
c_{{h}}(\frak{d}_{\xi})\frak{f}_{\xi}^{\kappa'-1/2}\zeta_{\Q}(2)^{-1}\prod_v \mathcal{W}_{B,v} & \mbox{ \rm if }\xi \in\bigcap_v \Q_v(\tilde{\pi}_v)\bigcap \Q^{\times} ,  \\ 0 & \mbox{ \rm otherwise}.
\end{cases}
\end{align*}
If $\xi \in \bigcap_v \Q_v(\tilde{\pi}_v)\bigcap \Q^{\times}$, and either $B<0$ or $b_i \notin \Z$ for some $i$, then
$$\mathcal{W}_{B,\infty}(h_{\infty})\prod_{p}\mathcal{W}_{B,p}(1)=0$$ 
by Lemmas \ref{L:SK p}, \ref{L:SK infinity}, and \cite[\S\,7.3 and Lemma 7.3]{Ichino2005}. Therefore, we may assume $\xi \in \bigcap_v \Q_v(\tilde{\pi}_v)\bigcap \Q^{\times}$, $B>0$, and $b_1,b_2,b_3 \in \Z$. In this case, by Lemmas \ref{L:SK p} and \ref{L:SK infinity}, and \cite[\S\,7.3 and Lemma 7.3]{Ichino2005}, and note that
$$c_{h}\left(\frac{4\xi}{d^2}\right)=c_{h}(\frak{d}_{\xi})\left(\frac{2\frak{f}_{\xi}}{d}\right)^{\kappa'-1/2}\prod_{p}\Psi_p\left(\frac{4\xi}{d^2};\alpha_p\right)$$
for $d\in \Q_{>0}$, we have
\begin{align*}
W_{\theta({V}_+^{m}{\bf h},\varphi^{(2)}),B}(h_{\infty}) &= 2^{-7/2} \zeta_{\Q}(2)^{-1}W_{B,\infty}(h_{\infty})\left[ \SL_2(\Z):\Gamma_0(N)\right ]^{-1}\\
&\times c_h(\frak{d}_{\xi})\frak{f}_{\xi}^{\kappa'-1/2}\prod_{p \mid N}\Psi_p(4\xi;\alpha_p)\prod_{p \nmid N}\sum_{n=0}^{\min({\rm ord}_{\Q_p}(b_i))} p^{n/2}\Psi_p\left (\frac{4\xi}{p^{2n}};\alpha_p \right )\\
&=2^{-\kappa'-3}\zeta_{\Q}(2)^{-1}\left[ \SL_2(\Z):\Gamma_0(N)\right ]^{-1}W_{B,\infty}(h_{\infty})A(B)\\
&= 2^{2m-2}\zeta_{\Q}(2)^{-1}\left [ \SL_2(\Z) : \Gamma_0(N)\right ]^{-1}\det(Y)^{(\kappa+1)/2}A(B)e^{2\pi \sqrt{-1}{\rm tr}(BZ)}\\
&\times \sum_{j=0}^{m}(-4\pi)^{j-m}\frac{\Gamma(\kappa-m+1/2)}{\Gamma(\kappa-2m+1/2+j)} {m \choose j}\det(B)^j\det(Y)^{j-m} \\
&\mbox{ }\mbox{ }\mbox{ }\times  \sum_{i=0}^{m-j}\frac{(2m-2j-i)!}{i! (m-j-i)!}(4\pi )^{i+j-m}\sum_{l=0}^{i}\frac{(\kappa+1)! (-4\pi)^{-l}}{(\kappa+1-l)!}{i \choose l}{\rm tr}(BY)^{i-l}.
\end{align*}
Comparing with the Fourier coefficients of $\Delta_{\kappa'+1}^{m}{F}$ in (\ref{E:Fourier expansion of nearly holomorphic Siegel modular form}), we obtain the assertion.
\end{proof}

\section{Seesaw identity (II)}\label{S:seesaw II}

We keep the notation of \S\,\ref{S:Automorphic forms and $L$-functions}. Let $(V^{(0)},Q^{(0)})$ be the quadratic space defined in \S\,\ref{SS:Theta lifts}. 
We make the following assumptions on the imaginary quadratic field $\K$ in \S\,\ref{SS:Automorphic forms on GL_2 over K}:
\begin{itemize}
\item $(D,N)=1$.
\item $-D\equiv 1\mbox{ mod }8$.
\item $(p,-\tau_pD)_{\Q_p} = -1 \mbox{ for }p\mid N$.
\item $\Lambda\left(\kappa',{f} \otimes \chi_{-D} \right) \neq 0.$
\end{itemize}
Under these assumptions, we have $D \in \bigcap_v \Q_v(\tilde{\pi}_v)$ and $c_h(D) \neq 0$. 
The existence of such fundamental discriminant $-D$ is guaranteed by the nonvanishing theorems in \cite{FriedbergHoffstein1995} and \cite{Wald1991}.
The aim of this section is to prove an explicit seesaw identity in Proposition \ref{P:SL_2-period to SO(2,1)-period}.
\subsection{Main results}
\subsubsection{Setting}
Let $(V^{(3)},Q^{(3)})$ and $(V^{(4)},Q^{(4)})$ be quadratic spaces over $\Q$ defined by
\begin{align*}
V^{(3)} & = \left \{ \left .\bp x_1 & x_2 \\ x_3 & -x_1 \ep \mbox{ }\right \vert \mbox{ }x_1,x_2,x_3 \in \mathbb{G}_a \right \},\quad Q^{(3)}[x] = -D\det(x),\\
V^{(4)} & = \left \{ \left . \begin{pmatrix} {x}_1 & \delta { x}_2 \\ \delta { x}_3 & { x}_1^{\tau}\end{pmatrix}   \mbox{ }\right \vert  \mbox{ }{x}_1 \in {\rm R}_{\K/\Q}\mathbb{G}_a,\,{ x}_2,{x}_3 \in \mathbb{G}_a      \right \},\quad Q^{(4)}[x] = \det(x).
\end{align*}
For $x = \bp x_1 & x_2 \\ x_3 & x_4\ep \in {\rm M}_2(\K)$, let $x^*= \bp x_4^{\tau} & -x_2^{\tau} \\ -x_3^{\tau} & x_1^{\tau} \ep$.
We have an isomorphism of quadratic spaces
\begin{align*}
V^{(0)} \oplus V^{(3)} &\longrightarrow V^{(4)}\\
\left ( x, \bp x_1 & x_2 \\ x_3 & -x_1 \ep \right ) &\longmapsto \bp x+\delta x_1  & \delta x_2 \\ \delta x_3 & x-\delta x_1\ep.
\end{align*}
Recall we have exact sequences 
\begin{align*}
&1 \longrightarrow {\mathbb G}_m \stackrel{\iota^{(3)} }{\longrightarrow} \GL_2 \stackrel{\rho^{(3)} }{\longrightarrow} {\rm SO}(V^{(3)} ) \longrightarrow 1,\\
&1 \longrightarrow {\rm{R}}_{\K/\Q} {\mathbb G}_m \stackrel{\iota^{(4)}}{\longrightarrow} {\mathbb G}_m\times {\rm{R}}_{\K/\Q} \GL_2 \stackrel{\rho^{(4)}}{\longrightarrow}{\rm GSO}(V^{(4)})\longrightarrow 1,
\end{align*}
where 
\begin{align*}
\iota^{(3)}(a) &= a{\bf 1}_2,\quad \iota^{(4)}(b) = ({\rm N}_{\K/\Q}(b)^{-1},b{\bf 1}_2),\\
\rho^{(3)}(h)x &= hxh^{-1},\quad \rho^{(4)}(z,h)y = zhyh^*,
\end{align*}
for $a \in \mathbb{G}_m$, $b \in {\rm{R}}_{\K/\Q} {\mathbb G}_m$, $h \in \GL_2$, $(z,h) \in {\mathbb G}_m\times {\rm{R}}_{\K/\Q} \GL_2$, and $x \in V^{(3)}$, $y \in V^{(4)}$. Note that the natural embedding ${\rm SO}(V^{(3)}) \hookrightarrow {\rm SO}(V^{(4)})$ is compatible with the embedding
\begin{align*}
\GL_2 &\longrightarrow  {\mathbb G}_m\times {\rm{R}}_{\K/\Q} \GL_2 \\
h &\longmapsto (\det(h)^{-1},h).
\end{align*}
\subsubsection{Schwartz functions}
Define $\varphi^{(3)} = \otimes_v \varphi_v^{(3)} \in S(V^{(3)}(\A))$ as follows:
\begin{itemize}
\item If $v=p \nmid N$, then $\varphi_v^{(3)}$ is the function $\varphi_v^{(6)}$ defined in \cite[\S\,9.2]{Ichino2005}.
\item If $v=p$ with $p \mid N$, then  
$$\varphi_p^{(3)} ( x  )= \I_{\Z_p}({ x}_1) \I_{\Z_p}({ x}_2) \I_{p\Z_p}({ x}_3).$$
\item If $v=\infty$, then
$$\varphi_{\infty}^{(3)}(x) = (2x_1-\sqrt{-1}x_2-\sqrt{-1}x_3)^{\kappa}e^{-\pi D \,{\rm tr}(x{}^tx)}.$$
\end{itemize}
Define $\varphi^{(4)} = \otimes_v \varphi_v^{(4)} \in S(V^{(4)}(\A))$ by 
$$\varphi_v^{(4)} = \varphi_v^{(0)}\otimes\varphi_v^{(3)}.$$

For $v=p \mid N$, we have
\begin{align}\label{E:equivariant property of test vector case c=1}
\omega_p((k,s_p(k)),\rho^{(3)}(k'))\varphi_p^{(3)} = \varphi_p^{(3)}
\end{align}
for $k \in K_0(p)$ and $k' \in {\bf K}_0(p)$.
\begin{align}\label{E:equivariant property of test vector case v=p base change lift}
\omega_{p}(k,\rho^{(4)}(k')) \varphi_p^{(4)} = \varphi_p^{(4)}
\end{align}
for $k \in {\bf K}_0(p), k' \in \mathbb{K}_0(p)$ such that $\det(k)={\rm N}_{\K_p/ \Q_p}(\det (k'))$.

For $v=\infty$, we have
\begin{align}\label{E:equivariant property of test vector archimedean case }
\omega_{\infty}(\tilde{k}_{\theta},\rho^{(3)}(k_{\theta'}))\varphi_{\infty}^{(3)} =e^{(\kappa+1/2)\sqrt{-1}\theta}e^{-2\kappa \sqrt{-1}\theta'}\varphi_{\infty}^{(3)}
\end{align}
for $\tilde{k}_{\theta} \in \widetilde{{\rm SO}}(2)$, $k_{\theta'} \in {\rm SO}(2)$. 
\subsubsection{Seesaw identity}

Via the homomorphism $\rho^{(3)}$ (resp.\,$\rho^{(4)}$), we identify ${\bf f}$ and ${\bf f}\otimes \chi_{-D}$ (resp.\,${\bf g}_\K^{\sharp}$ and $\chi_{-D}\times {\bf g}_\K^{\sharp}$) as automorphic form on ${\rm SO}(V^{(3)})(\A)$ (resp.\,${\rm GSO}(V^{(4)})(\A)$). We have a seesaw identity
\begin{align}\label{E:seesaw 2}
\left\< \theta({V}_+^{2m}{\bf f}\otimes \chi_{-D},\varphi^{(3)})\cdot\Theta,{\bf g}^{\sharp}\right\>_{\SL_2} = \left\<{V}_+^{2m}{\bf f}\otimes \chi_{-D},\overline{\theta(\overline{\bf g}^{\sharp},\varphi^{(4)})} \right\>_{{\rm SO}(V^{(3)})}.
\end{align}
On the other hand, we have two explicit theta lifts
\begin{align}\label{E:theta lifts 3 4}
\begin{split}
\theta({V}_+^{2m}{\bf f}\otimes \chi_{-D}, \varphi^{(3)} )&=-2^{-2m-1}\pi^{-2m}(\sqrt{-1})D^{-\kappa+m+1/2}\overline{c_h(D)}\zeta_{\Q}(2)^{-1}\\
&\times \left ( \SL_2(\Z):\Gamma_0(N)\right )^{-1} \left(\frac{\Gamma(\kappa-m)\Gamma(2m+1)}{\Gamma(\kappa-2m)\Gamma(m+1)}\right )\< f,f\> \<h,h\>^{-1}\cdot {V}_+^{m}{\bf h},\\
\theta(\overline{{\bf g}}^{\sharp},\varphi^{(4)})&=2^{-1}(\sqrt{-1})D^{(-\kappa+1)/2}\zeta_{\Q}(2)^{-1}\left( \SL_2(\Z):\Gamma_0(N)\right )^{-1} \cdot(\chi_{-D}\times{{\bf g}_{\K}^{\sharp}}).
\end{split}
\end{align}
The identities in (\ref{E:theta lifts 3 4}) will be proved in Lemmas \ref{L:theta 3} and \ref{L:theta 4} below. By the seesaw identity (\ref{E:seesaw 2}) and (\ref{E:theta lifts 3 4}), we obtain the following

\begin{prop}\label{P:SL_2-period to SO(2,1)-period}
We have 
$$\left\langle {V}_+^m{\bf h}\cdot \Theta, {\bf g}^{\sharp} \right\rangle_{\SL_2}=-(2\pi)^{2m}D^{\kappa'/2}\overline{c_{{h}}(D)}^{-1}\left (\frac{\Gamma(\kappa-2m)\Gamma(m+1)}{\Gamma(\kappa-m)\Gamma(2m+1)}\right ) \langle  { f},{ f} \rangle^{-1}\langle { h},{ h} \rangle \cdot \left\< {V}_+^{2m}{\bf f}, \overline{\bf g}_{\K}^{\sharp} \right\>_{{\rm SO}(V^{(3)})}.$$
\end{prop}

The rest of this section are devoted to prove (\ref{E:theta lifts 3 4}) in Lemmas \ref{L:theta 3} and \ref{L:theta 4}.

\subsection{Shintani lifts}\label{SS:Shintani}

Let $(V,Q) = (V^{(3)},Q^{(3)}).$

Let $${\bf e} = \bp 0 & 1 \\ 0 & 0 \ep,\quad {\bf e}^* = \bp 0 & 0 \\ 1 & 0 \ep.$$
Then $V_1=(\Q {\bf e} + \Q {\bf e}^*)^{\perp}= \{x\in V \mbox{ }\vert \mbox{ } x_2=x_3=0 \}$. For each place $v$ of $\Q$, let $\omega_v$ be the Weil representation of $\widetilde{\SL}_2(\Q_v)\times {\rm O}(V)(\Q_v)$ on $\mathcal{S}(V(\Q_v))$, and $\hat{\omega}_v$ be the representation of $\widetilde{\SL}_2(\Q_v)\times {\rm O}(V)(\Q_v)$ on $\mathcal{S}(V_1(\Q_v))\otimes \mathcal{S}(\Q_v^2)$ defined in \S\,\ref{SS:Change of polarizations}.

For each place $v$ of $\Q$, let $\varphi_v=\varphi_v^{(3)}$ and define a function $\mathcal{W}_v : \GL_2(\Q_v)\rightarrow \C$ by
\begin{align*}
\mathcal{W}_v(h)=\begin{cases}
\displaystyle{\int_{U(\Q_v) \backslash {\SL_2(\Q_p)} }\hat{\omega}_p(g,\rho^{(3)}(h))\hat{\varphi}_p(2^{-1}D^{-1};0,1)\overline{W_{D,p}(t(2^{-1}D^{-1})g)}dg} & \mbox{ if }v=p,\\
\displaystyle{\int_{U(\R) \backslash {\SL_2(\R)} }\hat{\omega}_\infty(g,\rho^{(3)}(h))\hat{\varphi}_\infty(2^{-1}D^{-1};0,1)\overline{V_+^mW_{D,\infty}(t(2^{-1}D^{-1})g)}dg} & \mbox{ if }v=\infty.
\end{cases}
\end{align*}
By abuse of notation, we write $h$ for $\rho^{(3)}(h) \in {\rm SO}(V)$.

\begin{lemma}\label{lemma 4.2}
Let $v=p \mid N$. We have
$$\mathcal{W}_p(1) = (1+p)^{-1}.$$
\end{lemma}

\begin{proof}
Note that 
\begin{align}\label{E:4.33}
\begin{split}
\hat{\varphi}_p(x_1;0,y_2) &= \I_{\Z_p}(x_1) \I_{\Z_p}(y_2),\\
\hat{\omega}_p(w,1)\hat{\varphi}_p(x_1;0,y_2) &= \hat{\omega}_p\left(\bp 1 & 0 \\ a & 1\ep,1 \right)\hat{\varphi}_p(x_1;0,y_2)=  \I_{\Z_p}(x_1) \I_{p\Z_p}(y_2) 
\end{split}
\end{align}
for $a \in \Z_p^{\times}$. For $k \in \SL_2(\Z_p)$, let
$$\mathcal{J}(k) = \sum_{n\in \Z}p^{3n/2}(p^n,-D)_{\Q_p}\gamma_{\Q_p}(p^n,\psi_p)\hat{\omega}_p(k,1)\hat{\varphi}_p^{ }(2^{-1}D^{-1}p^n;0,p^{-n})\overline{W_{D,p}(t(2^{-1}D^{-1})t(p^n)k)}.$$
By (\ref{E:4.33}) and (\ref{E:equivariant property of test vector case c=1}), $\hat{\omega}_p(k,1)\hat{\varphi}_p^{ }(2^{-1}D^{-1}p^n;0,p^{-n}) \neq 0$ if and only if $n=0$ and $k \in K_0(p)$. Note that $W_{D,p}(1)=1$ by (\ref{L:Whittaker function at 1 for unramified special}). Therefore, 
$$\mathcal{J}(k) = \left \{ \begin{array}{ll} 1 &\mbox{ if }k \in {K}_0(p),\\
0 & \mbox{ otherwise.} \end{array}\right. $$
We conclude from the above calculation and (\ref{E:equivariant property of test vector case c=1}) that
\begin{align*}
\mathcal{W}_p(1)=(1+p)^{-1}\sum_{k \in \SL_2(\Z_p) / K_0(p) }\mathcal{J}(k)=(1+p)^{-1}.
\end{align*}
This completes the proof.
\end{proof}

\begin{lemma}\label{lemma 4.3}
We have
\begin{align*}
\<V_+^{2m}{\bf f},V_+^{2m}{\bf f}\>_{{\rm SO}(V)}\< V_+^m{\bf h},V_+^m{\bf h}\>_{\SL_2}^{-1}= 2(2\pi)^{-2m}\left(\frac{\Gamma(\kappa-m)\Gamma(2m+1)}{\Gamma(\kappa-2m)\Gamma(m+1)}\right )\< f,f\> \<h,h\>^{-1}.
\end{align*}
\end{lemma}

\begin{proof}
Let $V_-=\overline{V_+}$. By \cite[Lemma 5.6]{JLbook} and \cite[p.\,22]{Wald1980}, we have
\begin{align*}
V_-^{2m}V_+^{2m}{\bf f} &= (4\pi)^{-4m}\frac{\Gamma(2\kappa'+2m)\Gamma(2m+1)}{\Gamma(2\kappa')} \cdot{\bf f},\\
V_-^mV_+^m{\bf h} & = (-1)^m(4\pi)^{-2m}\frac{\Gamma(\kappa'+m+1/2)\Gamma(m+1)}{\Gamma(\kappa'+1/2)} \cdot{\bf h}.
\end{align*}
Note that
$$\< {\bf f}, {\bf f}\>_{{\rm SO}(V)}=\zeta_{\Q}(2)^{-1}\<f,f\>,\quad  \< {\bf h}, {\bf h}\>_{\widetilde{\SL}_2}=2^{-1}\zeta_{\Q}(2)^{-1}\<h,h\>.$$
This completes the proof.
\end{proof}

\begin{lemma}\label{lemma 4.4}
Let $v=\infty$. We have
\begin{align*}
\mathcal{W}_{\infty}(1)=2^{-1/2}D^{-\kappa+m-1}e^{-2\pi}.
\end{align*}
\end{lemma}

\begin{proof}
Note that
\begin{align*}
\hat{\varphi}_{\infty} (x_1;y_1,y_2)=(2\sqrt{\pi D})^{-\kappa}H_{\kappa}(\sqrt{\pi D}(2x_1-\sqrt{-1}y_1+y_2))e^{-\pi D(2x_1^2+y_1^2+y_2^2)}
\end{align*}
by \cite[Lemma 7.4]{Ichino2005}.

By (\ref{E:equivariant property of test vector archimedean case }) and the results of Waldspurger in \cite[IV]{Wald1991} on local Shimura lifts, we have
\begin{align*}
\mathcal{W}_{\infty}=C\cdot \overline{ {V}_+^{2m}W_{{\bf f},\infty}\otimes\chi_{-D,\infty}}
\end{align*}
for some constant $C \in \C$. By (\ref{E:metaplectic weight raising Whittaker function archimedean case}), (\ref{E:equivariant property of test vector archimedean case }), and Lemma \ref{L:archimedean local integral for Shimura lifts and Saito-Kurokawa lifts}, for $y \in \R_{>0}$ we have
\begin{align*}
&\mathcal{W}_{\infty}(a(y))\\
&=2y \int_{0}^{\infty}a^{-5/2}\hat{\varphi}_{\infty} (2^{-1}D^{-1}a;0,ya^{-1})\overline{{V}_+^{m}W_{D,\infty}({t}(2^{-1}D^{-1}a))}da\\
&=2y(2\sqrt{\pi D})^{-\kappa} \sum_{j=0}^{m}(-4\pi)^{j-m}\frac{\Gamma(\kappa-m+1/2)}{\Gamma(\kappa-2m+1/2+j)}{m \choose j}(2^{-1}D^{-1})^{\kappa-2m+2j+1/2}D^{j}\\
&\times \int_{0}^{\infty}a^{\kappa-2m+2j-2}H_{\kappa}(\sqrt{\pi D}(D^{-1}a+ya^{-1}))e^{-\pi(D^{-1}a^2+Dy^2a^{-2})}da\\
&=2^{-1/2}D^{-\kappa+m-1}e^{-2\pi y}\sum_{j=0}^{m}(-\pi)^{j-m}y^{\kappa-2m+2j}\frac{\Gamma(\kappa-m+1/2)}{\Gamma(\kappa-2m+1/2+j)}{m \choose j}I(m-j,\kappa,y)\\
&=2^{-1/2}D^{-\kappa+m-1}e^{-2\pi y}\sum_{j=0}^{m}(-\pi)^{j-m}y^{\kappa-2m+2j}\frac{\Gamma(\kappa-m+1/2)}{\Gamma(\kappa-2m+1/2+j)}{m\choose j}\\
&\times \sum_{i=0}^{m-j}\frac{(2m-2j-i)!}{i! (m-j-i)!}(4\pi )^{i+j-m}\sum_{l=0}^{i}\frac{\kappa! (-4\pi)^{-l}}{(\kappa-l)!}{i \choose l}y^{i-l}.
\end{align*}
Both $\mathcal{W}_{\infty}(a(y)) $ and $\overline{{V}_+^{2m}W_{{\bf f},\infty}\otimes \chi_{-D,\infty}}\left(a(y) \right )$ are product of $e^{-2\pi y}$ and polynomials in $y$. Comparing the coefficients for $y^{\kappa}$ with (\ref{E:Whittaker real}), we conclude that $C=-2^{-1/2}D^{-\kappa+m-1}e^{-2\pi}$. This completes the proof.

\end{proof}

\begin{lemma}\label{L:theta 3}
We have
\begin{align*}
\theta({V}_+^{2m}{\bf f}\otimes \chi_{-D}, \varphi^{(3)} )&=-2^{-2m-1}\pi^{-2m}(\sqrt{-1})D^{-\kappa+m+1/2}\overline{c_h(D)}\zeta_{\Q}(2)^{-1}\\
&\times \left ( \SL_2(\Z):\Gamma_0(N)\right )^{-1} \left(\frac{\Gamma(\kappa-m)\Gamma(2m+1)}{\Gamma(\kappa-2m)\Gamma(m+1)}\right )\< f,f\> \<h,h\>^{-1}\cdot {V}_+^{m}{\bf h}.
\end{align*}
\end{lemma}

\begin{proof}
Following the proof of \cite[Lemma 9.1]{Ichino2005}, with the equivariant properties of $\varphi^{(3)}$ in (\ref{E:equivariant property of test vector case c=1}), (\ref{E:equivariant property of test vector archimedean case }) and \cite[(9.1)-(9.4)]{Ichino2005}, we have
\begin{align*}
\theta({V}_+^{2m}{\bf f}\otimes \chi_{-D}, \varphi^{(3)} ) = \overline{c_h(D)}\zeta_\Q(2)^{-1}\overline{W_{{\bf f}\otimes \chi_{-D}}(1)}^{-1}\prod_v\mathcal{W}_v(1) \<V_+^{2m}{\bf f},V_+^{2m}{\bf f}\>_{{\rm SO}(V)}\< V_+^m{\bf h},V_+^m{\bf h}\>_{\SL_2}^{-1} \cdot {V}_+^{m}{\bf h}.
\end{align*}
By Lemmas \ref{lemma 4.2}-\ref{lemma 4.4}, and \cite[\S\,9.3, Lemmas 9.3, 9.6]{Ichino2005}, we have
\begin{align*}
\prod_v\mathcal{W}_v(1) \<V_+^{2m}{\bf f},V_+^{2m}{\bf f}\>_{{\rm SO}(V)}\< V_+^m{\bf h},V_+^m{\bf h}\>_{\SL_2}^{-1}&=-2^{-2m-1}\pi^{-2m}e^{-2\pi}(\sqrt{-1})D^{-\kappa+m+1/2}\\
&\times\left( \SL_2(\Z):\Gamma_0(N)\right )^{-1} \left(\frac{\Gamma(\kappa-m)\Gamma(2m+1)}{\Gamma(\kappa-2m)\Gamma(m+1)}\right )\< f,f\> \<h,h\>^{-1}.
\end{align*}
This completes the proof.
\end{proof}

\subsection{Base change lifts}\label{SS:BC}
Let $(V,Q) = (V^{(4)},Q^{(4)}).$

Let $${\bf e} = \bp 0 & \delta \\ 0 & 0 \ep,\quad {\bf e}^* = \bp 0 & 0 \\ \delta & 0 \ep.$$
Then $V_1=(\Q {\bf e} + \Q {\bf e}^*)^{\perp}= \{x\in V \mbox{ }\vert \mbox{ } x_2=x_3=0 \}$. For each place $v$ of $\Q$, let $\omega_v$ be the Weil representation of $R(\Q_v)$ on $\mathcal{S}(V(\Q_v))$, and $\hat{\omega}_v$ be the representation of $R(\Q_v)$ on $\mathcal{S}(V_1(\Q_v))\otimes \mathcal{S}(\Q_v^2)$ defined in \S\,\ref{SS:Change of polarizations}. 

For each place $v$ of $\Q$, let $\varphi_v = \varphi_v^{(4)}$ and define $\mathcal{W}_v \in \C$ by
\begin{align*}
\mathcal{W}_v&=\int_{U(\Q_v) \backslash \SL_2(\Q_v)}\hat{\omega}_{v}(g,1)\hat{\varphi}_v(D^{-1};0,1) \overline{W_{{\bf g}^{\sharp},v}(a(D^{-2})g)}dg.
\end{align*}

\begin{lemma}\label{L:4.3}
Let $v=p \mid N$. We have
$$\mathcal{W}_p = (1+p)^{-1}.$$
\end{lemma}

\begin{proof}
Note that 
\begin{align}\label{E:4.3}
\hat{\omega}_{p}(k,1)\hat{\varphi}_p(x_1;y) = \I_{\mathcal{O}_p}(x_1)\I_{p\Z_p \times \Z_p}(yk)
\end{align}
for $k \in \SL_2(\Z_p)$. For $k \in \SL_2(\Z_p)$, let 
$$\mathcal{J}(k)= \sum_{n \in \Z}p^n(-D,p^n)_{\Q_p}\hat{\omega}_{p}(k,1)\hat{\varphi}_p(D^{-1}p^{n};0,p^{-n})\overline{W_{{\bf g}^{\sharp},p}(a(D^{-2})t(p^n)k)}.$$
By (\ref{E:4.3}), $\hat{\omega}_{p}(k,1)\hat{\varphi}_p(D^{-1}p^n;0,p^{-n}) \neq 0$ if and only if $n=0$ and $k \in K_0(p)$. Therefore,
\begin{align*}
\mathcal{J}(k) = \left \{ \begin{array}{ll} 1 & \mbox{ if } k\in K_0(p), \\
0 & \mbox{ otherwise}.     \end{array} \right .
\end{align*}
We conclude from the above calculation and (\ref{E:equivariant property of test vector case v=p base change lift}) that
\begin{align*}
\mathcal{W}_p=(1+p)^{-1}\sum_{k \in \SL_2(\Z_p) / K_0(p) }\mathcal{J}(k)=(1+p)^{-1}.
\end{align*}
This completes the proof.
\end{proof}

\begin{lemma}\label{L:theta 4}
We have
$$\theta(\overline{{\bf g}}^{\sharp},\varphi^{(4)})=2^{-1}(\sqrt{-1})D^{(-\kappa+1)/2}\zeta_{\Q}(2)^{-1}\left( \SL_2(\Z):\Gamma_0(N)\right )^{-1} \cdot(\chi_{-D}\times{{\bf g}_{\K}^{\sharp}}).$$
\end{lemma}

\begin{proof}
Following the proof of \cite[Lemma 10.1]{Ichino2005}, with the equivariant properties of $\varphi^{(4)}$ in (\ref{E:equivariant property of test vector case v=p base change lift}) and \cite[(10.1), (10.2), (10.4), and Lemma 10.7]{Ichino2005}, we have
\begin{align*}
\theta(\overline{{\bf g}}^{\sharp},\varphi^{(4)}) = \zeta_\Q(2)^{-1}W_{{\bf g}_\K^{\sharp}}(1)^{-1}\prod_v\mathcal{W}_v\cdot (\chi_{-D}\times{{\bf g}_{\K}^{\sharp}}). 
\end{align*}
By (\ref{E:Whittaker complex case}), Lemma \ref{L:4.3}, and \cite[\S\,10.3, Lemmas 10.3, 10.5, and 10.8]{Ichino2005}, we have
\begin{align*}
W_{{\bf g}_\K^{\sharp}}(1)^{-1}\prod_v\mathcal{W}_v = 2^{-1}(\sqrt{-1})^{(-\kappa+1)/2}\left( \SL_2(\Z):\Gamma_0(N)\right )^{-1}.
\end{align*}
This completes the proof.
\end{proof}

\section{Central value formula for twisted triple product $L$-function}\label{S:triple}

We keep the notation of \S\,\ref{S:Automorphic forms and $L$-functions}.
\subsection{Twisted triple product $L$-function}
Let ${\rm As}(\sigma_\K)$ be the Asai transfer of $\sigma_\K$ to an isobaric automorphic representation of $\GL_4(\A)$ (cf.\,\cite{Kris2003}). Let $\Pi = \pi\times{\rm As}(\sigma_\K)$ be an automorphic representation of $\GL_2(\A)\times\GL_4(\A)$. Let 
$$L(s, \Pi)$$
be the $\GL_2 \times \GL_4$ Rankin-Selberg automorphic $L$-functions. Note that it was proved in \cite[Theorem D]{CCI2018} that $L(s,\Pi)$ coincides with the twisted triple product $L$-function defined by the integral representations in \cite{PSR1987} and \cite{Ikeda1989}. Since $\sigma_\K$ is a base change lift, we have a decomposition into motivic $L$-functions defined in \S\,\ref{SS:motivic L}
\begin{align}\label{E:factorization of L}
L(s, \Pi) = \Lambda\left(s-\frac{1}{2}+\kappa+\kappa',{\rm Sym}^2(g)\otimes f\right)\Lambda\left(s-\frac{1}{2}+\kappa',f\otimes\chi_{-D}\right).
\end{align}
We also denote $L(s,\Pi,{\rm Ad})$ the adjoint $L$-function of $\Pi$.

In this section, by specializing Ichino's formula in \cite{Ichino2008}, we obtain an explicit formula in Proposition \ref{P:explicit central value formula} which relating the central value $L(\frac{1}{2},\Pi)$ with the global period integral 
$$\int_{\A^{\times}\GL_2(\Q)\backslash\GL_2(\A)}V_+^{2m}{\bf f}(g){\bf g}_\K^{\sharp}(g)dg.$$

\subsection{Local trilinear period integrals}\label{SS:Local trilinear period integrals}
Let $v$ be a place of $\Q$. Let $E_v=\Q_v\times\K_v$. Let $\<\cdot,\cdot\>_v$ be a $\GL_2(E_v)$-invariant bilinear pairing on $\Pi_v \times \Pi_v$. Let $\phi_v \in \Pi_v$ be a non-zero vector satisfying the following condition:
\begin{itemize}
\item If $v=p \nmid N$, then $\phi_p$ is $\GL_2(\Z_p)\times\GL_2(\mathcal{O}_p)$-invariant.
\item If $v=p \mid N$, then $\phi_p$ is ${\bf K}_0(p)\times\mathbb{K}_0(p)$-invariant.
\item If $v=\infty$, then $\phi_\infty$ is in the minimal $({\rm SO}(2)\times{\rm SU}(2))$-type of $\Pi_\infty$ and $(0,X) \cdot \phi_\infty=0$. Here $(0,X) \in \frak{gl}_2(\R)\otimes_\R\C \oplus \frak{gl}_2(\C)\otimes_{\R}\C$ and $X$ is defined by
$$X= \frac{1}{2}\bp 0 & -1 \\ 1 & 0 \ep\otimes 1 + \frac{1}{2}\bp  0 & \sqrt{-1} \\ \sqrt{-1} & 0 \ep\otimes \sqrt{-1}.$$
\end{itemize}
Note that the condition uniquely determines $\phi_v \in \Pi_v$ up to scalars. We call $\C\cdot\phi_v $ the new line of $\Pi_v$. Define the local trilinear period integral $\mathcal{I}(\Pi_v) \in \C$ as follows:
\begin{itemize}
\item If $v=p$, then
\begin{align*}
\mathcal{I}(\Pi_p) = \frac{1}{\zeta_{\K_p}(2)}\cdot \frac{L(1,\Pi_p,{\rm Ad})}{L(1/2,\Pi_p)}\cdot \int_{\Q_p^{\times} \backslash \GL_2(\Q_p)}\frac{\<\Pi_p(g)\phi_p,\phi_p \>_p}{\<\phi_p,\phi_p\>_p} dg.
\end{align*}
\item If $v=\infty$, then
\begin{align*}
\mathcal{I}(\Pi_{\infty}) &= \frac{1}{\zeta_{\C}(2)}\cdot \frac{L(1,\Pi_\infty,{\rm Ad})}{L(1/2,\Pi_\infty)}\cdot\int_{\R^{\times} \backslash \GL_2(\R)}\frac{\<\Pi_\infty(g(V_+^{2m},{\bf t}_\infty))\phi_\infty,\Pi_\infty((V_+^{2m},{\bf t}_\infty))\phi_\infty \>_\infty}{\<\Pi_\infty((a(-1),a(-1){\bf t}_\infty))\phi_\infty,\Pi_\infty((1,{\bf t}_\infty))\phi_\infty\>_\infty} dg.
\end{align*}
\end{itemize}
By \cite[Lemma 2.1]{Ichino2008}, the integrals defining $\mathcal{I}(\Pi_v)$ are absolutely convergent. Note that $\mathcal{I}(\Pi_v)$ only depends on the representation $\Pi_v$ and its new line, not on the choices of the pairing $\<\cdot,\cdot\>_v$ and $\phi_v \in \Pi_v$.

\begin{lemma}\label{L:period integral p nmid N}
For $v=p \nmid N$, we have
$$\mathcal{I}(\Pi_p)=1.$$
\end{lemma}

\begin{proof}
If $\K_p$ is unramified or split over $\Q_p$, it was proved in \cite[Lemma 2.2]{Ichino2008}. If $\K_p$ is ramified over $\Q_p$, it was proved in \cite[Proposition 4.7]{CC2017}.
\end{proof}

\begin{lemma}\label{L:period integral p mid N}
For $v=p \mid N$, we have
$$\mathcal{I}(\Pi_p)=\begin{cases}
2p^{-2}(1+p)^{-2}\left(\GL_2(\Z_p)\times\GL_2(\mathcal{O}_p): {\bf K}_0(p)\times\mathbb{K}_0(p)\right) & \mbox{ if }(p,-\tau_pD)_{\Q_p}=-1,\\
0 & \mbox{ otherwise}.
\end{cases}
$$
\end{lemma}

\begin{proof}
Note that we have assume $p\nmid D$. If $\K_p$ is split over $\Q_p$, it was proved in \cite[\S\,7]{IchinoIkeda2010}. If $\K_p$ is unramified over $\Q_p$, it was proved in \cite[Proposition 4.8]{CC2017}.
\end{proof}

\begin{rmk}
Let $\epsilon(s,\Pi_p,\psi_p)$ be the $\epsilon$-factor associated to $\Pi_p$ with respect to $\psi_p$. Then 
$$\epsilon\left( \frac{1}{2},\Pi_p,\psi_p\right) = -(p,-\tau_pD)_{\Q_p}.$$
By the results of Prasad \cite{Prasad1992}, the vanishing of $\mathcal{I}(\Pi_p)$ is a consequence of the dichotomy criterion of trilinear forms that 
$${\rm Hom}_{\GL_2(\Q_p)}(\Pi_p,\C) \neq 0 \mbox{ if and only if }\epsilon\left( \frac{1}{2},\Pi_p,\psi_p\right)=1.$$
\end{rmk}

Let $C_\infty(\kappa,\kappa') \in \Q$ defined by
\begin{align}\label{E:archimedean rational number}
\begin{split}
C_\infty(\kappa,\kappa') &= 2^{4m}\frac{\Gamma(2\kappa)}{\Gamma(4\kappa)\Gamma(\kappa')^2\Gamma(2m+1)} \sum_{j=0}^{2m}\sum_{n=0}^{2\kappa}\sum_{i=0}^{2m-j}(-1)^{i+j}{2m \choose j}{2\kappa \choose n}\Gamma(\kappa+1+n)\Gamma(3\kappa+1-n)\\
&\times\frac{\Gamma(\kappa'+j+n)\Gamma(\kappa+\kappa'+j)}{\Gamma(2\kappa'+j)(2\kappa+\kappa'+j-n)}\\
&\times \frac{\Gamma(2\kappa+\kappa'+j-n+1)\Gamma(\kappa+\kappa'+j+1/2+i)\Gamma(2m-j+1)}{\Gamma(2\kappa+\kappa'+j-n+1+i)\Gamma(\kappa+\kappa'+j+1/2)\Gamma(2m-j+1-i)\Gamma(\kappa'+j+n+1+i)}.
\end{split}
\end{align}

\begin{lemma}\label{P:local period integral archimedean}
We have 
$$\mathcal{I}(\Pi_{\infty}) = 2^{-6\kappa+6\kappa'-2}\pi^{-4m}(2\kappa+1)C_\infty(\kappa,\kappa').$$
Moreover, $C_\infty(\kappa,\kappa')\neq 0$.
\end{lemma}

The proof of Lemma \ref{P:local period integral archimedean} will be given in \S\,\ref{S: Local trilinear period integral in the C times R case}.

\begin{rmk}\noindent
\begin{itemize}
\item[(1)]We conjecture that 
$$C_\infty(\kappa,\kappa')=\left ( \frac{\Gamma(\kappa-m)\Gamma(2m+1)}{\Gamma(\kappa-2m)\Gamma(m+1)}\right)^2.$$
Verified by using computer, the equality holds for all $m \leq 500$.
\item[(2)]For arbitrary semisimple cubic algebra $E_\infty$ over $\R$, except in the cases considered here, the local trilinear period integrals were calculated in \cite{Ikeda1998}, \cite{Ikeda1999}, \cite{Watson2008}, \cite{CC2017}, and \cite{YaoThesis}.
\end{itemize}
\end{rmk}

\subsection{Central value formula}
Let $E=\Q\times\K$. Let $\phi = {\bf f}\times{\bf g}_\K \in \Pi$ be a cusp form on $\GL_2(\A_E)$. Note that $\Pi((V_+^{2m},{\bf t}_\infty))\phi = V_+^{2m}{\bf f} \times {\bf g}_\K^{\sharp}$. Let
\begin{align*}
\mathcal{I}(\Pi) &= \int_{\A^{\times}\GL_2(\Q)\backslash\GL_2(\A)}\Pi((V_+^{2m},{\bf t}_\infty))\phi(g)dg,\\
\Omega(\Pi) &=\int_{\A_E^{\times}\GL_2(E) \backslash \GL_2(\A_E)}\phi(g((a(-1),a(-1){\bf t}_{\infty})))\phi(g(1,{\bf t}_{\infty}))dg.
\end{align*}

\begin{lemma}
\label{L:Petersson product}
We have
\begin{align*}
\frac{\Omega(\Pi)}{L(1,{\Pi},{\rm Ad})}&=2^{-2\kappa'-3}\zeta_{\Q}(2)^{-1}\zeta_{\K}(2)^{-1}D^{-1/2}(2\kappa+1)^{-1}\\
&\times N^3\left(\GL_2(\widehat{\Z})\times\GL_2(\widehat{\mathcal{O}}): {\bf K}_0(N\widehat{\Z})\times \mathbb{K}_0(N\widehat{\mathcal{O}}) \right)^{-1}.
\end{align*}
\end{lemma}

\begin{proof}
The formula is obtained by specializing the formula in \cite[Proposition 6]{Wald1985}. We leave the details to the readers.
\end{proof}

\begin{prop}\label{P:explicit central value formula}
If $(p,-\tau_pD)_{\Q_p}=1$ for some $p \mid N$, then
$$\mathcal{I}(\Pi) = 0.$$
If $(p,-\tau_pD)_{\Q_p}=-1$ for all $p \mid N$, then
\begin{align*}
\mathcal{I}(\Pi)^2 & = 2^{-6\kappa+4\kappa'-6+\nu(\Pi)}\pi^{-4m}D^{-1/2}\zeta_{\Q}(2)^{-2}C_\infty(\kappa,\kappa')\prod_{p \mid N}p(1+p)^{-2}\cdot L\left (\frac{1}{2},\Pi\right ).
\end{align*}
Here $\nu(N)$ is the number of prime divisors of $N$, and $C_\infty(\kappa,\kappa')\in\Q$ is defined in (\ref{E:archimedean rational number}).
\end{prop}

\begin{proof}
By Ichino's formula in \cite[Theorem 1.1 and Remark 1.3]{Ichino2008}, we have
\begin{align*}
\frac{\mathcal{I}(\Pi)^2}{\Omega(\Pi)} & = \frac{2\zeta_{\Q}(2)^{-1}}{2^2}\cdot \zeta_\K(2)\cdot\frac{L(1/2,\Pi)}{L(1,\Pi,{\rm Ad})}\cdot\prod_v\mathcal{I}(\Pi_v).
\end{align*}
The assertions then follow from Lemmas \ref{L:period integral p nmid N}-\ref{L:Petersson product}. This completes the proof.
\end{proof}

\section{Proof of main results}\label{S:main results}
We keep the notation of \S\,\ref{S:Automorphic forms and $L$-functions}. We make the following assumptions on the imaginary quadratic field $\K$ in \S\,\ref{SS:Automorphic forms on GL_2 over K}:
\begin{itemize}
\item $(D,N)=1$.
\item $-D\equiv 1\mbox{ mod }8$.
\item $(p,-\tau_pD)_{\Q_p} = -1 \mbox{ for }p\mid N$.
\item $\Lambda\left(\kappa',{f} \otimes \chi_{-D} \right) \neq 0.$
\end{itemize}
By the nonvanishing theorems in \cite{FriedbergHoffstein1995} and \cite{Wald1991}, such fundamental discriminant exists.

Let $C(\kappa,\kappa') \in \Q$ defined by
\begin{align}\label{E:archimedean constant}
C(\kappa,\kappa') = C_\infty(\kappa,\kappa')\left (\frac{\Gamma(\kappa-2m)\Gamma(m+1)}{\Gamma(\kappa-m)\Gamma(2m+1)}\right )^2.
\end{align}
Here $C_\infty(\kappa,\kappa') \in \Q$ is defined in (\ref{E:archimedean rational number}). By Lemma \ref{P:local period integral archimedean}, $C(\kappa,\kappa')$ is non-zero. The following is our main theorem.
\begin{thm}\label{T:pullback formula}
Assume Hypothesis (H) holds. We have
\begin{align*}
\frac{\vert \langle \Delta_{\kappa'+1}^m{F}|_{\frak{H}\times \frak{H}}, { g}\times{ g} \rangle \vert^2} {\<g,g\>^2}&=2^{-\kappa-6m-1}\prod_{p\mid N}p(1+p)^{-2}C(\kappa,\kappa')
\frac{\langle { h},{ h} \rangle}{\langle { f},{ f} \rangle} \Lambda\left(\kappa+\kappa',{\rm Sym}^2(g)\otimes f\right) .
\end{align*}
\end{thm}

\begin{proof}
We may assume $c_h(n) \in \R$ for all $n\in {\mathbb N}$. By seesaw identities in Propositions \ref{P:SO(2,2)-period to SL_2 period} and \ref{P:SL_2-period to SO(2,1)-period}, the central value formula in Proposition \ref{P:explicit central value formula}, and the factorization in (\ref{E:factorization of L}), we have
\begin{align*}
\frac{\vert \langle \Delta_{\kappa'+1}^m{F}|_{\frak{H}\times \frak{H}}, { g}\times{ g} \rangle \vert^2} {\<g,g\>^2}&=2^{-\kappa-6m-1}\prod_{p\mid N}p(1+p)^{-2}C(\kappa,\kappa')
\frac{\langle { h},{ h} \rangle}{\langle { f},{ f} \rangle} \Lambda\left(\kappa+\kappa',{\rm Sym}^2(g)\otimes f\right)\\
&\times 2^{\kappa'-1+\nu(N)}c_h(D)^{-2}D^{\kappa'-1/2}\<f,f\>^{-1}\<h,h\>\Lambda\left ( \kappa', f \otimes \chi_{-D}\right ).
\end{align*}
By the Kohnen-Zagier formula \cite[Corollary 1]{Kohnen1985},  
$$\frac{c_{h}(D)^2}{\langle  h, h \rangle}=2^{\kappa'-1+\nu(N)}D^{\kappa'-1/2}\frac{\Lambda\left ( \kappa', f \otimes \chi_{-D}\right )}{\langle  f,  f \rangle}.$$
This completes the proof.
\end{proof}

As a corollary of Theorem \ref{T:pullback formula}, we obtain Deligne conjecture for the central critical value $\Lambda(\kappa+\kappa',{\rm Sym}^2(g)\otimes f)$.

\begin{corollary}\label{C:algebraicity}
Assume Hypothesis (H) holds. For $\sigma \in {\rm Aut}(\C)$, 
$$\left (\frac{\Lambda(\kappa+\kappa',{\rm Sym}^2({ g})\otimes { f})}{\langle { g},{ g}\rangle^2\Omega_f^+} \right )^{\sigma}=\frac{\Lambda(\kappa+\kappa',{\rm Sym}^2({ g}^{\sigma})\otimes {f}^{\sigma})}{\langle { g}^{\sigma},{ g}^{\sigma}\rangle^2\Omega_{f^{\sigma}}^+}.$$
Here $\Omega_f^+$ is the plus period of ${ f}$ defined in \cite{Shimura1977}.
\end{corollary}

\begin{proof}
%If $\kappa'$ is even, then 
%$$\epsilon\left(\frac{1}{2},{\rm Sym}^2(\sigma)\otimes\pi\right)=-1.$$
%Therefore $\Lambda(\kappa+\kappa',{\rm Sym}^2(g)\otimes f) =0$ by functional equation and the conclusion holds.

%Assume $\kappa'$ is odd. Put $m= \frac{\kappa-\kappa'}{2}$. 
Let $\Q(f)$ and $\Q(g)$ be the Hecke fields of $f$ and $g$, respectively. We may assume $c_h(n) \in \Q(f)$ for all $n \in {\mathbb N}$ (cf.\,\cite[Proposition 4.5]{Shimura1982} and \cite[Theorem 4.5]{Prasanna2009}). Then, 
$$\frac{ \langle \Delta_{\kappa'+1}^m{F}|_{\frak{H}\times \frak{H}}, { g}\times{ g} \rangle } {\<g,g\>^2} \in \Q(f)\Q(g).$$
In particular, $\langle \Delta_{\kappa'+1}^m{F}|_{\frak{H}\times \frak{H}}, { g}\times{ g} \rangle \in \R.$

Let $\sigma \in {\rm Aut}(\C)$. Note that $h^{\sigma}\in S_{\kappa'+1/2}^{+}(\Gamma_0(4N))$ is a newform associated to $f^{\sigma}$ (cf.\,\cite[Theorem 4.5]{Prasanna2009}), and $F^{\sigma}$ is the Saito-Kurokawa lift of $h^{\sigma}$. By the Kohnen-Zagier formula \cite[Corollary 1]{Kohnen1985},  
\begin{align*}
\frac{ c_{h}(D)^2}{\langle  h, h \rangle}&=2^{\kappa'-1+\nu(N)}D^{\kappa'-1/2}\frac{\Lambda(\kappa',f \otimes \chi_{-D})}{\langle  f,  f \rangle},\\
\frac{c_{h^{\sigma}}(D)^2}{\langle  h^{\sigma}, h^{\sigma} \rangle}&=2^{\kappa'-1+\nu(N)}D^{\kappa'-1/2}\frac{\Lambda(\kappa',f^{\sigma} \otimes \chi_{-D})}{\langle  f^{\sigma},  f^{\sigma} \rangle}.
\end{align*}
By \cite{Shimura1977}, 
\begin{align*}
\left( 
\frac{\Lambda \left (\kappa',f \otimes \chi_{-D}\right )}{D^{1/2}
\Omega_f^{+}}\right )^{\sigma}
&=
\frac{\Lambda \left (\kappa',f^{\sigma}\otimes\chi_{-D}\right )}{D^{1/2}\Omega_{f^{\sigma}}^{+}}.
\end{align*}
We conclude that 
\begin{align}\label{L:Galois invariance of ratio of Perersson norm}
\left ( \frac{\<f,f\>}{\<h,h\> \Omega_{f}^+} \right )^{\sigma} = \frac{\<f^{\sigma},f^{\sigma}\>}{\<h^{\sigma},h^{\sigma}\> \Omega_{f^{\sigma}}^+}.
\end{align}
Recall that if $\phi : \frak{H}\rightarrow \C$ is a nearly holomorphic modular form with Fourier expansion
$$\phi(\tau) = \sum_{i=0}^{k}\left (\frac{1}{4\pi y} \right )^{i}\sum_{n=0}^{\infty} a_{i,n} q^n,$$
then $\phi^{\sigma}$ is a nearly holomorphic modular form whose Fourier expansion is given by
$$\phi^{\sigma}(\tau) = \sum_{i=0}^{k}\left (\frac{1}{4\pi y} \right )^{i}\sum_{n=0}^{\infty} a_{i,n}^{\sigma} q^n.$$
Therefore, by the Fourier expansion of $\Delta_{\kappa'+1}^m F$ in (\ref{E:Fourier expansion of nearly holomorphic Siegel modular form}), we have
$$\Delta_{\kappa'+1}^m F^{\sigma} \vert_{\frak{H}\times \frak{H}}=(\Delta_{\kappa'+1}^m F \vert_{\frak{H}\times \frak{H}})^{\sigma}.$$
On the other hand, since $g$ is a newform and $\Delta_{\kappa'+1}^m F \vert_{\frak{H}\times \frak{H}}$ is a nearly holomorphic modular form on $\frak{H}\times \frak{H}$, we have (cf.\,\cite[Theorem 4]{Sturm1980} and
\cite{Shimura1976})
\begin{align*}
\left (\frac{\langle   \Delta_{\kappa'+1}^m{ F}|_{\frak{H}\times \frak{H}} ,{ g}\times { g}     \rangle}{\langle { g},{ g}\rangle^2} \right )^{\sigma} = \frac{\langle   (\Delta_{\kappa'+1}^mF|_{\frak{H}\times \frak{H}})^{\sigma} ,{ g^{\sigma}}\times { g^{\sigma}}     \rangle}{\langle { g^{\sigma}},{ g^{\sigma}}\rangle^2}.
\end{align*}
We conclude that
\begin{align}\label{E:Galois invariance}
\left (\frac{\langle   \Delta_{\kappa'+1}^m{ F}|_{\frak{H}\times \frak{H}} ,{ g}\times { g}     \rangle}{\langle { g},{ g}\rangle^2} \right )^{\sigma}=\frac{\langle   \Delta_{\kappa'+1}^mF^{\sigma}|_{\frak{H}\times \frak{H}} ,{ g^{\sigma}}\times { g^{\sigma}}     \rangle}{\langle { g^{\sigma}},{ g^{\sigma}}\rangle^2}.
\end{align}
The corollary then follows from Theorem \ref{T:pullback formula}, (\ref{L:Galois invariance of ratio of Perersson norm}), and (\ref{E:Galois invariance}). This completes the proof.
\end{proof}

\begin{corollary}\label{C:algebraicity 2}
Assume Hypothesis (H) holds and $\Lambda(\kappa+\kappa',{\rm Sym}^2({ g})\otimes { f}) \neq 0$. Let $n \in \Z$ be a critical value for ${\rm Sym}^2({ g})\otimes {f}$, and $\chi$ be a Dirichlet character such that $(-1)^{n}\chi(-1)=1$. For $\sigma \in {\rm Aut}(\C)$, we have
$$\left (\frac{\Lambda(n,{\rm Sym}^2({ g})\otimes { f}\otimes \chi)}{G({\chi})^3(2\pi \sqrt{-1})^{3(n-\kappa-\kappa')}\langle { g},{ g}\rangle^2\Omega_f^+} \right )^{\sigma}=\frac{\Lambda(n,{\rm Sym}^2({ g}^{\sigma})\otimes {f}^{\sigma}\otimes \chi^{\sigma})}{G({\chi}^{\sigma})^3(2\pi \sqrt{-1})^{3(n-\kappa-\kappa')}\langle { g}^{\sigma},{ g}^{\sigma}\rangle^2\Omega_{f^{\sigma}}^+}.$$
Here $G(\chi)$ is the Gauss sum associated to $\chi$, and $\Omega_f^+$ is the plus period of ${ f}$ defined in \cite{Shimura1977}.
\end{corollary}

\begin{proof}
By \cite[Theorem A]{Januszewski2017}, there exists cohomological periods $\Omega_{\pm}(f,g) \in \C^{\times}$ such that
\begin{align}\label{E:Fabian}
\left (\frac{\Lambda(n,{\rm Sym}^2({ g})\otimes { f}\otimes \chi)}{G({\chi})^3(2\pi \sqrt{-1})^{3n} \Omega_{(-1)^n\chi(-1)}(f,g)}\right )^{\sigma} = \frac{\Lambda(n,{\rm Sym}^2({ g^{\sigma}})\otimes { f^{\sigma}}\otimes \chi^{\sigma})}{G({\chi}^{\sigma})^3(2\pi \sqrt{-1})^{3n} \Omega_{(-1)^n\chi(-1)}(f^{\sigma},g^{\sigma})}
\end{align}
for $\sigma \in {\rm Aut}(\C)$. Note that the condition $\kappa\geq\kappa'$ is equivalent to the balanced condition in \cite[Theorem A]{Januszewski2017}. For $n=\kappa+\kappa'$ and $\chi=1$, by Corollary \ref{C:algebraicity}, (\ref{E:Fabian}), and the assumption $\Lambda(\kappa+\kappa',{\rm Sym}^2({ g})\otimes { f}) \neq 0$, we have
\begin{align}\label{E:comparison of periods}
\left( \frac{\<g,g\>^2\Omega_f^+}{(2\pi\sqrt{-1})^{3(\kappa+\kappa')}\Omega_+(f,g)}\right)^{\sigma} = \frac{\<g^{\sigma},g^{\sigma}\>^2\Omega_{f^{\sigma}}^+}{(2\pi\sqrt{-1})^{3(\kappa+\kappa')}\Omega_+(f^{\sigma},g^{\sigma})}
\end{align}
for $\sigma \in {\rm Aut}(\C)$. The assertion follows from (\ref{E:Fabian}) and (\ref{E:comparison of periods}). This completes the proof.
\end{proof}

\section{Local trilinear period integral in the $\R \times \C$ case}\label{S: Local trilinear period integral in the C times R case}
\subsection{Setting}The aim of this section is to give a proof of Lemma \ref{P:local period integral archimedean}. The main results of this section are Proposition \ref{P:archimedean local period integral} and Corollary \ref{C:nonvanishing of rational number}.

We follow the normalization of measures as in \S\,\ref{SS:measures}. Let $\psi_1$ be the standard additive character of $\R$, and $\psi_2 = \psi_1 \circ {\rm tr}_{\C/\R}$ be an additive character of $\C$. Let $\pi_1$ be the discrete series representation of $\GL_2(\R)$ of weight $2\kappa'$, and $\sigma_2$ be the principal series representation 
$${\rm Ind}_{{\bf B}(\C)}^{\GL_2(\C)}(\mu^{\kappa} \boxtimes \mu^{-\kappa})$$
for some positive integers $\kappa$ and $\kappa'$. Here $\mu(z)=(z/\overline{z})^{1/2}$ for $z \in \C^{\times}$. We assume $$\kappa-\kappa' =2m \in 2\Z_{\geq 0}.$$ 

Let $W_{\R} \in \mathcal{W}(\pi_1,\psi_1)$ and $W_{\C} \in \mathcal{W}(\sigma_2,\psi_2)$ be Whittaker functions of $\pi_1$ and $\sigma_2$ with respect to $\psi_1$ and $\psi_2$ defined by
$$W_{\R}(a(y)k_{\theta}) = e^{2\sqrt{-1}\kappa' \theta}y^{\kappa'}e^{-2\pi y}\I_{\R_{>0}}(y)$$
for $y \in \R^{\times}$ and $k_\theta \in {\rm SO}(2)$.
$$W_{\C}(zt(a)k)=a^{2\kappa+2}\sum_{n=0}^{2\kappa}{2\kappa \choose n}(\sqrt{-1})^{n}\alpha^{2\kappa-n}\overline{\beta}^{n}K_{\kappa-n}(4\pi a^2)
$$
for $z \in \C^{\times}$, $a \in \R_{>0}$ and $k=\begin{pmatrix} \alpha&\beta \\ -\overline{\beta} & \overline{\alpha}  \end{pmatrix} \in {\rm SU}(2).$
Note that $W_{\R}$ is in the minimal ${\rm SO}(2)$-type of $\pi_1$, and $W_{\C}$ is in the minimal ${\rm SU}(2)$-type of $\sigma_2$ such that $X\cdot W_\C=0$.
Here $X \in \frak{gl}_2(\C)\otimes_{\R}\C$ is defined by
$$X= \frac{1}{2}\bp 0 & -1 \\ 1 & 0 \ep\otimes 1 + \frac{1}{2}\bp  0 & \sqrt{-1} \\ \sqrt{-1} & 0 \ep\otimes \sqrt{-1}.$$

Let $\<\mbox{ },\mbox{ } \>_1$ and $\<\mbox{ },\mbox{ } \>_2$ be invariant pairings on $\mathcal{W}(\pi_1,\psi_1)\otimes \mathcal{W}(\pi_1,\psi_1) $ and $\mathcal{W}(\sigma_2,\psi_2) \otimes \mathcal{W}(\sigma_2,\psi_2)$ defined by
\begin{align*}
\langle W_1, W_2 \rangle_1&= \int_{\R^{\times}}W_1(a(t))W_2(a(-t)) d^{\times}t,\\
\langle W_1, W_2 \rangle_2&= \int_{\C^{\times}}W_1(a(t))W_2(a(-t)) d^{\times}t.
\end{align*}
Define $\mathcal{I}(\Pi_{\infty}) \in \C$ by
\begin{align*}
\mathcal{I}(\Pi_{\infty})&= \frac{1}{\zeta_{\C}(2)}\cdot \frac{L(1,\Pi_{\infty},{\rm Ad})}{L(1/2,\Pi_{\infty})}\cdot \int_{\R^{\times} \backslash \GL_2(\R)}\frac{\<\pi_1(g){V}_+^{2m}W_{\R}, {V}_+^{2m}W_{\R}\>_1 \< \sigma_2(g{\bf t}_{\infty})W_{\C},\sigma_2({\bf t}_{\infty})W_{\C}\>_2}{\<\pi_1(a(-1))W_{\R},W_{\R}\>_1 \<\sigma_2(a(-1){\bf t}_{\infty})W_{\C},\sigma_2({\bf t}_{\infty})W_{\C}\>_2}dg.
\end{align*}
Here ${\bf t}_{\infty} = \displaystyle{\frac{1}{\sqrt{2}}\bp 1 & -\sqrt{-1} \\ -\sqrt{-1} & 1 \ep  \in {\rm SU}(2).}$ The aim of this section is devoted to the calculation of this integral.

Let $(V,Q)$ be the quadratic space over $\R$ defined by $V={\rm M}_2$ and $Q[x]=\det(x)$. Note that $${\mathbb G}_m \backslash (\GL_2\times \GL_2)  \simeq {\rm GSO}(V)$$ as described in \S\,\ref{SS:5.1}. We write $[h_1,h_2] \in {\rm GSO}(V)$ for the image of $(h_1,h_2) \in \GL_2\times\GL_2$ under this isomorphism. Let $\omega_1$ and $\omega_2$ be the Weil representations of $\SL_2(\R)\times {\rm O}(V)(\R)$ and $\SL_2(\C) \times {\rm O}(V)(\C)$ on $S(V(\R))$ and $S(V(\C))$ with respect to $\psi_1$ and $\psi_2$, respectively. We extend the Weil representations to representations of $R(\R)$ and $R(\C)$. 
Let $\varphi_{\R} \in S(V(\R))$ and $\varphi_{\C} \in S(V(\C))$ be defined by
\begin{align*}
\varphi_{\R}(x)&=(x_1+\sqrt{-1}x_2+\sqrt{-1}x_3-x_4)^{2\kappa}e^{-\pi\,{\rm tr}(x{}^tx)},\\
\varphi_{\C}(x)&=(2\kappa+1)x_3^{2\kappa}e^{-2\pi \,{\rm tr}(x{}^{t}\overline{x})}.
\end{align*} 

\subsubsection{Outline of the proof}
We briefly sketch the idea of the proof of Lemma \ref{P:local period integral archimedean}.

Let $E=\R\times\C$. Let $\Pi_{\infty} = \pi_1 \times \sigma_2$ be an irreducible admissible representation of $\GL_2(E)$. Let $W_E \in \mathcal{W}(\Pi_{\infty}, \psi_1\circ {\rm tr}_{E/ \R})$ be a Whittaker function of $\Pi_{\infty}$ with respect to $\psi_1\circ{\rm tr}_{E/\R}$ defined by
$$W_E(g) = ({V}_+^{2m}W_{\R}\otimes \sigma_2(w)W_{\C})(g).$$
Define $\varphi_E \in S(V^3(\R))$ by 
$$\varphi_E=\varphi_{\R}\otimes \omega_{2}(1,({\bf t}_{\infty},{\bf t}_{\infty}))\varphi_{\C}.$$
Here we identify $S(V^3(\R))$ with $S(V(E))$ by the isomorphism
\begin{align*}
\R^3 &\longrightarrow E\\
({x}_1,{x}_2,{x}_3) & \longmapsto ({ x}_1,{ x}_2+\sqrt{-1}{x}_3).
\end{align*}

 By \cite[Proposition 5.1]{Ichino2008}, we have
\begin{align*}
\int_{\R^{\times} \backslash \GL_2(\R)}\Phi(g;\varphi_E,W_E)dg&=2^{-1}\zeta_{\R}(2)\cdot Z_{\infty}(0,W_E,f_{\varphi_E}).
\end{align*}
Here $\Phi(g;\varphi_E,W_E)$ is a matrix coefficient of $\Pi_{\infty}$ defined in (\ref{E:matrix coeff by Weil rep}) by the invariant pairings $B_1$ and $B_2$ defined in \S\,\ref{SS:Comparison of invariant pairings}, and $Z_{\infty}(s,W_E,f_{\varphi_E})$ is the local zeta integral defined in (\ref{E:local zeta integral}). On the other hand, the matrix coefficient $\Phi(g;\varphi_E,W_E)$ is proportionl to 
$$\<\pi_1(g_1){V}_+^{2m}W_{\R},{V}_+^{2m}W_{\R}\>_1\< \sigma_2(g_2{\bf t}_{\infty})W_{\C},\sigma_2({\bf t}_{\infty})W_{\C}\>_2$$
for $g=(g_1,g_2) \in \GL_2(\R) \times \GL_2(\C).$ We determine the constant of proportionality in Lemma \ref{L:comparison of invariant pairings}. 
The calculation of $Z_{\infty}(0,W_E,f_{\varphi_E})$ is carried out in Proposition \ref{P:local zeta integral}. The first part of Lemma \ref{P:local period integral archimedean} follows and is concluded in Proposition \ref{P:archimedean local period integral}. In Corollary \ref{C:nonvanishing of rational number}, we prove that the constant $C_{\infty}(\kappa,\kappa')$ is non-zero. The idea is to relate the local period integral $\mathcal{I}(\Pi_{\infty})$ with the square of a local zeta integral $\Psi_{\infty}(f_{\R} \otimes \sigma_2({\bf t}_{\infty})W_{\C})$ defined in (\ref{E:Rankin-Selberg local zeta integral}). Finally, it was proved by Ghate that the integral $\Psi_{\infty}(f_{\R} \otimes \sigma_2({\bf t}_{\infty})W_{\C})$ is non-zero.

\subsection{Comparison of invariant pairings}\label{SS:Comparison of invariant pairings}
Following \cite{Ichino2008} and \cite{Wald1985}, we define invariant pairings on $\mathcal{W}(\pi_1,\psi_1)\otimes \mathcal{W}(\pi_1,\psi_1) $ and $\mathcal{W}(\sigma_2,\psi_2) \otimes \mathcal{W}(\sigma_2,\psi_2)$ by realizing they as quotients of Weil representations via the Jacquet-Langlands-Shimizu lifts. Let 
\begin{align*}
\theta_1 : S(V(\R))\otimes \mathcal{W}(\pi_1,\psi_1)&\longrightarrow \mathcal{W}(\pi_1,\psi_1)\otimes \mathcal{W}(\pi_1,\psi_1),\\
\theta_2 : S(V(\C))\otimes \mathcal{W}(\sigma_2,\psi_2)&\longrightarrow \mathcal{W}(\sigma_2,\psi_2)\otimes \mathcal{W}(\sigma_2,\psi_2)
\end{align*}
be the equivariant maps defined by
\begin{align*}
\theta_1(\varphi\otimes W)(h_1)&=\int_{\SL_2(\R)}\hat{\omega}_{1}(g_1',h)\hat{\varphi}(g) W(gg_1')dg,\\
\theta_2(\varphi\otimes W)(h_2)&=\int_{\SL_2(\C)}\hat{\omega}_{2}(g_2',h)\hat{\varphi}(g) W(gg_2')dg
\end{align*}
with $(g_1',h_1) \in R(\R)$ and $(g_2',h_2) \in R(\C)$.

 Define a map $\tilde{B}_1 : S(V(\R))\otimes \mathcal{W}(\pi_1,\psi_1) \longrightarrow \C$ by 
$$\tilde{B}_1(\varphi,W) = \int_{U(\R) \backslash \SL_2(\R)}\omega_{1}(g,1)\varphi(1)W(a(-1)g)dg.$$ 
Similarly, define a map $\tilde{B}_2 : S_{2}(V(\C))\otimes \mathcal{W}(\sigma_2,\psi_2) \longrightarrow \C$ by
$$\tilde{B}_2(\varphi,W) = \int_{U(\C) \backslash \SL_2(\C)}\omega_{2}(g,1)\varphi(1)W(a(-1)g)dg.$$
By \cite[Lemma 3.2]{Ichino2008}, there exists invariant pairings 
\begin{align*}
B_1 &: \mathcal{W}(\pi_1,\psi_1)\otimes \mathcal{W}(\pi_1,\psi_1) \longrightarrow \C, \\
B_2 &: \mathcal{W}(\sigma_2,\psi_2)\otimes \mathcal{W}(\sigma_2,\psi_2) \longrightarrow \C
\end{align*}
such that
\begin{align*}
\tilde{B}_i=B_i \circ \theta_i
\end{align*}
for $i=1,2.$ In particular, for $\varphi \otimes W \in S(V(\R))\otimes \mathcal{W}(\pi_1,\psi_1)$, the map 
\begin{align*}
\Psi_1(\mbox{ }; \varphi,W) : \GL_2(\R) &\longrightarrow \C \\
g &\longmapsto \tilde{B}_1(\omega_{1}(1,(g,1))\varphi,W) 
\end{align*}
is a matrix coefficient of $\pi_1$. Similar for $\sigma_2$.

\begin{lemma}\label{L:comparison of invariant pairings}
We have
\begin{align*}
\theta_1(\varphi_{\R},{V}_+^{2m}W_{\R}) &= 2^{2\kappa}\cdot{V}_+^{2m}W_{\R} \otimes {V}_+^{2m}W_{\R},\\
\theta_2(\varphi_{\C},\sigma_2(w)W_{\C}) &= W_{\C}\otimes W_{\C}
\end{align*}
and
$$\< \mbox{ }, \mbox{ }\>_i=B_i$$
for $i=1,2$.
\end{lemma}

\begin{proof}
For $l \in \Z_{\geq 0}$, let $\varphi_l \in S(V(\R))$ be defined by
$$\varphi_l(x) = (x_1+\sqrt{-1}x_2+\sqrt{-1}x_3-x_4)^{l}e^{-\pi\,{\rm tr}(x{}^tx)}.$$
Note that $\omega_1(k_{\theta},(k_{\theta_1},k_{\theta_2}))\varphi_l= e^{\sqrt{-1}l(-\theta+\theta_1+\theta_2)}\varphi_l$ for $k_{\theta},k_{\theta_1}, k_{\theta_2} \in {\rm SO}(2)$, and 
$$\hat{\varphi}_l(x)(x_1-\sqrt{-1}x_2+\sqrt{-1}x_3+x_4)^{l}e^{-\pi\,{\rm tr}(x{}^tx)}.$$
Therefore, for each $l \in \Z_{\geq 0}$, there exists a constant $C_{l}$ such that
$$\theta_1(\varphi_{2\kappa'+2l},{V}_+^{l}W_{\R}) = C_l\cdot{V}_+^{l}W_{\R}\otimes {V}_+^{l}W_{\R}.$$
Then,
$$\theta_1 (\omega_{1}\left ( a(-1),\left[a(-1),1 \right ]     \right ) \varphi_{2\kappa'+2l}, \pi_1(a(-1)) {V}_+^{l}W_{\R} ) = C_l \pi_1(a(-1)){V}_+^{l}W_{\R}\otimes {V}_+^{l}W_{\R}.$$
By 
Lemma \ref{L:archimedean local integral for Shimura lifts and Saito-Kurokawa lifts},
\begin{align}\label{E:R1}
\begin{split}
\theta_1(\varphi_{2\kappa'},W_{\R})(1) &=2^{1-2\kappa'} \pi^{-\kappa'}e^{4\pi}\int_{\R_{>0}}a^{2\kappa'-2}H_{2\kappa'}(\sqrt{\pi}(2a+a^{-1}))e^{-\pi(2a+a^{-1})^2}da\\
&=2^{2\kappa'}e^{-4\pi}.
\end{split}
\end{align}
A simple calculation shows that
\begin{align}\label{E:R2}
\begin{split}
 \langle \pi_1(a(-1)) {V}_+^{l}W_{\R} \otimes {V}_+^{l}W_{\R}  \rangle_1&=(4\pi)^{-2l}\frac{\Gamma(l+1)\Gamma(2\kappa'+l)}{\Gamma(2\kappa')}\left \langle \pi_1(a(-1)) W_{\R} \otimes W_{\R}   \right \rangle_1\\
&=(4\pi)^{-2\kappa'-2l}\Gamma(l+1)\Gamma(2\kappa'+l).
\end{split}
\end{align}
Note that
\begin{align*}
\omega_{1} \left ({t}(a)k_{\theta}a(-1) , \left [a(-1) ,1 \right ] \right )\varphi_{l}(1)&=2^{l}e^{l \sqrt{-1}\theta}a^{l+2}e^{-2\pi a^2}.
\end{align*}
By \cite[Lemma 2.1]{Ikeda1998}, 
\begin{align}\label{E:R3}
\begin{split}
&\tilde{B}_1( \omega_{1}\left ( a(-1),\left[a(-1),1 \right ]     \right ) \varphi_{2\kappa'+2l}, \pi_1(a(-1)) {V}_+^{l}W_{\R} )\\
&=2^{2\kappa'+2l}\int_{\R^{\times}} \sum_{j=0}^{l}a^{4\kappa'+2l+2j}e^{-4\pi a^2}(-4\pi)^{j-l}\frac{\Gamma(2\kappa'+l)}{\Gamma(2\kappa'+j)}{l \choose j}d^{\times}a\\
&=2^{2\kappa'+2l}(4\pi)^{-2\kappa}\Gamma(2\kappa'+2l)\sum_{j=0}^{l}(-1)^j{l \choose j}\frac{\Gamma(2\kappa'+l+j)}{\Gamma(2\kappa'+j)}\\
&=(-1)^l2^{2\kappa'+2l}(4\pi)^{-2\kappa'-2l}\Gamma(l+1)\Gamma(2\kappa'+l).
\end{split}
\end{align}
It follows from (\ref{E:R1}), (\ref{E:R2}), and (\ref{E:R3}) that $C_l=(-1)^{l}2^{2\kappa'+2l}$ and $\<\mbox{ },\mbox{ }\>_1=B_1$.

Note that $\hat{\varphi}_{\C}=\varphi_{\C}$ and
\begin{align*}
(H,0)\cdot \varphi_{\C} &= 2\kappa \varphi_{\C} ,\quad  (X,0)\cdot \varphi_{\C}  = 0,\\
(0,H)\cdot \varphi_{\C} &= 2\kappa \varphi_{\C} , \quad (0,X)\cdot \varphi_{\C}  = 0.
\end{align*}
Here $H,X \in \frak{gl}_2(\C)\otimes_{\R}\C$ are defined by
\begin{align*}
H= \bp -\sqrt{-1} & 0 \\ 0  & \sqrt{-1} \ep \otimes \sqrt{-1} ,\quad X= \frac{1}{2}\bp 0 & -1 \\ 1 & 0 \ep\otimes 1 + \frac{1}{2}\bp  0 & \sqrt{-1} \\ \sqrt{-1} & 0 \ep\otimes \sqrt{-1}.
\end{align*}
Therefore, there exits a constant $C$ such that
$$\theta_2(\varphi_{\C},\sigma_2(w)W_{\C}) = C\cdot W_{\C}\otimes W_{\C}.$$
By \cite[Lemma 6.6]{Ichino2005}, and \cite[6.653.2]{Table2000}
\begin{align*}
&\theta_2(\varphi_{\C},\sigma_2(w)W_{\C})(1)\\
&=(2\kappa+1)\sum_{n=0}^{2\kappa}{2\kappa \choose n}(-\sqrt{-1})^n \int_{{\rm SO(2)}}e^{-2n\sqrt{-1}\theta}d\theta \int_{{\rm SU(2)}}\overline{\alpha}^{n}{\beta}^{2\kappa-n}\overline{\beta}^{2\kappa} dk \\
&\times \int_{\C}\int_{\R_{>0}} r^{-2}e^{-2\pi(r^2+r^{-2}+r^{-2}|x|^2)+2\pi \sqrt{-1}\,{\rm tr}_{\C/\R}(x)}K_{\kappa-n}(4\pi r^2)d^{\times}rdx\\
&=\int_{\R_{>0}}e^{-4\pi r^2-2\pi r^{-2} }K_{\kappa}(4\pi r^2)d^{\times}r\\
&=K_{\kappa}(4\pi)^2.
\end{align*}
We conclude that $C=1$ and 
\begin{align*}
\theta_2\left ( \omega_{2} \left(\begin{pmatrix}  0 & 1 \\ 1 & 0 \end{pmatrix},\left[a(-1){\bf t}_{\infty},{\bf t}_{\infty}\right]\right )\varphi_{\C},\sigma_2(a(-1))W_{\C} \right) = \sigma_2(a(-1){\bf t}_{\infty}) W_{\C} \otimes \sigma_2({\bf t}_{\infty})W_{\C}.
\end{align*}
By \cite[6.576.4]{Table2000}, 
\begin{align}\label{E:C1}
\begin{split}
&\langle \sigma_2(a(-1){\bf t}_{\infty})W_{\C}\otimes \sigma_2({\bf t}_{\infty})W_{\C}\rangle_2\\
&=2^{-2\kappa+1}a^{2\kappa+2}\sum_{n=0}^{2\kappa}\sum_{m=0}^{2\kappa}{2\kappa \choose n}{2\kappa \choose m} \int_{\R / 2\pi \Z}e^{(2\kappa-m-n)\sqrt{-1}\theta}d{\theta}\int_{\R_{>0}}r^{2\kappa+2}K_{\kappa-n}(4\pi r )K_{\kappa-m}(4\pi r)d^{\times}r\\
&=2^{-2\kappa+2}\pi a^{2\kappa+2}\sum_{n=0}^{2\kappa}{2\kappa \choose n}^2 \int_{\R_{>0}}r^{2\kappa+2}K_{\kappa-n}(4\pi r)^2d^{\times}r\\
&=2^{-2\kappa-3}\pi^{-2\kappa-1}(2\kappa+1)^{-1}\Gamma(\kappa+1)^2.
\end{split}
\end{align}
Note that
\begin{align*} 
W_{\C}\left (a(-1){t}(r)ka(-1) \right )&=r^{2\kappa+2}\sum_{n=0}^{2\kappa}{2\kappa \choose n}(\sqrt{-1})^{-n}\alpha^{2\kappa-n}\overline{\beta}^{n}K_{\kappa-n}(4\pi r^2),\\
{\omega}_{2}\left ( d(-1)k\begin{pmatrix}  0 & 1 \\ 1 & 0 \end{pmatrix},1\right )\varphi_{\C} (x) &=(2\kappa+1)\sum_{n=0}^{2\kappa}{2\kappa \choose n}(\sqrt{-1})^{2\kappa+n}\overline{\alpha}^{2\kappa-n}\beta^{n}\overline{{x}}_2^{2\kappa-n}{x}_3^ne^{-2\pi\,{\rm tr}({x}{}^t\overline{x})}.
\end{align*}
for $r \in \R_{>0}$ and $k=\begin{pmatrix} \alpha&\beta \\ -\overline{\beta} & -\alpha  \end{pmatrix} \in {\rm SU}(2).$ Therefore, by \cite[Lemma 6.6]{Ichino2005} and \cite[6.621.3]{Table2000} 
\begin{align}\label{E:C2}
\begin{split}
&\tilde{B}_2\left ( \omega_{2} \left(\begin{pmatrix}  0 & 1 \\ 1 & 0 \end{pmatrix},\left[a(-1){\bf t}_{\infty},{\bf t}_{\infty}\right]\right )\varphi_{\C},\sigma_2(a(-1))W_{\C} \right)\\
&=(2\kappa+1)\sum_{n=0}^{2\kappa}\sum_{m=0}^{2\kappa}{2\kappa \choose n}{2\kappa \choose m}(\sqrt{-1})^{n-m}\int_{\R / 2\pi \Z}e^{(2n-2m)\sqrt{-1}\theta}d\theta \int_{{\rm SU}(2)}\overline{\alpha}^{2\kappa-n}\beta^{n}\alpha^{2\kappa-m}\overline{\beta}^{m}dk\\
&\times \int_{\R_{>0}}r^{4\kappa+2}e^{-4\pi r^2}K_{\kappa-m}(4\pi r^2)d^{\times}r\\
&=2^{-1}  \sum_{n=0}^{2\kappa} {2\kappa \choose n} \int_{\R_{>0}}r^{2\kappa+1}e^{-4\pi r}K_{\kappa-n}(4\pi r)d^{\times}r\\
&=2^{-2\kappa-3}\pi^{-2\kappa-1}(2\kappa+1)^{-1}\Gamma(\kappa+1)^2
\end{split}
\end{align}
It follows from (\ref{E:C1}) and (\ref{E:C2}) that $\langle \mbox{ },\mbox{ } \rangle_2={B}_{2}.$ This completes the proof.
\end{proof}

\subsection{Local trilinear period integral and local zeta integral}
In \cite[Proposition 5.1]{Ichino2008}, Ichino established an equality between the local trilinear period integral and the local zeta integral of Piatetski-Shapiro and Rallis \cite{PSR1987}. In this section, specializing to our case, we calculate the corresponding local zeta integral explicitly and deduce the value $\mathcal{I}(\Pi_{\infty})$ from it. 

Let $$\G = \{g\in {\rm R}_{E/\R}\GL_2 \mbox{ }\vert \mbox{ }\nu(g) \in {\mathbb G}_m \}.$$
%Then $$\G(\R)=\{(g_1,g_2)\in \GL_2(\R)\times \GL_2(\C)\mbox{ } \left \vert \mbox{ } \det(g_1)=\det(g_2) \right . \}.$$
We regard the space $E^2$ of row vectors as a symplectic space over $\R$ with nondegenerate antisymmetric bilinear form
$$\< {x},{y}   \>={\rm tr}_{E/\R}({x}_1{y}_2-{x}_2{y}_1),$$
for ${x}=({x}_1,{x}_2),{y}=({y}_1,{y}_2)\in E^2$. We choose a basis $\{e_1,e_2,e_3,e_1',e_2',e_3'   \}$ of $E^2$ over $\R$ as follows:
\begin{align*}
e_1&= ((0,1),(0,0) ), \quad e_2= ((0,\sqrt{-1}),(0,0) ),\quad e_3=((1,0),(0,0)),\\
e_1'&=\left ((0,0),\left (0,\frac{1}{2}\right ) \right ), \quad e_2'= \left ((0,0),\left (0,\frac{-\sqrt{-1}}{2} \right ) \right),\quad e_3'=((0,0),(1,0)).
\end{align*}
With respect to this basis, we have an embedding 
\begin{align}\label{E:embedding}
\begin{split}
\G(\R) &\longrightarrow \GSp_6(\R)\\
\left (\begin{pmatrix}a' & b' \\ c' & d' \end{pmatrix},\begin{pmatrix}a & b \\ c & d \end{pmatrix} \right ) &\longmapsto \left (\begin{array}{cccccc} a_1 & a_2 & 0 & 2b_1 & -2b_2 & 0 \\
-a_2 & a_1 & 0 & -2b_2 &-2b_1 & 0 \\
0 & 0 &a'&0&0&b'\\
c_1/2 & c_2/2& 0& d_1 & -d_2&0\\
c_2/2& -c_1/2& 0 & d_2 & d_1 & 0\\
0&0&c'&0&0&d'
   \end{array} \right ),
\end{split}
\end{align}
here $a=a_1+\sqrt{-1}a_2,b=b_1+\sqrt{-1}b_2,c=c_1+\sqrt{-1}c_2,d=d_1+\sqrt{-1}d_2.$ Let $\omega$ be the Weil representation of ${\rm G}(\Sp_6 \times {\rm O}(V))(\R)$ on $S(V^3(\R))$ with respect to $\psi_1$. We have an isomorphism 
\begin{align*}
\R^3 &\longrightarrow E\\
({x}_1,{x}_2,{x}_3) & \longmapsto ({ x}_1,{ x}_2+\sqrt{-1}{x}_3).
\end{align*}
We identify $S(V^3(\R))$ with $S(V(E))$ by this isomorphism. Then, 
$$\omega((g_1,g_2),h)(\varphi_1 \otimes \varphi_2)=\omega_{1}(g_1,h)\varphi_1 \otimes \omega_{2}(g_2,h)\varphi_2$$
for $(g_1,g_2) \in \G(\R)$ and $h \in {\rm GO}(V)(\R)$ such that $\det(g_1)=\det(g_2)=\nu(h)$. Recall $\varphi_E \in S(V^3(\R))$ is defined by 
$$\varphi_E=\varphi_{\R}\otimes \omega_{2}(1,({\bf t}_{\infty},{\bf t}_{\infty}))\varphi_{\C}.$$ 

Let 
$${\rm P} = \left \{   \bp  * & * \\ {\bf 0}_3 & *   \ep\in \GSp_6 \right \}$$ 
be the standard Siegel parabolic subgroup of $\GSp_6$ and 
$$\rho_{{\rm P}}\left(\bp  A & X \\ {\bf 0}_3 & \nu {}^tA^{-1}   \ep\right)=|\det(A)|_{\R}^2 |\nu|_{\R}^{-3}.$$
For $s \in \C$, let ${\bf I}(s) = {\rm Ind}_{{\rm P}(\R)}^{\GSp_6(\R)}(\rho_{{\rm P}}^{s})$ be a degenerate principal series representation of $\GSp_6(\R)$. Define
\begin{align}\label{E:SW sections}
f_{\varphi_E}^{(0)}(g) = |\nu(g)|_{\R}^{-3}\omega\left ( \begin{pmatrix} {\bf 1}_3 & {\bf 0}_3 \\ {\bf 0}_3 & \nu(g)^{-1}{\bf 1}_3 \end{pmatrix}g,1\right )\varphi_E(0)
\end{align}
for $g \in \GSp_6(\R)$. Then $f_{\varphi_E}^{(0)}$ belongs to the space of ${\bf I}(0).$ Let $$K=\left \{ \left . \begin{pmatrix}\alpha & \beta \\ -\beta & \alpha \end{pmatrix} \mbox{ } \right \vert \mbox{ } \alpha + \sqrt{-1}\beta \in {\rm U}(3)\right \}$$ be a maximal compact subgroup of $\GSp_6(\R)$ and $K'=\gamma K\gamma^{-1}$, where $$\gamma = \bp 0 & 0 & 0 & 1 & 0 & 0 \\
0 & 0 & 0 & 0 &1 & 0 \\
0 & 0 &0&0&0&1/\sqrt{2}\\
-1/2 & 0& 0& 0 & 0&0\\
0& -1/2& 0 & 0 & 0 & 0\\
0&0&-1/\sqrt{2}&0&0&0
   \ep \in \GSp_6(\R).$$ Then ${\rm SO}(2)\times {\rm SU}(2)  \subset K'$ via the embedding (\ref{E:embedding}). For $s \in \C$, we extend $f_{\varphi_E}^{(0)}$ to a holomorphic section $f_{\varphi_E}^{(s)}$ in the space of  ${\bf I}(s)$ so that its restriction to $K'$ is equal to $f_{\varphi_E}^{(0)}.$ 
Let 
$${\bf U}_0 = \{ u(x)\mbox{ } \vert \mbox{ }x \in {\rm R}_{E/\R}{\mathbb G}_a, \mbox{ }{\rm tr}_{E / \R}(x)=0 \}$$
and
$$\eta = \bp 0 & 0 & 0 & -1 & 0 & 0 \\
0 & 1 & 0 & 0&0 & 0 \\
0 & 0 &1&0&0&0\\
1 & 0& 1& 0 & 0&0\\
0& 0& 0 & 0 & 1 & 0\\
0&0&0&-1&0&1
   \ep \in \Sp_6(\Z).$$

Define the local zeta integral (cf.\,\cite{PSR1987})
\begin{align}\label{E:local zeta integral}
Z_{\infty}(s, W_E,f_{\varphi_E})=\int_{\R^{\times}{\bf U}_0(\R)\backslash \G(\R) }f_{\varphi_E}^{(s)}(\eta g)W_E(a(-1)g)dg.
\end{align}
Note that the integral is absolutely convergent for ${\rm Re}(s)>-1/2$ (cf.\,\cite[Lemma 2.1]{Ikeda1992}). 
The measure on $\R^{\times}\backslash \G(\R)$ is defined by 
$$\int_{\R^{\times}\backslash \G(\R)}f(g)dg= \int_{\SL_2(E)}f\left ( g\right )dg+ \int_{\SL_2(E)}f\left ( d(-1)g\right )dg$$
for $f \in L^1(\R^{\times}\backslash \G(\R))$, here the Haar measure $dg$ on $\SL_2(E)$ is the product measure of $\SL_2(\R)$ and $\SL_2(\C)$. Then, 
$$\int_{\R^{\times}{\bf U}_0(\R)\backslash \G(\R)}f(g)dg= \int_{{\bf U}_0(\R)\backslash \SL_2(E)}f\left ( g\right )dg+ \int_{{\bf U}_0(\R)\backslash \SL_2(E)}f\left ( d(-1)g\right )dg$$
for $f \in L^1(\R^{\times}{\bf U}_0(\R)\backslash \G(\R))$, here 
$$\int_{{\bf U}_0(\R)\backslash \SL_2(E)}f\left ( g\right )dg=\int_{U(E)\backslash \SL_2(E)}\int_{\R}f\left ( (u(x),1)g\right )dxdg.$$

\begin{lemma}\label{L:evalutation of holomorphic section}
For $x \in \R$, $a_1,a_2 \in \R_{>0}$, $k_{\theta} \in {\rm SO}(2)$, and $k=\begin{pmatrix} \alpha&\beta \\ -\overline{\beta} & -\alpha  \end{pmatrix} \in {\rm SU}(2)$, we have
\begin{align*}
&f_{\varphi_E}^{(s)}(\eta ({u}(x){t}(a_1)k_{\theta},{t}(a_2)k))\\&=\pi^{-2\kappa}\Gamma(2\kappa+2)e^{-2\kappa \sqrt{-1}\theta}\sum_{n=0}^{2\kappa}{2\kappa \choose n}(\sqrt{-1})^{n}\alpha^{n}\overline{\beta}^{2\kappa-n}\\
&\times a_1^{2s+2\kappa+2}a_2^{4s+2\kappa+4}(a_1^2+2a_2^2+\sqrt{-1}x)^{-s-2\kappa-1}(a_1^2+2a_2^2-\sqrt{-1}x)^{-s-1}.
\end{align*}
\end{lemma}

\begin{proof}
By the Iwasawa decomposition of $\eta({ t}(a_1),{u}(x){t}(a_2))$ with respect to $K'$, we have
\begin{align}\label{E:P1}
f_{\varphi_E}^{(s)}(\eta ({u}(x){t}(a_1)k_{\theta},{t}(a_2)k))=a_1^{2s}a_2^{4s}(x^2+(a_1^2+2a_2^2)^2)^{-s}f_{\varphi_E}^{(0)}(\eta ({t}(a_1)k,{u}(x){t}(a_2)k_{\theta})).
\end{align}
By the Bruhat decomposition of $\eta$ (cf.\,\cite[Lemma 12.1]{Ichino2005}),
\begin{align}\label{E:P2}
f_{\varphi_E}^{(0)}(\eta g)=\int_{V(\R)}\omega(g,1)\varphi_E(y)dy
\end{align}
for $g \in \Sp_6(\R)$. Here $dy$ is the Haar measure self-dual with respect to the pairing $\psi_1((x,y))$.
Therefore, by (\ref{E:P1}), (\ref{E:P2}), and \cite[Lemma 6.9]{IchinoIkeda2008} 
\begin{align*}
&f_{\varphi_E}^{(s)}(\eta ({u}(x){t}(a_1)k_{\theta},{t}(a_2)k))\\
&=a_1^{2s+2}a_2^{4s+4}(x^2+(a_1^2+2a_2^2)^2)^{-s}\int_{V(\R)}\omega_{1}(k_{\theta},1)\varphi_{\R}(a_1y)\omega_{2}(k,1)\varphi_{\C}(a_2{\bf t}_{\infty}^{-1}y{\bf t}_{\infty})\psi_1(x \det (y))dy\\
&=2^{-2\kappa}(2\kappa+1)e^{-2\kappa \sqrt{-1}\theta}a_1^{2s+2\kappa+2}a_2^{4s+2\kappa+4}(x^2+(a_1^2+2a_2^2)^2)^{-s}\sum_{n=0}^{2\kappa}{2\kappa \choose n}(\sqrt{-1})^{n}\alpha^{n}\overline{\beta}^{2\kappa-n}\\
&\times\int_{\R^4}((y_1+y_4)^2+(y_2+y_3)^2)^{2\kappa}e^{-\pi(a_1^2+2a_2^2)(y_1^2+y_2^2+y_3^2+y_4^2)}e^{-2\pi \sqrt{-1}x(y_1y_4+y_2y_3)}dy_1dy_2dy_3dy_4\\
&=\pi^{-2\kappa}\Gamma(2\kappa+2)e^{-2\kappa \sqrt{-1}\theta}\sum_{n=0}^{2\kappa}{2\kappa \choose n}(\sqrt{-1})^{n}\alpha^{n}\overline{\beta}^{2\kappa-n}\\
&\times a_1^{2s+2\kappa+2}a_2^{4s+2\kappa+4}(a_1^2+2a_2^2+\sqrt{-1}x)^{-s-2\kappa-1}(a_1^2+2a_2^2-\sqrt{-1}x)^{-s-1}.
\end{align*}
This completes the proof.
\end{proof}

\begin{prop}\label{P:local zeta integral}
We have
\begin{align*}
&Z_{\infty}(s,W_E,f_{\varphi_E})\\
&=\pi^{-s-4\kappa+3/2}2^{-6s-8\kappa-2}\frac{\Gamma(2\kappa+1)\Gamma(\kappa+\kappa')}{\Gamma(s+2\kappa+1)\Gamma(s+2\kappa+1/2)}\sum_{j=0}^{2m}\sum_{n=0}^{2\kappa}(-1)^{j}{2m \choose j}
{2\kappa \choose n}\\
&\times \frac{\Gamma(2s+\kappa+1+n)\Gamma(2s+3\kappa+1-n)\Gamma(s+\kappa+\kappa'+j)\Gamma(s+\kappa'+j+n)\Gamma(s+2\kappa+\kappa'+j-n)}{\Gamma(2\kappa'+j)\Gamma(2s+2\kappa+\kappa'+j-n+1)\Gamma(2s+\kappa'+j+n+1)}\\
&\times {}_{3}F_{2}(s+1,j-2m,2s+\kappa+\kappa'+j+1/2;2s+\kappa'+j+n+1,2s+2\kappa+\kappa'+j-n+1;1).
\end{align*}
\end{prop}

\begin{proof}
For $\lambda ,\mu, \nu \in \C$, and $n_1,n_2 \in \Z$, put
\begin{align*}
I(\lambda,\mu,\nu;n_1,n_2)&=\int_{\R_{>0}}\int_{\R_{>0}}\int_{\R}a_1^{s+n_1}a_2^{2s+n_2}(a_1+a_2+\sqrt{-1}x)^{-s-\lambda}(a_1+a_2-\sqrt{-1}x)^{-s-\mu}\\
&\times e^{-a_1+\sqrt{-1}x}K_{\nu}(a_2)dxd^{\times}a_1d^{\times}a_2.
\end{align*}
By \cite[Lemmas 12.8 and 12.9]{Ichino2005}, 
\begin{align}\label{E:Ichino's calculation}
\begin{split}
&I(\lambda,\mu,\nu;n_1,n_2)\\
&= 2^{-3s-n_1-n_2+1}\pi^{3/2}\frac{\Gamma(2s+\nu+n_2)\Gamma(2s-\nu+n_2)\Gamma(s+n_1)\Gamma(s-\lambda-\mu+\nu+n_1+n_2+1)}{\Gamma(2s+n_2+1/2)\Gamma(2s-\lambda+\nu+n_1+n_2+1)\Gamma(s+\lambda)}\\
&\times {}_3F_2(s+\mu,2s+\nu+n_2,\nu+1/2 ; 2s+n_2+1/2 , 2s-\lambda+\nu+n_1+n_2+1).
\end{split}
\end{align}
By \cite[Lemma 6.6]{Ichino2005}, and  Lemma \ref{L:evalutation of holomorphic section},
\begin{align*}
&Z_{\infty}(s,W_E,f_{\varphi_E})\\
%&=\int_{U(E)\backslash \SL_2(E) }f_{\varPsi}^{(0)}(\eta g)(W_{\C}\otimes \widetilde{V}_+^mW_{\R})\left (\left( \begin{pmatrix}1 & 0 \\ 0 & 1 \end{pmatrix},{u}(x)\right )g(w,1) \right)dg\\
%&= \int_{\R_{>0}}a_1^{-4}d^{\times}a_1\int_{{\rm SO}(2)}d\theta_1\int_{{\rm SU}(2)}dk\int_{\R}dx\int_{\R^{\times}}|a_2|_{\R}^{-2}d^{\times}a_2\int_{{\rm SO}(2)}d\theta \\
%&\times f_{\varPsi}^{(0)}\left (\eta  \left ({t}(a_1 e^{\sqrt{-1}\theta_1})k,{u}(x){t}(a_2)k_{\theta}\right )\right) W_{\C}( {t}(a_1 e^{\sqrt{-1}\theta_1})k w)W_{\R}({u}(x){ t}(a_2)k_{\theta})\\
%&=\int_{\R_{>0}}a_1^{-4}d^{\times}a_1\int_{{\rm SO}(2)}d\theta_1\int_{{\rm SU}(2)}dk\int_{\R}dx\int_{\R^{\times}}|a_2|_{\R}^{-2}d^{\times}a_2\int_{{\rm SO}(2)}d\theta \\
%&\times \Gamma(2\kappa+2)\pi^{-2\kappa}e^{-2\kappa \sqrt{-1}\theta}\sum_{n=0}^{2\kappa}e^{(2n-2\kappa)\sqrt{-1}\theta_1}{2\kappa \choose n}(\sqrt{-1})^{n}\alpha^{n}\overline{\beta}^{2\kappa-n} \\
%&\times a_1^{4s+2\kappa+4}a_2^{2s+2\kappa+2}(2a_1^2+a_2^2+\sqrt{-1}x)^{-s-2\kappa-1}(2a_1^2+a_2^2-\sqrt{-1}x)^{-s-1} \\
%&\times a_1^{2\kappa+2}\sum_{n'=0}^{2\kappa}e^{(2\kappa-2n')\sqrt{-1}\theta_1}{2\kappa \choose n'}(-\sqrt{-1})^{n'}\overline{\alpha}^{n'}{\beta}^{2\kappa-n'}K_{\kappa-n'}(4\pi a_1^2)\\
%&\times 2 e^{2\pi \sqrt{-1}x} e^{2\kappa \sqrt{-1}\theta}e^{-2\pi a_2^2} \sum_{j=0}^{m}a_2^{2\kappa'+2j}(-4\pi)^{j-m}\frac{\Gamma(2\kappa'+m)}{\Gamma(2\kappa'+j)}{m \choose j} \\
&=2\pi^{-2\kappa}\Gamma(2\kappa+1) \sum_{j=0}^{2m}\sum_{n=0}^{2\kappa}{2m \choose j}{2\kappa \choose n}\frac{\Gamma(\kappa+\kappa')}{\Gamma(2\kappa'+j)}(-4\pi)^{j-2m} \\
&\times \int_{\R_{>0}}\int_{\R_{>0}} \int_{\R} a_1^{2s+2\kappa+2\kappa'+2j}a_2^{4s+4\kappa+2}(a_1^2+2a_2^2+\sqrt{-1}x)^{-s-2\kappa-1}(a_1^2+2a_2^2-\sqrt{-1}x)^{-s-1} \\
&\times e^{-2\pi a_1^2+2\pi \sqrt{-1}x} K_{\kappa-n}(4\pi a_2^2)d^{\times}a_1 d^{\times}a_2 dx\\
&=\pi^{-s-4\kappa}\Gamma(2\kappa+1) \sum_{j=0}^{2m}\sum_{n=0}^{2\kappa}(-1)^j{2m \choose j}{2\kappa \choose n}\frac{\Gamma(\kappa+\kappa')}{\Gamma(2\kappa'+j)}\\&\times2^{-3s-4\kappa-2m+j-2}  I(2\kappa+1,1,\kappa-n;\kappa+\kappa'+j,2\kappa+1).\\ 
\end{align*}
By (\ref{E:Ichino's calculation}),
\begin{align*}
&I(2\kappa+1,1,\kappa-n;\kappa+\kappa'+j,2\kappa+1)\\
&=2^{-3s-3\kappa-\kappa'-j}\pi^{3/2}\frac{\Gamma(2s+\kappa+n+1)\Gamma(s+\kappa+\kappa'+j)\Gamma(s+2\kappa+\kappa'+j-n)\Gamma(2s+3\kappa-n+1)}{\Gamma(s+2\kappa+1)\Gamma(2s+2\kappa+3/2)\Gamma(2s+2\kappa+\kappa'+j-n+1)}\\
&\times {}_{3}F_{2}(s+1,2s+3\kappa-n+1,\kappa-n+1/2 ; 2s+2\kappa+3/2,2s+2\kappa+\kappa'+j-n+1 ; 1).
\end{align*}
By the two-term relation for ${}_3F_2$ in \cite[(4.3.1.3)]{Slater}, we have
\begin{align*}
&_{3}F_{2}(s+1,2s+3\kappa-n+1,\kappa-n+1/2 ; 2s+2\kappa+3/2,2s+2\kappa+\kappa'+j-n+1 ; 1)\\
&=\frac{\Gamma(s+\kappa'+j+n)\Gamma(2s+2\kappa+3/2)}{\Gamma(s+2\kappa+1/2)\Gamma(2s+\kappa'+j+n+1)}\\
&\times _{3}F_{2}(s+1,j-2m,2s+\kappa+\kappa'+j+1/2 ; 2s+\kappa'+j+n+1, 2s+2\kappa+\kappa'+j-n+1 ; 1).
\end{align*}
This completes the proof.
%{\color{red}Also we have
%\begin{align*}
%&_{3}F_{2}(s+1,2s+3\kappa-n+1,\kappa-n+1/2 ; 2s+2\kappa+3/2,2s+2\kappa+\kappa'+j-n+1 ; 1)\\
%&=\frac{\Gamma(2s+2\kappa+3/2)\Gamma(2s+2\kappa+\kappa'+j-n+1)\Gamma(s+\kappa'+j+n)}{\Gamma(2s+3\kappa-n+1)\Gamma(s+\kappa+\kappa'+j+1/2)\Gamma(2s+\kappa'+j+n+1)}\\
%&\times {}_{3}F_{2}(n-\kappa+1/2, j-m, s+\kappa'+j+n; s+\kappa+\kappa'+j+1/2, 2s+\kappa'+j+n+1; 1)\\
%&=\frac{\Gamma(2s+2\kappa+3/2)\Gamma(2s+2\kappa+\kappa'+j-n+1)}{\Gamma(2s+3k-n+1)\Gamma(n-\kappa+1/2)\Gamma(j-m)}\\
%&\times \sum_{i=0}^{m-j}\frac{\Gamma(n-\kappa+1/2+i)\Gamma(j-m+i)\Gamma(s+\kappa'+j+n+i)}{i! \Gamma(s+\kappa+\kappa'+j+1/2+i)\Gamma(2s+\kappa'+j+n+1+i)}.
%\end{align*}}
%Therefore, we have
%\begin{align*}
%&Z_{\infty}(s)\\
%&=\pi^{-s-4\kappa+3/2}2^{-6s-8\kappa-2}\frac{\Gamma(2\kappa+1)\Gamma(\kappa+\kappa')}{\Gamma(s+2\kappa+1)\Gamma(s+2\kappa+1/2)}\sum_{j=0}^{m}\sum_{n=0}^{2\kappa}(-1)^{j}{m \choose j}
%{2\kappa \choose n}\\
%&\times \frac{\Gamma(2s+\kappa+1+n)\Gamma(2s+3\kappa+1-n)\Gamma(s+\kappa+\kappa'+j)\Gamma(s+\kappa'+j+n)\Gamma(s+2\kappa+\kappa'+j-n)}{\Gamma(2\kappa'+j)\Gamma(2s+2\kappa+\kappa'+j-n+1)\Gamma(2s+\kappa'+j+n+1)}\\
%&\times {}_{3}F_{2}(s+1,j-m,2s+\kappa+\kappa'+j+1/2;2s+\kappa'+j+n+1,2s+2\kappa+\kappa'+j-n+1;1).
%\end{align*}
\end{proof}

\begin{prop}\label{P:archimedean local period integral}
We have 
$$\mathcal{I}(\Pi_{\infty}) = 2^{-6\kappa+6\kappa'-2}\pi^{-4m}(2\kappa+1)C_\infty(\kappa,\kappa').$$
Here $C_\infty(\kappa,\kappa')$ is the rational number defined in (\ref{E:archimedean rational number}).
\end{prop}

\begin{proof}
Note that
\begin{align*}
L(s,\Pi_{\infty})&=\zeta_{\C}(s+\kappa+\kappa'-1/2)\zeta_{\C}(s+\kappa'-1/2)^2\zeta_{\C}(s+2m+1/2),\\
L(s,\Pi_{\infty},{\rm Ad})&=\zeta_{\R}(s)\zeta_{\R}(s+1)^2\zeta_{\C}(s+2\kappa'-1)\zeta_{\C}(s+\kappa)^2,\\
\<\pi_1(a(-1))W_{\R},W_{\R}\>_1&= (4\pi)^{-2\kappa'}\Gamma(2\kappa'),\\
\<\sigma_2(a(-1){\bf t}_{\infty})W_{\C},\sigma_2({\bf t}_{\infty})W_{\C}\>_2&=2^{-2\kappa-3}\pi^{-2\kappa-1}(2\kappa+1)^{-1}\Gamma(\kappa+1)^2.
\end{align*}
It suffices to show that 
\begin{align*}
&\int_{\R^{\times} \backslash \GL_2(\R)}\<\pi_1(g){V}_+^{2m}W_{\R},{V}_+^{2m}W_{\R}\>_1\< \sigma_2(g{\bf t}_{\infty})W_{\C},\sigma_2({\bf t}_{\infty})W_{\C}\>_2 dg\\
&=2^{-6\kappa-4m-4}\pi^{-4\kappa}\Gamma(\kappa+\kappa')\Gamma(\kappa')^2\Gamma(2m+1)C_\infty(\kappa,\kappa').
\end{align*}
Let $\Phi(\mbox{ };\varphi_E,W_E)$ be a matrix coefficient of $\Pi_{\infty}$ defined by
\begin{align}\label{E:matrix coeff by Weil rep}
\Phi((g_1,g_2);\varphi_E,W_E) = \tilde{B}_1(\omega_{1}(1,[g_1,1])\varphi_{\R},{V}_+^{2m}\varphi_{\R})\tilde{B}_2(\omega_{2}(1,[g_2{\bf t}_{\infty},{\bf t}_{\infty}])\varphi_{\C},\sigma_2(w)W_{\C}).
\end{align}
By \cite[Proposition 5.1]{Ichino2008} and Proposition \ref{P:local zeta integral},
\begin{align*}
\int_{\R^{\times} \backslash \GL_2(\R)}\Phi(g;\varphi_E,W_E)dg&=2^{-1}\zeta_{\R}(2)\cdot Z_{\infty}(0,W_E,f_{\varphi_E})\\
&=2^{-4\kappa-4m-4} \pi^{-4\kappa}\Gamma(\kappa+\kappa')\Gamma(\kappa')^2\Gamma(2m+1)C_\infty(\kappa,\kappa').
\end{align*} 
Note that there is an extra factor $2^{-1}$, which is due to the normalization of measures, comparing with \cite[Proposition 5.1]{Ichino2008}. On the other hand, by Lemma \ref{L:comparison of invariant pairings}, 
\begin{align*}
\int_{\R^{\times} \backslash \GL_2(\R)}\<\pi_1(g){V}_+^{2m}W_{\R},{V}_+^{2m}W_{\R}\>_1\< \sigma_2(g{\bf t}_{\infty})W_{\C},\sigma_2({\bf t}_{\infty})W_{\C}\>_2 dg = 2^{-2\kappa}\int_{\R^{\times} \backslash \GL_2(\R)}\Phi(g;\varphi_E,W_E)dg.
\end{align*}
This completes the proof.
\end{proof}
To complete the proof of Lemma \ref{P:local period integral archimedean}, it remains to show that $C_\infty(\kappa,\kappa')\neq 0$. We introduce yet another local trilinear form of $\Pi_{\infty}$ by the Rankin-Selberg local zeta integral in the non-split $\C /\R$ case (cf.\,\cite[\S\,3]{CCI2018}). The corresponding Rankin-Selberg local zeta integral was considered by Ghate in a classical context in \cite[\S 6]{Ghate1999} with a conjectural formula, which is proved by Lanphier-Skogman and Ochiai in \cite{LS2014}. 

Let $\mathcal{V}$ be a model of $\pi_1$ realized as the irreducible subspace of 
$${\rm Ind}_{{\bf B}(\R)}^{\GL_2(\R)}( \vert \mbox{ }\vert_{\R}^{\kappa'-1/2}\boxtimes \vert \mbox{ }\vert_{\R}^{-\kappa'+1/2}).$$ 
Let $V =\mathcal{V}\otimes \mathcal{W}(\sigma_2,\psi_2)$ be a model of $\Pi_{\infty}$. Define $\Psi_{\infty} \in {\rm Hom}_{\GL_2(\R)}(\Pi_{\infty},\C)$ by 
\begin{align}\label{E:Rankin-Selberg local zeta integral}
\Psi_{\infty}(f\otimes W) = \int_{\R^{\times}U(\R) \backslash \GL_2(\R)}W(a(\sqrt{-1})g)f(g)dg
\end{align}
for $f \otimes W \in V.$
Note that the integral is absolutely convergent. Then $$\Psi_{\infty} \otimes \Psi_{\infty} \in {\rm Hom}_{\GL_2(\R)\times \GL_2(\R)}(\Pi_{\infty}\boxtimes \Pi_{\infty},\C).$$ Let $f_{\R} \in \mathcal{V}$ be the section of ${\rm SO}(2)$-type $2\kappa$ normalized so that $f_{\R}(1)=1$.

\begin{prop}\label{L:Rankin-Selberg integral archimedean}{\rm (}\cite[\S 6]{Ghate1999} {\rm and} \cite{LS2014}{\rm)} 
We have
$$\Psi_{\infty}(f_{\R} \otimes \sigma_2({\bf t}_{\infty})W_{\C}) = (-1)^{\kappa+m}2^{-2\kappa-\kappa'}\pi^{-\kappa-\kappa'}\frac{\Gamma(2m+1)\Gamma(\kappa-m)\Gamma(\kappa')}{\Gamma(m+1)}.$$
\end{prop}

\begin{proof}
By \cite[6.561.16]{Table2000}, 
\begin{align*}
\Psi_{\infty}(f_{\R} \otimes \sigma_2({\bf t}_{\infty})W_{\C})&=2^{-\kappa+1}(\sqrt{-1})^{\kappa}\sum_{n=0}^{2\kappa}{2\kappa \choose n}(\sqrt{-1})^{n}\int_{\R_{>0}}y^{\kappa+\kappa'}K_{\kappa-n}(4\pi y)d^{\times}y\\
&=2^{-2\kappa-\kappa'-1}\pi^{-\kappa-\kappa'}(\sqrt{-1})^{\kappa} \sum_{n=0}^{2\kappa}{2\kappa \choose n}(\sqrt{-1})^n \Gamma \left (\frac{\kappa'+n}{2} \right )\Gamma \left ( \frac{\kappa'+2\kappa-n}{2}\right ) \\
&=(-1)^{\kappa}2^{-2\kappa-\kappa'-1}\pi^{-\kappa-\kappa'} \sum_{\scriptstyle{n=0} \atop \scriptstyle{n \equiv \kappa ({\rm mod }\mbox{ }2)}}^{2\kappa}{2\kappa \choose n}(-1)^{(n-\kappa)/2} \Gamma \left (\frac{\kappa'+n}{2} \right )\Gamma \left ( \frac{\kappa'+2\kappa-n}{2}\right ).
\end{align*}
In \cite[\S\,6]{Ghate1999}, it was proved that the last combinatorial sum is non-zero. Moreover, in the notation of \cite[Theorem 1.1]{LS2014}, put $n=\kappa-1$, $m=0$, and $s= \kappa'-\kappa$, we have
\begin{align*}
&\sum_{\scriptstyle{n=0} \atop \scriptstyle{n \equiv \kappa ({\rm mod }\mbox{ }2)}}^{2\kappa}{2\kappa \choose n}\frac{(-1)^{(n-\kappa)/2}}{2} \Gamma \left (\frac{\kappa'+n}{2} \right )\Gamma \left ( \frac{\kappa'+2\kappa-n}{2}\right ) = (-1)^{m}\frac{\Gamma(2m+1)\Gamma(\kappa-m)\Gamma(\kappa')}{\Gamma(m+1)}.
\end{align*}
This completes the proof.
\end{proof}

\begin{corollary}\label{C:nonvanishing of rational number}
We have 
$$C_\infty(\kappa,\kappa') \neq 0.$$
\end{corollary}

\begin{proof}
Let $D$ be the division quaternion algebra over $\R$. Let $\Pi_{\infty}^{D}$ be the irreducible admissible representation of $D^{\times}(E)$ associated to $\Pi_{\infty}$ by the Jacquet-Langlands correspondence. By \cite{Loke2001}, 
\begin{align}\label{E:C1}
{\rm dim}_{\C}{\rm Hom}_{\GL_2(\R)}(\Pi_{\infty} , \C) =1 ,\quad  {\rm dim}_{\C}{\rm Hom}_{D^{\times}(\R)}(\Pi_{\infty}^{D} , \C)=0.  
\end{align}
Let 
$$\mathcal{I}_{\infty} \in {\rm Hom}_{\GL_2(\R)\times \GL_2(\R)}(\Pi_{\infty}\boxtimes \Pi_{\infty},\C)$$
be a trilinear form defined by
\begin{align}\label{E:trilinear form}
\mathcal{I}_{\infty}(W_1\otimes W_2, W_1'\otimes W_2') = \int_{\R^{\times}\backslash \GL_2(\R)}\<\pi_1(g)W_1,W_1'\>_1\<\sigma_2(g)W_2,W_2'\>_2dg.
\end{align}
 By Proposition \ref{L:Rankin-Selberg integral archimedean} and (\ref{E:C1}), the trilinear form $\Psi_\infty\otimes\Psi_\infty$ is non-zero and proportional to $\mathcal{I}_\infty$. By Proposition \ref{P:archimedean local period integral}, $C_\infty(\kappa,\kappa') \neq 0$ if and only if $$\mathcal{I}_\infty(V_+^{2m}W_\R\otimes \sigma_2({\bf t}_\infty)W_\C, V_+^{2m}W_\R\otimes \sigma_2({\bf t}_\infty)W_\C) \neq 0.$$
Since $\Psi_{\infty}(f_{\R} \otimes \sigma_2({\bf t}_{\infty})W_{\C})\neq 0$ by Proposition \ref{L:Rankin-Selberg integral archimedean}, we conclude that $C_\infty(\kappa,\kappa') \neq 0$ if and only if $\mathcal{I}_{\infty}$ is non-zero.

Let $(V',Q')$ be the quadratic space over $\R$ defined by $V' = D$ and $Q'[x]={\rm N}_{D / \R}(x).$ Consider the Weil representation of $\Sp_6(\R) \times {\rm O}(V')(\R)$ on $S(V'^3(\R))$ with respect to $\psi_1$. Let $\Theta({\bf 1})$ (resp.\,$\Theta'({\bf 1})$) be the maximal quotient of $S(V^3(\R))$ (resp.\,$S(V'^3(\R))$) which is trivial as Harish-Chandra module of ${\rm O}(V)(\R)$ (resp.\,${\rm O}(V')(\R)$) with respect to certain maximal compact subgroups. The representations $\Theta({\bf 1})$ and $\Theta'({\bf 1})$ can be realized as the images of the intertwining maps defined as in (\ref{E:SW sections}) 
\begin{align*}
S(V^3(\R)) &\longrightarrow {\bf I}(0),\quad \varphi\longmapsto f_\varphi ,\\
S(V'^3(\R)) &\longrightarrow {\bf I}(0),\quad \varphi\longmapsto f'_\varphi.
\end{align*}
By \cite[Theorem 4.12]{LZ1997}, 
\begin{align}\label{E:dichotomy}
{\bf I}(0) = \Theta({\bf 1}) \oplus \Theta'({\bf 1}).
\end{align}
On the other hand, we have an invariant map
\begin{align*}
{\bf I}(0) \times \mathcal{W}(\Pi_\infty,\psi_E) &\longrightarrow \C\\
(f,W)&\longmapsto Z(f,W)=\int_{\R^\times{\bf U}_0(\R)\backslash {\bf G}(\R)}f(\eta g)W(a(-1)g)dg.
\end{align*}
By \cite[Lemma 2.1]{Ikeda1992} and \cite[Proposition 3.3]{PSR1987}, the integrals are absolutely convergent and the defining invariant map is non-zero. Let $\mathcal{I}'_\infty$ the trilinear form on $\Pi_\infty^D \boxtimes \Pi_\infty^D$ defined by the integration on $\R^\times \backslash D^\times(\R)$ of matrix coefficients of $\Pi_\infty^D$ as in (\ref{E:trilinear form}). By \cite[Proposition 5.1]{Ichino2008}, the trilinear form $\mathcal{I}_\infty$ (resp.\,$\mathcal{I}_\infty'$) is non-zero if and only if the intertwining map 
\begin{align*}
S(V^3(\R))\times \mathcal{W}(\Pi_\infty,\psi_E) &\longrightarrow \C, \quad (\varphi,W)\longmapsto Z(f_\varphi,W)\\
(resp.\,S(V'^3(\R))\times \mathcal{W}(\Pi_\infty,\psi_E) &\longrightarrow \C, \quad (\varphi,W)\longmapsto Z(f'_\varphi,W))
\end{align*}
is non-zero. Therefore, we conclude from (\ref{E:C1}) and (\ref{E:dichotomy}) that $\mathcal{I}_\infty$ is non-zero. This completes the proof.
\end{proof}

\appendix

\bibliographystyle{alpha}
\bibliography{ref}		
\end{document}